\documentclass{amsart}

\usepackage[margin=1.4in]{geometry} 
\usepackage{amsmath, amsfonts, amsthm, amstext, amssymb} 
\usepackage[colorlinks=true,linkcolor=blue,citecolor=blue]{hyperref} 
\usepackage{todonotes} 
\usepackage[nameinlink, capitalize,noabbrev]{cleveref} 
\usepackage{tikz-cd, tikz} 
\usepackage[shortlabels]{enumitem} 
\usepackage[mathscr]{euscript} 
\usepackage{ stmaryrd} 
\usepackage{adjustbox}
\usepackage{units}

\usetikzlibrary{shapes.geometric}

\usepackage[shortalphabetic]{amsrefs}

\numberwithin{equation}{section} 
\numberwithin{table}{section} 
\numberwithin{figure}{section}

\theoremstyle{definition}
\newtheorem{definition}[equation]{Definition}

\newtheorem{remark}[equation]{Remark}
\newtheorem{example}[equation]{Example}

\theoremstyle{plain}
\newtheorem{lemma}[equation]{Lemma}
\newtheorem{proposition}[equation]{Proposition}
\newtheorem{prop}[equation]{Proposition}
\newtheorem{theorem}[equation]{Theorem}
\newtheorem{corollary}[equation]{Corollary}

\newcommand{\bbF}{\mathbb{F}} 
\newcommand{\F}{\bbF} 
\newcommand{\bbZ}{\mathbb{Z}} 
\newcommand{\Z}{\bbZ} 

\DeclareMathOperator{\res}{res} 
\DeclareMathOperator{\Tor}{Tor} 

\newcommand{\Mack}{{\textnormal{Mack}}} 

\newcommand{\mytinymatrix}[1]
{\scalebox{0.35}{$\begin{pmatrix} #1 \end{pmatrix}$}}
\newcommand{\mylittlematrix}[1]
{\scalebox{0.5}{$\begin{pmatrix} #1 \end{pmatrix}$}}

\usepackage{xcolor}

\usepackage{hyperref}
\hypersetup{
	hyperfootnotes=false,
	hidelinks,
    colorlinks=true, 
    linktoc=all,     
    linkcolor=blue,  
}

\newcommand{\bfSigma}{\textbf{$\Sigma$}}

\usepackage{amsmath, amsfonts, amssymb, amsthm, mathtools, mathrsfs, tikz, titletoc, adjustbox, cleveref}

\usetikzlibrary{shapes.geometric, arrows.meta, calc, positioning, svg.path}

\definecolor{odddiff}{rgb}{1 0 0}
\definecolor{evendiff}{rgb}{0 0 1}
\definecolor{orddiff}{rgb}{0 0.5 0}

\definecolor{gencolor}{rgb}{1 0 1}
\definecolor{mybackgd}{rgb}{0.2 0.2 0.2}
\newcommand{\diffcolor}{orddiff}
\newcommand{\linecolor}{blue}

\newcommand{\amult}{red}
\newcommand{\umult}{blue}
\newcommand{\xmult}{green}
\newcommand{\ymult}{purple}
\newcommand{\zmult}{pink}

\newif\ifDarkMode

\ifDarkMode
	\pagecolor{mybackgd}
	\newcommand{\foreground}{white}
	\color{\foreground}
	\renewcommand{\diffcolor}{yellow}
	\renewcommand{\linecolor}{cyan}
\fi

\newcommand{\R}{\mathbf{R}}  
\newcommand{\Q}{\mathbb{Q}}
\newcommand{\K}{\mathcal{K}}
\newcommand{\N}{\mathbb{N}}

\newcommand{\Sp}{\text{Sp}}
\newcommand{\Ho}{\text{Ho}}
\newcommand{\Mod}{\text{Mod}}
\newcommand{\sphere}{\mathbb{S}}

\newcommand{\Hom}{\text{Hom}}

\newcommand{\CpMack}[4]{
	\begin{tikzpicture}[scale=1]
	\node (TOP) at (0,1) {$#1$}; 
	\node (BOT) at (0,0) {$#3$};
	
	\draw[bend right=15,->] (TOP) to node[left={0.5ex}] {$#2$} (BOT);
	\draw[bend right=15,->] (BOT) to node[right={0.5ex}] {$#4$} (TOP);
	\end{tikzpicture}
}

\newcommand{\burnside}{
\begin{tikzpicture}[scale=0.9]
\node (AK) at (0,6) {\scriptsize $\Z\{\K,\K/L,\K/D,\K/R,1\}$};
\node (AL) at (-3,3) {\scriptsize $\Z\{L,1\}$};
\node (AD) at (0,3) {\scriptsize $\Z\{D,1\}$};
\node (AR) at (3,3) {\scriptsize $\Z\{R,1\}$};
\node (Ae) at (0,0) {\scriptsize $\Z$};

\draw[bend right=10,->] (AL) to node[left={0.5ex}] {\mylittlematrix{2 & 1}} (Ae);
\draw[bend right=10,->] (AD) to node[left] {\mylittlematrix{2 & 1}} (Ae);
\draw[bend right=10,->] (AR) to node[above={1ex}] {\mylittlematrix{2 & 1}} (Ae);

\draw[bend right=10,->] (Ae) to node[above={0.5ex}] {\mylittlematrix{1 \\ 0}} (AL);
\draw[bend right=10,->] (Ae) to node[right]{\mylittlematrix{1 \\ 0 }} (AR);
\draw[bend right=10,->] (Ae) to node[right]{\mylittlematrix{1 \\ 0 }} (AD);

\draw[bend right=10,->] (AK) to node[above={0.5ex}, rotate=45] {\mytinymatrix{2 & 0 & 1 & 1 & 0 \\ 0 & 2 & 0 & 0 & 1}} (AL);
\draw[bend right=10,->] (AK) to node[above={0.5ex}, rotate=90] {\mytinymatrix{2 & 1 & 0 & 1 & 0 \\ 0 & 0 & 2 & 0 & 1}} (AD);
\draw[bend right=10,->] (AK) to node[below right={0.5ex}, rotate=-45] {\mytinymatrix{2 & 1 & 1 & 0 & 0 \\ 0 & 0 & 0 & 2 & 1}} (AR);

\draw[bend right=10,->] (AL) to node[below={1ex}] {\mytinymatrix{1 & 0 \\ 0 & 1 \\ 0 & 0 \\ 0 & 0 \\ 0 & 0}} (AK);
\draw[bend right=10,->] (AD) to node[right={0.05ex}]{\mytinymatrix{1 & 0 \\ 0 & 0 \\0 & 1 \\ 0 & 0 \\ 0 & 0 }} (AK);
\draw[bend right=10,->] (AR) to node[above right]{\mytinymatrix{1 & 0 \\ 0 & 0 \\ 0 & 0 \\ 0 & 1 \\ 0 & 0 }} (AK);

\end{tikzpicture}
}

\newcommand{\constant}{
\begin{tikzpicture}[scale=0.8]
\node (ZK) at (0,4) {$\Z$};
\node (ZL) at (-2,2) {$\Z$};
\node (ZD) at (0,2) {$\Z$};
\node (ZR) at (2,2) {$\Z$};
\node (Ze) at (0,0) {$\Z$};

\draw[bend right=10,->] (ZL) to node[fill=white, inner sep=0.75pt] {\tiny $1$} (Ze);
\draw[bend right=10,->] (ZD) to node[fill=white, inner sep=0.75pt] {\tiny $1$} (Ze);
\draw[bend right=10,->] (ZR) to node[fill=white, inner sep=0.75pt] {\tiny $1$} (Ze);

\draw[bend right=10,->] (Ze) to node[fill=white, inner sep=0.75pt] {\tiny $2$} (ZL);
\draw[bend right=10,->] (Ze) to node[fill=white, inner sep=0.75pt]{\tiny $2$} (ZD);
\draw[bend right=10,->] (Ze) to node[fill=white, inner sep=0.75pt]{\tiny $2$} (ZR);

\draw[bend right=10,->] (ZK) to node[fill=white, inner sep=0.75pt] {\tiny $1$} (ZL);
\draw[bend right=10,->] (ZK) to node[fill=white, inner sep=0.75pt] {\tiny $1$} (ZD);
\draw[bend right=10,->] (ZK) to node[fill=white, inner sep=0.75pt] {\tiny $1$} (ZR);

\draw[bend right=10,->] (ZL) to node[fill=white, inner sep=0.75pt] {\tiny $2$} (ZK);
\draw[bend right=10,->] (ZD) to node[fill=white, inner sep=0.75pt]{\tiny $2$} (ZK);
\draw[bend right=10,->] (ZR) to node[fill=white, inner sep=0.75pt]{\tiny $2$} (ZK);

\end{tikzpicture}
}

\newcommand{\dualconstant}{
\begin{tikzpicture}[scale=0.8]
\node (ZK) at (0,4) {$\Z$};
\node (ZL) at (-2,2) {$\Z$};
\node (ZD) at (0,2) {$\Z$};
\node (ZR) at (2,2) {$\Z$};
\node (Ze) at (0,0) {$\Z$};

\draw[bend right=10,->] (ZL) to node[fill=white, inner sep=0.75pt] {\tiny $2$} (Ze);
\draw[bend right=10,->] (ZD) to node[fill=white, inner sep=0.75pt]{\tiny $2$} (Ze);
\draw[bend right=10,->] (ZR) to node[fill=white, inner sep=0.75pt]{\tiny $2$} (Ze);
			
\draw[bend right=10,->] (Ze) to node[fill=white, inner sep=0.75pt] {\tiny $1$} (ZL);
\draw[bend right=10,->] (Ze) to node[fill=white, inner sep=0.75pt] {\tiny $1$} (ZD);
\draw[bend right=10,->] (Ze) to node[fill=white, inner sep=0.75pt] {\tiny $1$} (ZR);

\draw[bend right=10,->] (ZK) to node[fill=white, inner sep=0.75pt] {\tiny $2$} (ZL);
\draw[bend right=10,->] (ZK) to node[fill=white, inner sep=0.75pt]{\tiny $2$} (ZD);
\draw[bend right=10,->] (ZK) to node[fill=white, inner sep=0.75pt]{\tiny $2$} (ZR);

\draw[bend right=10,->] (ZL) to node[fill=white, inner sep=0.75pt] {\tiny $1$} (ZK);
\draw[bend right=10,->] (ZD) to node[fill=white, inner sep=0.75pt] {\tiny $1$} (ZK);
\draw[bend right=10,->] (ZR) to node[fill=white, inner sep=0.75pt] {\tiny $1$} (ZK);

\end{tikzpicture}
}

\newcommand{\phiLDRMack}[8]{
\begin{tikzpicture}[scale=0.8]
\node (K) at (0,4) {$#1$};
\node (L) at (-2,2) {$#2$};
\node (D) at (0,2) {$#2$};
\node (R) at (2,2) {$#2$};
\node (e) at (0,0) {$0$};

\draw[bend right=10,->] (L) to (e);
\draw[bend right=10,->] (D) to (e);
\draw[bend right=10,->] (R) to (e);

\draw[bend right=10,->] (e) to (L);
\draw[bend right=10,->] (e) to (R);
\draw[bend right=10,->] (e) to (D);

\draw[bend right=10,->] (K) to node[fill=white, inner sep=0.5pt, rotate=45] {\tiny $#3$} (L);
\draw[bend right=10,->] (K) to node[fill=white, inner sep=0.5pt, rotate=90] {\tiny $#5$} (D);
\draw[bend right=10,->] (K) to node[fill=white, inner sep=0.5pt, rotate=-45] {\tiny $#7$} (R);

\draw[bend right=10,->] (L) to node[fill=white, inner sep=0.5pt, rotate=45] {\tiny $#4$} (K);
\draw[bend right=10,->] (D) to node[fill=white, inner sep=0.5pt, rotate=90]{\tiny $#6$} (K);
\draw[bend right=10,->] (R) to node[fill=white, inner sep=0.5pt, rotate=-45]{\tiny $#8$} (K);

\end{tikzpicture}
}

\newcommand{\infKMack}[1]{
\begin{tikzpicture}[scale=0.8]
\node (K) at (0,4) {$#1$};
\node (L) at (-2,2) {$0$};
\node (D) at (0,2) {$0$};
\node (R) at (2,2) {$0$};
\node (e) at (0,0) {$0$};

\draw[bend right=10,->] (L) to (e);
\draw[bend right=10,->] (D) to (e);
\draw[bend right=10,->] (R) to (e);

\draw[bend right=10,->] (e) to (L);
\draw[bend right=10,->] (e) to (R);
\draw[bend right=10,->] (e) to (D);

\draw[bend right=10,->] (K) to (L);
\draw[bend right=10,->] (K) to (D);
\draw[bend right=10,->] (K) to (R);

\draw[bend right=10,->] (L) to (K);
\draw[bend right=10,->] (D) to (K);
\draw[bend right=10,->] (R) to (K);

\end{tikzpicture}
}

\newcommand{\IndeKZ}{
\begin{tikzpicture}[scale=0.7]
\node (K) at (0,6.75) {\scriptsize $\Z[\K/\K]$};
\node (L) at (-3.75,3) {\scriptsize $\Z[\K/L]$};
\node (D) at (0,3) {\scriptsize $\Z[\K/D]$};
\node (R) at (3.75,3) {\scriptsize $\Z[\K/R]$};
\node (e) at (0,-0.75) {\scriptsize $\Z[\K]$};

\draw[bend right=10,->] (L) to node[fill=white, rotate=-45, inner sep=1pt] {\tiny $(1+l)$} (e);
\draw[bend right=10,->] (D) to node[fill=white, rotate=90, inner sep=1pt] {\tiny $(1+d)$} (e);
\draw[bend right=10,->] (R) to node[fill=white, rotate=45, inner sep=1pt] {\tiny $(1+r)$} (e);

\draw[bend right=10,->] (e) to node[fill=white, inner sep=1pt, rotate=-45] {\tiny $q$} (L);
\draw[bend right=10,->] (e) to node[fill=white, inner sep=1pt] {\tiny $q$} (D);
\draw[bend right=10,->] (e) to node[fill=white, inner sep=1pt, rotate=45] {\tiny $q$} (R);

\draw[bend right=10,->] (K) to node[fill=white, inner sep=1pt] {\tiny $\Delta$} (L);
\draw[bend right=10,->] (K) to node[fill=white, inner sep=1pt] {\tiny $\Delta$} (D);
\draw[bend right=10,->] (K) to node[fill=white, inner sep=1pt] {\tiny $\Delta$} (R);

\draw[bend right=10,->] (L) to node[fill=white, inner sep=1pt] {\tiny $\nabla$} (K);
\draw[bend right=10,->] (D) to node[fill=white, inner sep=1pt]{\tiny $\nabla$} (K);
\draw[bend right=10,->] (R) to node[fill=white, inner sep=1pt]{\tiny $\nabla$} (K);
\end{tikzpicture}
}

\newcommand{\TripleSumIndHKZConst}{
\begin{tikzpicture}[scale=0.7]
\node (K) at (0,6.75) {\scriptsize $\Z^3$};
\node (L) at (-3.75,3) {\scriptsize $\Z[\K/L] \oplus \Z^2$};
\node (D) at (0,3) {\scriptsize $\Z \oplus \Z[\K/D] \oplus \Z$};
\node (R) at (3.75,3) {\scriptsize $\Z^2 \oplus \Z[\K/R]$};
\node (e) at (0,-0.75) {\scriptsize $\Z[\K/L] \oplus \Z[\K/D] \oplus \Z[\K/R]$};

\draw[bend right=10,->] (L) to node[fill=white, rotate=-45, inner sep=1pt] {\tiny $\text{id} \oplus \Delta \oplus \Delta$} (e);
\draw[bend right=10,->] (D) to node[fill=white, rotate=90, inner sep=1pt]  {\tiny $\Delta \oplus \text{id} \oplus \Delta$} (e);
\draw[bend right=10,->] (R) to node[fill=white, rotate=45, inner sep=1pt]  {\tiny $\Delta \oplus \Delta \oplus \text{id}$} (e);

\draw[bend right=10,->] (e) to node[fill=white, inner sep=1pt, rotate=-45] {\tiny $2 \oplus \nabla \oplus \nabla$} (L);
\draw[bend right=10,->] (e) to node[fill=white, inner sep=1pt, rotate=90]  {\tiny $\nabla \oplus 2 \oplus \nabla$} (D);
\draw[bend right=10,->] (e) to node[fill=white, inner sep=1pt, rotate=45]  {\tiny $\nabla \oplus \nabla \oplus 2$} (R);

\draw[bend right=10,->] (K) to node[fill=white, inner sep=1pt, rotate=45]  {\tiny $\Delta \oplus \text{id} \oplus \text{id}$} (L);
\draw[bend right=10,->] (K) to node[fill=white, inner sep=1pt, rotate=90]  {\tiny $\text{id} \oplus \Delta \oplus \text{id}$} (D);
\draw[bend right=10,->] (K) to node[fill=white, inner sep=1pt, rotate=-45] {\tiny $\text{id} \oplus \text{id} \oplus \Delta$} (R);

\draw[bend right=10,->] (L) to node[fill=white, inner sep=1pt, rotate=45]  {\tiny $\nabla \oplus 2 \oplus 2$} (K);
\draw[bend right=10,->] (D) to node[fill=white, inner sep=1pt, rotate=90]  {\tiny $2 \oplus \nabla \oplus 2$} (K);
\draw[bend right=10,->] (R) to node[fill=white, inner sep=1pt, rotate=-45] {\tiny $2 \oplus 2 \oplus \nabla$} (K);
\end{tikzpicture}
}

\newcommand{\TripleSumIndHKZDual}{
\begin{tikzpicture}[scale=0.7]
\node (K) at (0,6.75) {\scriptsize $\Z^3$};
\node (L) at (-3.75,3) {\scriptsize $\Z[\K/L] \oplus \Z^2$};
\node (D) at (0,3) {\scriptsize $\Z \oplus \Z[\K/D] \oplus \Z$};
\node (R) at (3.75,3) {\scriptsize $\Z^2 \oplus \Z[\K/R]$};
\node (e) at (0,-0.75) {\scriptsize $\Z[\K/L] \oplus \Z[\K/D] \oplus \Z[\K/R]$};

\draw[bend right=10,->] (L) to node[fill=white, rotate=-45, inner sep=1pt] {\tiny $2 \oplus \Delta \oplus \Delta$} (e);
\draw[bend right=10,->] (D) to node[fill=white, rotate=90, inner sep=1pt]  {\tiny $\Delta \oplus 2 \oplus \Delta$} (e);
\draw[bend right=10,->] (R) to node[fill=white, rotate=45, inner sep=1pt]  {\tiny $\Delta \oplus \Delta \oplus 2$} (e);

\draw[bend right=10,->] (e) to node[fill=white, inner sep=1pt, rotate=-45] {\tiny $\text{id} \oplus \nabla \oplus \nabla$} (L);
\draw[bend right=10,->] (e) to node[fill=white, inner sep=1pt, rotate=90]  {\tiny $\nabla \oplus \text{id} \oplus \nabla$} (D);
\draw[bend right=10,->] (e) to node[fill=white, inner sep=1pt, rotate=45]  {\tiny $\nabla \oplus \nabla \oplus \text{id}$} (R);

\draw[bend right=10,->] (K) to node[fill=white, inner sep=1pt, rotate=45]  {\tiny $\Delta \oplus 2 \oplus 2$} (L);
\draw[bend right=10,->] (K) to node[fill=white, inner sep=1pt, rotate=90]  {\tiny $2 \oplus \Delta \oplus 2$} (D);
\draw[bend right=10,->] (K) to node[fill=white, inner sep=1pt, rotate=-45] {\tiny $2 \oplus 2 \oplus \Delta$} (R);

\draw[bend right=10,->] (L) to node[fill=white, inner sep=1pt, rotate=45]  {\tiny $\nabla \oplus \text{id} \oplus \text{id}$} (K);
\draw[bend right=10,->] (D) to node[fill=white, inner sep=1pt, rotate=90]  {\tiny $\text{id} \oplus \nabla \oplus \text{id}$} (K);
\draw[bend right=10,->] (R) to node[fill=white, inner sep=1pt, rotate=-45] {\tiny $\text{id} \oplus \text{id} \oplus \nabla$} (K);
\end{tikzpicture}
}

\newcommand{\IndeKF}{
\begin{tikzpicture}[scale=0.7]
\node (K) at (0,6.75) {\scriptsize $\F_2[\K/\K]$};
\node (L) at (-3.75,3) {\scriptsize $\F_2[\K/L]$};
\node (D) at (0,3) {\scriptsize $\F_2[\K/D]$};
\node (R) at (3.75,3) {\scriptsize $\F_2[\K/R]$};
\node (e) at (0,-0.75) {\scriptsize $\F_2[\K]$};

\draw[bend right=10,->] (L) to node[fill=white, rotate=-45, inner sep=1pt] {\tiny $(1+l)$} (e);
\draw[bend right=10,->] (D) to node[fill=white, rotate=90, inner sep=1pt] {\tiny $(1+d)$} (e);
\draw[bend right=10,->] (R) to node[fill=white, rotate=45, inner sep=1pt] {\tiny $(1+r)$} (e);

\draw[bend right=10,->] (e) to node[fill=white, inner sep=1pt, rotate=-45] {\tiny $q$} (L);
\draw[bend right=10,->] (e) to node[fill=white, inner sep=1pt] {\tiny $q$} (D);
\draw[bend right=10,->] (e) to node[fill=white, inner sep=1pt, rotate=45] {\tiny $q$} (R);

\draw[bend right=10,->] (K) to node[fill=white, inner sep=1pt] {\tiny $\Delta$} (L);
\draw[bend right=10,->] (K) to node[fill=white, inner sep=1pt] {\tiny $\Delta$} (D);
\draw[bend right=10,->] (K) to node[fill=white, inner sep=1pt] {\tiny $\Delta$} (R);

\draw[bend right=10,->] (L) to node[fill=white, inner sep=1pt] {\tiny $\nabla$} (K);
\draw[bend right=10,->] (D) to node[fill=white, inner sep=1pt]{\tiny $\nabla$} (K);
\draw[bend right=10,->] (R) to node[fill=white, inner sep=1pt]{\tiny $\nabla$} (K);
\end{tikzpicture}
}

\newcommand{\whitecirc}{
\begin{tikzpicture}
\draw[fill=white] (0,0) circle (0.075cm) {}; 
\end{tikzpicture}
}

\newcommand{\whitesquare}{
\begin{tikzpicture}
\node[draw, fill=white, inner sep=1pt, regular polygon sides=4, text=white] at (0,0) {\large $\ast$};
\end{tikzpicture}
}

\newcommand{\whitesquaredual}{
\begin{tikzpicture}
\node[draw, fill=white!100, inner sep=1pt, regular polygon sides=4] at (0,0) {\large $\ast$};
\end{tikzpicture}
}

\newcommand{\whitepent}{
\begin{tikzpicture}
\node[draw, fill=white, inner sep=0.5pt, regular polygon, regular polygon sides=5, text=white, scale=0.8] at (0,0) {\large $\ast$};
\end{tikzpicture}
}

\newcommand{\whitepentdual}{
\begin{tikzpicture}
\node[draw, fill=white, inner sep=0.5pt, regular polygon, regular polygon sides=5, scale=0.8] at (0,0) {\large $\ast$};
\end{tikzpicture}
}

\newcommand{\whitetrap}{
\begin{tikzpicture}
\node[draw, fill=white, inner sep=0.5pt, trapezium, text=white, scale=0.8] at (0,0) {$\ast$};
\end{tikzpicture}
}

\newcommand{\fillcirc}{
\begin{tikzpicture}
\draw[fill=black] (0,0) circle (0.075cm) {}; 
\end{tikzpicture}
}

\newcommand{\fillsquare}{
\begin{tikzpicture}
\node[draw, fill=black, inner sep=1pt, regular polygon sides=4] at (0,0) {$\ast$};
\end{tikzpicture}
}

\newcommand{\fillsquaredual}{
\begin{tikzpicture}
\node[draw, fill=black, inner sep=1pt, regular polygon sides=4, text=white] at (0,0) {\large $\ast$};
\end{tikzpicture}
}

\newcommand{\fillpent}{
\begin{tikzpicture}
\node[draw, fill=black, inner sep=0.5pt, regular polygon, regular polygon sides=5, scale=0.8] at (0,0) {\large $\ast$};
\end{tikzpicture}
}

\newcommand{\fillpentdual}{
\begin{tikzpicture}
\node[draw, fill=black, inner sep=0.5pt,regular polygon, regular polygon sides=5, text=white, scale=0.8] at (0,0) {\large $\ast$};
\end{tikzpicture}
}

\newcommand{\filltrap}{
\begin{tikzpicture}
\node[draw, fill=black, inner sep=0.5pt, trapezium, scale=0.8] at (0,0) {\large $\ast$};
\end{tikzpicture}
}

\newcommand{\filltrapdual}{
\begin{tikzpicture}
\node[draw, fill=black, inner sep=0.5pt, trapezium, scale=0.8, text=white] at (0,0) {\large $\ast$};
\end{tikzpicture}
}

\newcommand{\phiLDRf}{
\begin{tikzpicture}[scale=0.6]
\draw[fill=white] (0,0) rectangle (0.5,0.15);
\end{tikzpicture}
}

\newcommand{\bardot}{
\begin{tikzpicture}[scale=0.6]
\draw[fill=white] (0,0) rectangle (0.5,0.15);
\draw[fill=black] (0.25,0.3) circle (0.08cm);
\end{tikzpicture}
}

\newcommand{\phiLDRQ}{
\begin{tikzpicture}[scale=0.15]
    \begin{scope}[rotate=90]
    \draw (1,0)
        \foreach \t in {0,72,...,360}{
            -- ({cos(\t)},{sin(\t)})
        };
    \draw[fill] (1,0) -- ({cos(72)},{sin(72)}) -- ({0.5*(1+cos(144))},{0.882}) -- ({0.5*(1+cos(144))},{-0.882}) -- ({cos(72)},{-sin(72)});
    \end{scope}
\end{tikzpicture}
}

\begin{document}

\title{The $RO(\K)$-graded coefficients of $H\underline{A}$}

\begin{abstract}
In $G$-equivariant stable homotopy theory, it is known that the equivariant Eilenberg-Mac Lane spectra representing ordinary equivariant cohomology have nontrivial $RO(G)$-graded homotopy corresponding to the equivariant (co)homology of representation spheres. We will compute the universal case of this ordinary $RO(G)$-graded homotopy in the case of $G=\K$, where $\K$ is the Klein-four group. In particular, we will compute a subring of the $RO(\K)$-graded homotopy of $H\underline{A}$ for $\underline{A}$ the Burnside Mackey functor. 
\end{abstract}

\author{Jesse Keyes}

\maketitle
\tableofcontents

\section{Introduction} \label{Introduction}
	In classical topology, spheres are some of the fundamental objects we would like to understand. While the homotopy groups of spheres are famously hard to compute, their cohomology is completely understood. In the $G$-equivariant setting for $G$ a finite group, new spheres and cohomology theories arise. One can define the representation sphere associated to a real representation $V$ of $G$ by forming the one-point compactification $S^V \coloneqq \hat{V}$ and having $G$ act trivially on the point at infinity. Then, $S^V$ carries a $G$-action and has underlying space given by $S^{dim(V)}$. In the $G$-equivariant setting, ordinary cohomology is extended to \textit{Bredon cohomology}, an equivariant cohomology theory that recovers the ordinary cohomology of the underlying space.  
	
	When the coefficients of a Bredon cohomology theory are some Mackey functor $\underline{M}$, the equivariant analogue of abelian groups in stable homotopy theory, the theory is represented by an equivariant Eilenberg-Mac Lane $G$-spectrum $H\underline{M}$. Completely analogous to the classical setting, the stable homotopy of these $G$-spectra is the Bredon cohomology of a point. Further, we understand how to map trivial spheres (i.e. spheres with a trivial action) into $H\underline{M}$. If we consider maps out of representation spheres, however, this is not completely understood. Our goal will be to understand the Bredon cohomology of a point with coefficients in the Burnside Mackey functor, in the case that $G$ is the Klein four group $\K \cong C_2 \times C_2$.  

	More precisely, we will compute $\underline{\pi}_{n + k\overline{\rho}}^{\K}(H\underline{A})$ for $n, k \in \Z$, where $\underline{A}$ is the Burnside Mackey functor for $\K$. It is worth noting that $\underline{A}$ is $\underline{\pi}^\K_0(\sphere_\K)$, so that $H\underline{A}$ is the $0^{th}$-Postnikov truncation of the $\K$-equivariant sphere. Here, $\overline{\rho} \in RO(\K)$ is the reduced regular representation for $\K$. Equivalently, this is a computation of the $\K$-equivariant ordinary cohomology of a point in a restricted range, and this is the perspective we'll take at several points. The complete $C_2$-equivariant cohomology of a point with coefficients in any Mackey functor can be found in \cite{Sikora}. The complete $C_2$-equivariant cohomology of a point with coefficients in $\underline{A}$, $\underline{\Z}$, and $\underline{\F}_2$ can be found in \cite{Lewis}, \cite{Zeng}, \cite{Holler-Kriz}. In the $\K$-setting, the cohomology of a point with coefficients in $\underline{\Z}$ was computed in the same restricted range in \cite{Slone}. For $\underline{\F}_2$-coefficients along with some other more general coefficients, see \cite{Guillou-Yarnall}. The full $\K$-equivariant computation for $\underline{\F}_2$-coefficients is due to \cite{Holler-Kriz} and \cite{ellisbloor}. 
	
	In a small range, we'll attack this by computing Bredon cohomology of $S^{k\overline{\rho}}$ using a cellular approach and following up with a series of small spectral sequences. In particular, we will use the spectral sequence associated to a double complex to inductively compute the homology of $S^{(k-1)\overline{\rho}} \wedge S^{\overline{\rho}} \simeq S^{k\overline{\rho}}$. While one may expect the use of a Kunneth spectral sequence to compute this homology, this is not what we will do here. The relationship between these two methods is explored in \cref{KunnethSection}. In the remaining range, we'll use a comparison to $H\underline{\Z}$ to determine the answer for $H\underline{A}$. 
	
	As we will see, the kernel of the map $\underline{A} \rightarrow \underline{\Z}$ vanishes at $(\K/e)$, the underlying level. Broadly speaking, this should tell us to expect that cohomology of $S^{k\overline{\rho}}$ with coefficients in this kernel should vanish in higher degrees as the cell decomposition of $S^{k\overline{\rho}}$ in higher degrees is exclusively a decomposition into free cells. We conclude that the homotopy of $k\overline{\rho}$-suspensions of $H\underline{A}$ and $H\underline{\Z}$ should agree in high degrees, and the homotopy of $k\overline{\rho}$-suspensions of $H\underline{\Z}$ was computed in \cite{Slone}. To fully understand the homotopy of $k\overline{\rho}$-suspensions of $H\underline{A}$, it will then suffice to compute it in low degrees. This is precisely the approach we take. In \cref{Mackey Functors} and \cref{Preliminaries}, we collect some preliminary notions and computations for the main computation. \cref{ComparisonPositive} and \cref{ComparisonNegative} contain the calculation of the cohomology of $S^{k\overline{\rho}}$ with coefficients in the aforementioned kernel. \cref{HAPositiveCone} and \cref{HANegativeCone} finish off the computation with the cohomology in lower degrees. \cref{RmksOnMultStructure} contains some remarks on the multiplicative structure of the cohomology and, in particular, how it isn't quite as straightforward to understand some multiplicative phenomena as it is in the $C_2$-equivariant setting. The difficulty is essentially captured in the fact that $a_{\overline{\rho}} \in \underline{\pi}_0(\Sigma^{\overline{\rho}} H\underline{A})(\K/\K)$, the Euler class associated to $S^{\overline{\rho}}$, is not an irreducible element. This is in contrast to the $C_2$-equivariant setting. \\
	
\textbf{Acknowledgments} I would like to thank Bert Guillou for his endless support and guidance during the work done on this project. I am deeply grateful for the sheer amount of time spent on reading drafts of this document and thoughtfulness in guiding me through the process of producing such a thing.
 
\subsection{Notations and Conventions}
	Recall the Klein four group $\K \cong C_2 \times C_2$. 
	
	\begin{itemize}
		\item $l$ and $r$ denote generators of the left and right factors, respectively, and $d \coloneqq lr$. 
		\item $L \coloneqq \langle l \rangle$, $D \coloneqq \langle d \rangle$, and $R \coloneqq \langle r \rangle$. 
		\item $i_J^H$ denotes the inclusion of $J$ into $H$ for any $J \leq H \leq \K$. When restricting group actions along this map, we will denote $(i_J^H)^*$ as $res_J^H$. The same notation will be used for the map induced on Burnside rings.
		\item $\tau_H^G: Set_H^{fin} \rightarrow Set_G^{fin}$ is the ``induction" of $H$-sets given by $H/J \mapsto G \times_H H/J$, where $G$ acts on the left. The same notation will be used for the map induced on Burnside rings.
		\item $\varphi_H: \K \rightarrow C_2$ denotes the canonical surjection with kernel $H \leq \K$. 
		\item $q: \K/J \rightarrow \K/H$ will denote the quotient map for $J \leq H$. 
		\item $\bigoplus\limits_H (-)$ will range over $H \in \{L,D,R\}$, unless otherwise mentioned. 
		\item In the Burnside ring $A(G)$, $1 \coloneqq G/G$. 
	\end{itemize}  

\section{Mackey Functors} \label{Mackey Functors}
	Mackey functors are an algebraic gadget that appear, for example, when working in equivariant stable homotopy theory. Of course, \cite{PurpleBook} is a wonderful account of the material collected in this section. While we are mostly concerned with $\K$-equivariant homotopy theory, we will lean on knowledge of $C_2$-equivariant homotopy theory quite often. This is precisely because there are three embeddings of $C_2$ into $\K$ and each embedding gives a quotient of $C_2$. So, we will begin with $C_2$-Mackey functors. From a computational perspective, the following definition will suffice: 
	\begin{definition}
		The data of a $C_2$-Mackey functor $\underline{M}$ is as follows:
		\begin{itemize}
			\item abelian groups $\underline{M}(C_2/e)$ and $\underline{M}(C_2/C_2)$
			\item a restriction map $res_e^{C_2}: \underline{M}(C_2/C_2) \rightarrow \underline{M}(C_2/e)$
			\item a transfer map $tr_e^{C_2}: \underline{M}(C_2/e) \rightarrow \underline{M}(C_2/C_2)$
			\item a $W_{C_2}(e) \cong C_2$-action on $\underline{M}(C_2/e)$.
		\end{itemize}
	To complete the definition, we require that the ``double coset formula" is satisfied. The formula describes how to restrict an element in the image of any transfer, but it won't appear explicitly throughout this computation, so it is omitted here. 
	\end{definition}	
	
	The data of a Mackey functor will be depicted in the so-called ``Lewis diagram", as introduced in \cite{Lewis} and seen here in \cref{C2MackeyFunctor}.
\begin{figure} 
\caption{The Lewis Diagram for a $C_2$-Mackey Functor}
\label{C2MackeyFunctor}
\begin{tikzpicture}[scale=1]
\node (CC) at (0,2) {\scriptsize $\underline{M}(C_2/C_2)$};
\node (Ce) at (0,0) {\scriptsize $\underline{M}(C_2/e)$};

\draw[bend right=15,->] (CC) to node[left={0.5ex}] {$res_e^{C_2}$} (Ce);
\draw[bend right=15,->] (Ce) to node[right={0.5ex}] {$tr_e^{C_2}$} (CC);
\draw[looseness=2, loop below, ->] (Ce) to node[right={0.05ex}]{\tiny $W_{C_2}(e)$} (Ce);
\end{tikzpicture}
\end{figure}
	As $C_2$-Mackey functors come up, we will not depict the Weyl actions in the Lewis diagram. Rather, we will give descriptions of $\underline{M}(C_2/H)$ as $W_{C_2}(H)$-modules. Examples of $C_2$-Mackey functors can be found in \cref{C2Zoo}.
	
	\begin{table}
	\caption{Some $C_2$-Mackey Functors} 
	\label{C2Zoo}
	\begin{tabular}{|| c | c | c | c ||}
	\hline &&&\\[-1em]
	$\underline{\Z} = \whitesquare$ & $\underline{\Z}^* = \whitesquaredual$ & $\underline{\F_2} = \fillsquare$ & $\underline{\F_2}^* = \fillsquaredual$ \\
	\hline &&&\\[-1em]
	$\CpMack{\Z}{1}{\Z}{2}$ & $\CpMack{\Z}{2}{\Z}{1}$ & $\CpMack{\F_2}{1}{\F_2}{0}$ & $\CpMack{\F_2}{0}{\F_2}{1}$ \\
	\hline &&&\\[-1em]
	$\langle \Z \rangle = \whitecirc$ & $\underline{f}$ & $\langle \F_2 \rangle = \fillcirc$ & $\underline{Q}$ \\
	\hline &&&\\[-1em]
	$\CpMack{\Z}{}{0}{}$ & $\CpMack{0}{}{\Z_-}{}$ & $\CpMack{\F_2}{}{0}{}$ & $\CpMack{\F_2}{0}{\Z_-}{q}$ \\
	\hline 
	\end{tabular}
	\end{table}
	
\begin{definition}
	A map $f: \underline{M} \rightarrow \underline{N}$ of $C_2$-Mackey functors is a collection of $W_{C_2}(H)$-equivariant maps $f_H: \underline{M}(C_2/H) \rightarrow \underline{N}(C_2/H)$ for $H \leq C_2$. Further, the maps must commute independently with restrictions and transfers. 
\end{definition}
	We get a category $\Mack_{C_2}$ of $C_2$-Mackey functors. In fact, $\Mack_{C_2}$ is an abelian category with a closed symmetric monoidal structure. We'll have the same structure on the category of $\K$-Mackey functors, so a more detailed description of it will be given in that case. We have the following definition of a $\K$-Mackey functor: 
	\begin{definition} \label{KMackeyFunctor}
		A Mackey functor $\underline{M}$ for the Klein-four group $\K$ is the following data: 
	\begin{itemize}
		\item abelian groups $\underline{M}(\K/H)$ for $H \leq \K$
		\item restriction maps $res_J^H: \underline{M}(\K/H) \rightarrow \underline{M}(\K/J)$ whenever $J \leq H \leq \K$ 
		\item transfer maps $tr_J^H: \underline{M}(\K/J) \rightarrow \underline{M}(\K/H)$ whenever $J \leq H \leq \K$ 
		\item group actions of $W_\K(H)$ on $\underline{M}(\K/H)$ whenever $H \leq K$. 
	\end{itemize} 	 
	\end{definition}
The data of a $\K$-Mackey functor will also be depicted with a Lewis diagram as in \cref{KMackeyFunctorDiagram}. 
\begin{figure}
\caption{The Lewis Diagram for a $\K$-Mackey Functor}
\label{KMackeyFunctorDiagram}
\begin{tikzpicture}[scale=1]
\node (AK) at (0,6) {\scriptsize $\underline{M}(\K/\K)$};
\node (AL) at (-4,3) {\scriptsize $\underline{M}(\K/L)$};
\node (AD) at (0,3) {\scriptsize $\underline{M}(\K/D)$};
\node (AR) at (4,3) {\scriptsize $\underline{M}(\K/R)$};
\node (Ae) at (0,0) {\scriptsize $\underline{M}(\K/e)$};

				
\draw[bend right=10,->] (AL) to node[below={0.5ex}, rotate=-45] {$res_e^L$} (Ae);
\draw[bend right=10,->] (AD) to node[left] {$res_e^D$} (Ae);
\draw[bend right=10,->] (AR) to node[above={0.5ex}, rotate=45] {$res_e^R$} (Ae);

				
\draw[bend right=10,->] (Ae) to node[above={0.5ex}, rotate=-45] {$tr_e^L$} (AL);
\draw[bend right=10,->] (Ae) to node[right]{$tr_e^D$} (AD);
\draw[bend right=10,->] (Ae) to node[below={0.5ex}, rotate=45]{$tr_e^R$} (AR);

				
\draw[bend right=10,->] (AK) to node[above={0.5ex}, rotate=45] {$res_L^\K$} (AL);
\draw[bend right=10,->] (AK) to node[left] {$res_D^\K$} (AD);
\draw[bend right=10,->] (AK) to node[below={0.5ex}, rotate=-45] {$res_R^\K$} (AR);

				
\draw[bend right=10,->] (AL) to node[below={1ex}, rotate=45] {$tr_L^\K$} (AK);
\draw[bend right=10,->] (AD) to node[right]{$tr_D^\K$} (AK);
\draw[bend right=10,->] (AR) to node[above, rotate=-45]{$tr_R^\K$} (AK);

				
\draw[looseness=2, loop left, ->] (AL) to node[left={0.05ex}]{\tiny $W_\K(L)$} (AL); 
\draw[looseness=2, loop left, ->] (AD) to node[left={0.05ex}]{\tiny $W_\K(D)$} (AD); 
\draw[looseness=2, loop right, ->] (AR) to node[right={0.05ex}]{\tiny $W_\K(R)$} (AR); 
\draw[looseness=2, loop right, ->] (Ae) to node[right={0.05ex}]{\tiny $W_\K(e)$} (Ae);
\end{tikzpicture}
\end{figure}
  We will treat the Weyl actions as in the $C_2$-case, in that we will depict $\underline{M}(\K/H)$ as a $W_\K(H)$-module rather than depict the actions explicitly in the Lewis diagram. Note that $W_\K(H)$ is isomorphic to $C_2$ for $H \in \{L,D,R\}$. Examples of $\K$-Mackey functors can be found in \cref{KZoo}, where $\oplus_H \uparrow_H^\K \langle \Z \rangle$ and $\oplus_H \uparrow_H^\K \langle \F_2 \rangle$ both have a $\K/H$-action of swapping the generators at level $\K/H$. 
  
  \begin{table} 
	\caption{Some $\K$ Mackey Functors}
	\label{KZoo}
	\begin{tabular}{|| c | c | c ||}
	\hline &&\\[-1em]
	$\underline{\Z} = \whitesquare$ & $\phi_{LDR}^* \underline{\F_2} = \fillpent$ & $\underline{mg} = \filltrap$ \\
	\hline &&\\[-1em]
	$\constant$ & $\phiLDRMack{\F_2^3}{\F_2}{p_1}{0}{p_2}{0}{p_3}{0}$ & $\phiLDRMack{\F_2^2}{\F_2}{p_1}{0}{\nabla}{0}{p_2}{0}$ \\
	\hline &&\\[-1em]
	$\underline{\Z}^* = \whitesquaredual$ & $\phi_{LDR}^* \underline{\F_2}^* = \fillpentdual$ &  $\underline{mg}^* = \filltrapdual$ \\
	\hline &&\\[-1em]
	$\dualconstant$  & $\phiLDRMack{\F_2^3}{\F_2}{0}{i_1}{0}{i_2}{0}{i_3}$ & $\phiLDRMack{\F_2^2}{\F_2}{0}{i_1}{0}{\Delta}{0}{i_2}$ \\
	\hline &&\\[-1em]
	$\phi_{LDR}^* \underline{\Z} = \whitepent$ & $\phi_{LDR}^* \underline{f} =$ \tiny $\phiLDRf$ & $\langle \Z \rangle = \whitecirc$  \\
	\hline &&\\[-1em]
	$\phiLDRMack{\Z^3}{\Z}{p_1}{2i_1}{p_2}{2i_2}{p_3}{2i_3}$ & $\phiLDRMack{0}{\Z_-}{0}{0}{0}{0}{0}{0}$ & $\infKMack{\Z}$ \\
	\hline &&\\[-1em]
	$\phi_{LDR}^* \underline{\Z}^* = \whitepentdual$ & $\phi_{LDR}^* \underline{Q} =$ \tiny $\phiLDRQ$ & $\langle \F_2 \rangle = \underline{g} = \fillcirc$ \\
	\hline &&\\[-1em]
	$\phiLDRMack{\Z^3}{\Z}{2p_1}{i_1}{2p_2}{i_2}{2p_3}{i_3}$ & $\phiLDRMack{\F_2^3}{\Z_-}{0}{i_1 \circ q}{0}{i_2 \circ q}{0}{i_3 \circ q}$ &  $\infKMack{\F_2}$ \\
	\hline &\\[-1em]
	$\bigoplus\limits_H \uparrow_H^\K \langle \Z \rangle = \widehat{\whitecirc}$ & $\bigoplus\limits_H \uparrow_H^\K \langle \F_2 \rangle = \widehat{\fillcirc}$ & \\
	\hline &\\[-1em]
	$\phiLDRMack{\Z^3}{\Z^2}{\Delta p_1}{i_1 \nabla}{\Delta p_2}{i_2\nabla}{\Delta p_3}{i_3\nabla}$ & $\phiLDRMack{\F_2^3}{\F_2^2}{\Delta p_1}{i_1 \nabla}{\Delta p_2}{i_2 \nabla}{\Delta p_3}{i_3 \nabla}$  & \\
	\hline
	\end{tabular} 
\end{table}
  
  	The filled in symbols are meant to indicate 2-torsion levelwise while the unfilled symbols indicate integrality levelwise. The (filled) pentagon was chosen because it has 3 vertices near the top, and the Mackey functor it represents has a rank three free abelian group (three-dimensional $\F_2$-vector space) at level $\K/\K$. Similarly, the trapezoid has two vertices at the top while the Mackey functor it represents has a two-dimensional $\F_2$-vector space at level $\K/\K$. Many of these notational choices are made in \cite{Guillou-Yarnall}. In the following sections, there are several $\K$-Mackey functors that show up in spectral sequence pages but not explicitly in the homotopy of the $\K$-spectra we are interested in. Those are collected in \cref{ExtraMackeyFunctors}, where we have the convention that $\oplus_H$ always ranges over $H \in \{L,D,R\}$. 
	
\begin{table} 
	\caption{Extra $\K$-Mackey Functors}
	\label{ExtraMackeyFunctors}
	\begin{tabular}{|| c | c ||}
	\hline &\\[-1em]
	$\bigoplus\limits_H \uparrow_H^\K \underline{\Z} = \widehat{\whitesquare}$ & $\overline{\whitesquare} = \uparrow_e^\K \Z$ \\
	\hline &\\[-1em]
	$\TripleSumIndHKZConst$ & $\IndeKZ$ \\
	\hline &\\[-1em]
	$\bigoplus\limits_H \uparrow_H^\K \underline{\Z}^* = \widehat{\whitesquaredual}$ & $\overline{\fillsquare} = \uparrow_e^\K \F_2$ \\
	\hline &\\[-1em]
	$\TripleSumIndHKZDual$ & $\IndeKF$ \\
	\hline &\\[-1em]
	$\bardot = \underline{E}$ & \\
	\hline &\\[-1em]
	$\phiLDRMack{\F_2}{\Z_-}{0}{1}{0}{1}{0}{1}$ & \\
	\hline
\end{tabular} 
\end{table}

  \begin{definition}
  	A map of $\K$-Mackey functors $f: \underline{M} \rightarrow \underline{N}$ is a $W_\K(H)$-equivariant map $f_H: \underline{M}(\K/H) \rightarrow \underline{N}(\K/H)$ for $H \leq \K$ such that $f$ commutes with restrictions and transfers independently. We will depict a map of Mackey functors via the their Lewis diagrams. For an example, see \cref{sigma_Hhomology} and the convention concerning maps.
  \end{definition}
  
  We get a category of $\K$-Mackey functors which we will denote $\Mack_\K$. One can prove that $\Mack_\K$ is an abelian category where the zero object $\underline{0}$ is given by the trivial group at every level. Further, finite coproducts are given by taking the direct sum across levels, restrictions, and transfers. $\Mack_\K$ also has a closed symmetric monoidal structure under the ``box product", denoted $\boxtimes$. The definition of the box product won't be important; we'll only compute box products in cases where a nice description can easily be given. However, it's worth noting that $\underline{A}$, the Burnside Mackey functor and subject of this discussion, is the unit for the box product. This makes the computation of $\underline{\pi}_\star^\K(H\underline{A})$ the universal case for $\K$-equivariant Bredon cohomology. More precisely, $\underline{\pi}_\star^\K(H\underline{M})$ is a module over $\underline{\pi}_\star^\K(H\underline{A})$ for any $\underline{M} \in \Mack_\K$.  
	
\begin{remark} \label{adjunctions}
		Let $J \leq H \leq \K$. Recall the restriction-induction adjunction 
	\begin{center}
			\begin{tikzpicture}
 			\node (00) at (0,0) {$\uparrow_H^\K: \Mack_H$};
 			\node (40) at (4,0) {$\Mack_\K :\downarrow_H^\K$};
 			\node at (2,0) {$\perp$};
 		
 			\draw[bend left=10, ->] (00) to (40);
 			\draw[bend left=10, ->] (40) to (00);
 			\end{tikzpicture}.
	\end{center}
	Restriction is a strong symmetric monoidal functor given on objects by 
	\[ \downarrow_H^\K \underline{M}(H/J) \cong \underline{M}(\tau_H^K(H/J)) \cong \underline{M}(\K/J). \] 
	Induction is given on objects by 
	\[ \uparrow_H^\K \underline{N}(\K/H') \cong \underline{N}(res_H^K(\K/H')). \] 
	Note that $\res_J^H$ and $\tau_J^H$ are the restriction and transfer of $G$-sets, respectively. 
\begin{definition}
Let $\underline{M}, \underline{N} \in \Mack_\K$. The internal hom in $\Mack_\K$ is defined as 
\[ \underline{\Hom}(\underline{M},\underline{N})(\K/H) \coloneqq \Hom(\uparrow_H^\K \downarrow_H^\K \underline{M}, \underline{N}) \cong \Hom(\downarrow_H^\K \underline{M}, \downarrow_H^\K \underline{N}). \]
\end{definition}	
We have a nice description of maps out of a Mackey functor of the form $\uparrow_H^\K \downarrow_H^\K \underline{A}$.   
\end{remark}

\begin{lemma}
\label{MapOutOfFreeFormula}
 Let $H \leq \K$. For any $\underline{M} \in \Mack_\K$, there is a natural isomorphism
 \[ \underline{\Hom}(\uparrow_H^\K \downarrow_H^\K \underline{A}, \underline{M})(\K/\K) \cong \underline{M}(\K/H). \]
\end{lemma}

\begin{proof} 
We need only observe that $\underline{A}$ is the unit for the box product in $\Mack_G$ for any finite $G$. Then,
\begin{eqnarray*}
\underline{\Hom}(\uparrow_H^\K \downarrow_H^\K \underline{A}, \underline{M})(\K/\K) 
&\cong& \Hom(\downarrow_H^\K \underline{A}, \downarrow_H^\K \underline{M}), \\
&\cong& \underline{\Hom}(\underline{A}, \downarrow_H^\K \underline{M})(H/H), \\
&\cong& \underline{M}(\K/H).
\end{eqnarray*}
\end{proof}

\begin{remark}
	There is an adjunction
		\begin{center}
			\begin{tikzpicture}
 			\node (00) at (0,0) {$\phi_H^*: \Mack_{\K/H}$};
 			\node (40) at (4,0) {$\Mack_\K :q^*$};
 			\node at (2,0) {$\perp$};
 		
 			\draw[bend left=10, ->] (00) to (40);
 			\draw[bend left=10, ->] (40) to (00);
 			\end{tikzpicture}.
		\end{center}
	We define $\phi_{LDR}^* \underline{M} \coloneqq \oplus_H \phi_H^* \underline{M}$.	Details of this adjunction can be found in \cite[pg.~7]{ThevenazWebb}, where $\phi_H^*$ goes by the name of $Inf_H^\K$. The name ``inflation" can be motivated by considering how $\phi_H$ looks on Lewis diagrams. The case of $H=L$ is depicted in \cref{InflationExample}, and it's worth noting that we have prescribed an appropriate Weyl action at $(\phi^*_L \underline{M})(\K/L)$.  
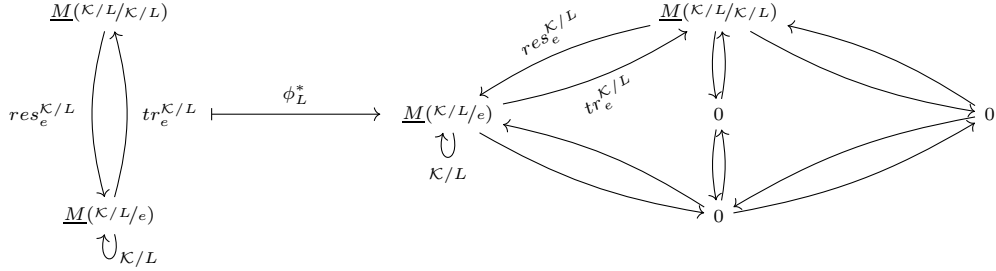
\begin{figure} 
\caption{Inflation Depicted via Lewis Diagrams}
\label{InflationExample}
\begin{tikzpicture}[scale=0.9]
\node (CC) at (-4,4.5) {\scriptsize $\underline{M}(\nicefrac{\K/L}{\K/L})$};
\node (Ce) at (-4,1.5) {\scriptsize $\underline{M}(\nicefrac{\K/L}{e})$};

\draw[bend right=15,->] (CC) to node[left={0.5ex}] {\scriptsize $res_e^{\K/L}$} (Ce);
\draw[bend right=15,->] (Ce) to node[right={0.5ex}] {\scriptsize $tr_e^{\K/L}$} (CC);
\draw[looseness=2, loop below, ->] (Ce) to node[right={0.05ex}]{\tiny $\K/L$} (Ce);

\node (AK) at (5,4.5) {\scriptsize $\underline{M}(\nicefrac{\K/L}{\K/L})$};
\node (AL) at (1,3) {\scriptsize $\underline{M}(\nicefrac{\K/L}{e})$};
\node (AD) at (5,3) {\scriptsize $0$};
\node (AR) at (9,3) {\scriptsize $0$};
\node (Ae) at (5,1.5) {\scriptsize $0$};

				
\draw[bend right=10,->] (AL) to (Ae);
\draw[bend right=10,->] (AD) to (Ae);
\draw[bend right=10,->] (AR) to (Ae);

				
\draw[bend right=10,->] (Ae) to (AL);
\draw[bend right=10,->] (Ae) to (AD);
\draw[bend right=10,->] (Ae) to (AR);

				
\draw[bend right=10,->] (AK) to node[above={0.05ex}, rotate=30] {\scriptsize $res_e^{\K/L}$} (AL);
\draw[bend right=10,->] (AK) to (AD);
\draw[bend right=10,->] (AK) to (AR);

				
\draw[bend right=10,->] (AL) to node[below={0.05ex}, rotate=30] {\scriptsize $tr_e^{\K/L}$} (AK);
\draw[bend right=10,->] (AD) to (AK);
\draw[bend right=10,->] (AR) to (AK);

				
\draw[looseness=2, loop below, ->] (AL) to node[below={0.05ex}]{\tiny $\K/L$} (AL); 

\draw[|->] (-2.5,3) to node[above={0.05ex}] {\scriptsize $\phi_L^*$} (0,3);
\end{tikzpicture}
\end{figure}
There are analogous descriptions of $\phi_D$ and $\phi_R$. 
\end{remark}

\begin{definition}
Let $\underline{A}_{\K/H}$ denote $\uparrow_H^\K \downarrow_H^\K \underline{A}$. 
\end{definition}
We saw in \cref{MapOutOfFreeFormula} that mapping out of these Mackey functors has a nice description. When taking a box product with these Mackey functors, we also have a nice description. 
\begin{lemma}\label{boxformula}
	Let $\underline{M} \in \Mack(\K)$. Then, $\underline{A}_{\K/H} \boxtimes \underline{M} \cong \uparrow_H^\K \downarrow_H^\K \underline{M}$. 
\end{lemma}
	
\begin{proof}
	The Mackey functors in question give the same corepresentable functors: 
	\begin{eqnarray*}
	\Hom_{\Mack_\K}(\underline{A}_{\K/H} \boxtimes \underline{M},\underline{N}) &\cong &\Hom_{\Mack(\K)}(\underline{A}_{\K/H}, \underline{\Hom}(\underline{M}, \underline{N})), \\
&\cong & \Hom(\underline{M}, \underline{N})(\K/H), \\
&\cong & \Hom_{\Mack_H}(\downarrow_H^\K \underline{M}, \downarrow_H^\K \underline{N}), \\
&\cong & \Hom_{\Mack_\K}(\uparrow_H^\K \downarrow_H^\K \underline{M}, \underline{N}).
	\end{eqnarray*}
\end{proof}

\section{Preliminaries} \label{Preliminaries}
	Throughout, we will be working in $Ho\Sp_\K$, the $\K$-equivariant stable homotopy category whose objects are ``genuine" $\K$-spectra. There are at least two notable, computational differences between classical (nonequivariant) spectra and genuine $\K$-spectra. First, the equivariant suspension spectra of representation spheres become invertible in $Ho\Sp_\K$ and it is common to grade homotopical structures over these. Second, given $X \in Ho\Sp_\K$ we can arrange the homotopy of $X$ into an $RO(\K)$-graded Mackey functor for the group $\K$ which we will denote $\underline{\pi}^\K_\star(X)$. Throughout this paper, $RO(G)$ denotes the Grothendieck completion of the semiring of isomorphism classes of finite dimensional real, orthogonal representations of $G$ with direct sum and tensor product. 
	
	\begin{definition} Let $V$ and $V'$ be finite-dimensional real, orthogonal representations of $\K$ and $W \in RO(\K)$.
	\begin{enumerate}[1.] 
		\item As mentioned, the suspension $\K$-spectrum $\Sigma_\K^\infty S^V$ is invertible in $Ho(Sp^\K)$. Let $\Sigma^V$ denote $\Sigma_\K^\infty S^V \wedge (-)$ and $\Sigma^{-V}$ denote its inverse. Then, we can define 
		\[ \Sigma^{V-V'} (-) \coloneqq \Sigma^V \Sigma^{-V'} (-). \]
This extends to suspension functors for any virtual representation.  
		\item Let $X \in Ho\Sp_\K$. The $W^{th}$-homotopy Mackey functor $\underline{\pi}_W^\K(X) \in \Mack_\K$ is given levelwise by 
		\[ \underline{\pi}^\K_W(X)(\K/H) \cong [\Sigma^W\Sigma_\K^\infty \K/H_+, X]. \]
	\end{enumerate}
	The restriction maps $\underline{\pi}_W^\K(X)(\K/H)\rightarrow \underline{\pi}_W^\K(X)(\K/J)$ for $J \leq H$ are defined via pulling back along the maps $\Sigma^\infty \K/J_+ \rightarrow \Sigma^\infty \K/H_+$. The conjugation maps are defined similarly, but the transfers are really a feature of genuine $G$-spectra. The transfers can be defined using self-duality of the stable orbits $\Sigma^\infty_\K \K/H_+$. 
	\end{definition}
	
	The irreducible representations of $\K$ are all 1-dimensional, and arise as pullbacks of irreducible $C_2$-representations along the maps $\varphi_H$ mentioned above. More precisely, we have an isomorphism of abelian groups
	\[ RO(\K) \cong \Z\{1, \sigma_L, \sigma_D, \sigma_R\}, \]
where $\sigma$ is the irreducible sign representation of $C_2$ and $\sigma_H \coloneqq \varphi_H^* \sigma$. 
	A topological analogue of \cref{boxformula} which will prove extremely useful is the following: 

\begin{lemma}\label{topologicalboxformula}
	Let $X \in Ho\Sp_\K$. Then, 
	\[ \underline{\pi}_V^\K(\Sigma^\infty_\K \K/H_+ \wedge X) \cong \underline{A}_{\K/H} \boxtimes \underline{\pi}_V^\K(X). \] 
\end{lemma}

\begin{proof}
	Chasing a few adjunctions and definitions yields the following: 
	\begin{eqnarray*}
	\underline{\pi}_{V}^\K(\Sigma^\infty_\K \K/H_+ \wedge X)(\K/J) &\cong & [\Sigma^{V} \Sigma_\K^\infty \K/J_+, \, \Sigma_\K^\infty \K/H_+ \wedge X] \\
	&\cong & [\Sigma^V \Sigma_\K^\infty \K/J_+ \wedge \Sigma_\K^\infty \K/H_+, \, X] \\
	&\cong & [\Sigma^V \Sigma_\K^\infty (\K/J \times \K/H)_+ \, X] \\
	&\cong & [\Sigma^V \Sigma_\K^\infty (\tau_H^\K res_H^\K \K/J)_+), \, X] \\
	&\cong & \uparrow_H^\K \downarrow_H^\K \underline{\pi}_V^\K(X)(\K/J). 
	\end{eqnarray*} 
\end{proof}

\begin{corollary}
	We have an isomorphism 
	\[ \widetilde{\underline{H}}_0^\K(\K/H_+,\underline{A}) \cong \underline{\pi}_0^\K(\K/H_+ \wedge H\underline{A}) \cong \underline{A}_{\K/H} \boxtimes \underline{\pi}_0^\K(H\underline{A}) \cong \underline{A}_{\K/H}. \]
\end{corollary}

Lastly, there are two cofiber sequences that will be used as foundations for many of the computations that occur throughout. Everything here is to be interpreted stably. The first is 
\[ \K/H_+ \rightarrow S^0 \rightarrow S^{\sigma_H}. \]
Second, we can take the dual cofiber sequence 
\[ S^{-\sigma_H} \rightarrow S^0 \rightarrow \K/H_+, \]
where we make use of self-duality of stable orbits and $S^{-\sigma_H}$ denotes the inverse of $S^{\sigma_H}$.   
	 
\subsection{A Family of Free Mackey Functors}
	The Burnside Mackey functor $\underline{A}$ for the group $\K$ is given levelwise by the abelian group underlying the Burnside ring, so we'll recall the notion of these rings first. 
	
	\begin{definition}
		The Burnside ring for the group $G$, denoted $A(G)$, is the Grothendieck completion of the semiring of isomorphism classes of finite $G$-sets under disjoint union and cartesian product.  
	\end{definition}
	 
	 \begin{example}
	  The Burnside ring $A(C_2)$ has underlying abelian group $\Z\{C_2/e, C_2/C_2\}$. The unit for the multiplication is $C_2/C_2$ and $(C_2/e)^2 = 2(C_2/e)$. We get a ring isomorphism 
	  \[ A(C_2) \cong \Z[x]/(x^2-2x). \]
	 \end{example} 
The Burnside ring for $\K$ admits a similar ring structure as a quotient of a polynomial ring over $\Z$, but the precise description is unnecessary for the computations here. 

\begin{definition}
	The Burnside Mackey functor $\underline{A}$ for $\K$ has the following data: 
	\begin{itemize}
	\item $\underline{A}(\K/H) \coloneqq A(H)$.
	\item restriction maps $A(H)$ to $A(J)$ given by $res_J^H$ of $G$-sets
	\item transfer maps $A(J)$ to $A(H)$ given by $\tau_J^H$ of $G$-sets 
	\item trivial Weyl actions.
	\end{itemize}
	We're overloading the notation $res_J^H$ in that it is both a restriction in an arbitrary Mackey functor and the Burnside Mackey functor, but it will always be clear from context what is meant. The Lewis diagram for $\underline{A}$ is given in \cref{Burnside}, where the abelian group $A(H)$ is given the basis ordered as written and all of the maps depicted are described by matrices with respect to these ordered basis.
\end{definition} 

\begin{figure} 
\caption{The Burnside Mackey Functor}
\label{Burnside}
\begin{tikzpicture}[scale=1]
\node (AK) at (0,6) {\scriptsize $\Z\{\K,\K/L,\K/D,\K/R,1\}$};
\node (AL) at (-3,3) {\scriptsize $\Z\{L,1\}$};
\node (AD) at (0,3) {\scriptsize $\Z\{D,1\}$};
\node (AR) at (3,3) {\scriptsize $\Z\{R,1\}$};
\node (Ae) at (0,0) {\scriptsize $\Z$};

				
\draw[bend right=10,->] (AL) to node[left={0.5ex}] {\mylittlematrix{2 & 1}} (Ae);
\draw[bend right=10,->] (AD) to node[left] {\mylittlematrix{2 & 1}} (Ae);
\draw[bend right=10,->] (AR) to node[above={1ex}] {\mylittlematrix{2 & 1}} (Ae);

				
\draw[bend right=10,->] (Ae) to node[above={0.5ex}] {\mylittlematrix{1 \\ 0}} (AL);
\draw[bend right=10,->] (Ae) to node[right]{\mylittlematrix{1 \\ 0 }} (AR);
\draw[bend right=10,->] (Ae) to node[right]{\mylittlematrix{1 \\ 0 }} (AD);

				
\draw[bend right=10,->] (AK) to node[above={0.5ex}, rotate=45] {\mytinymatrix{2 & 0 & 1 & 1 & 0 \\ 0 & 2 & 0 & 0 & 1}} (AL);
\draw[bend right=10,->] (AK) to node[above={0.5ex}, rotate=90] {\mytinymatrix{2 & 1 & 0 & 1 & 0 \\ 0 & 0 & 2 & 0 & 1}} (AD);
\draw[bend right=10,->] (AK) to node[below right={0.5ex}, rotate=-45] {\mytinymatrix{2 & 1 & 1 & 0 & 0 \\ 0 & 0 & 0 & 2 & 1}} (AR);

				
\draw[bend right=10,->] (AL) to node[below={1ex}] {\mytinymatrix{1 & 0 \\ 0 & 1 \\ 0 & 0 \\ 0 & 0 \\ 0 & 0}} (AK);
\draw[bend right=10,->] (AD) to node[right={0.05ex}]{\mytinymatrix{1 & 0 \\ 0 & 0 \\0 & 1 \\ 0 & 0 \\ 0 & 0 }} (AK);
\draw[bend right=10,->] (AR) to node[above right]{\mytinymatrix{1 & 0 \\ 0 & 0 \\ 0 & 0 \\ 0 & 1 \\ 0 & 0 }} (AK);

\end{tikzpicture}
\end{figure}
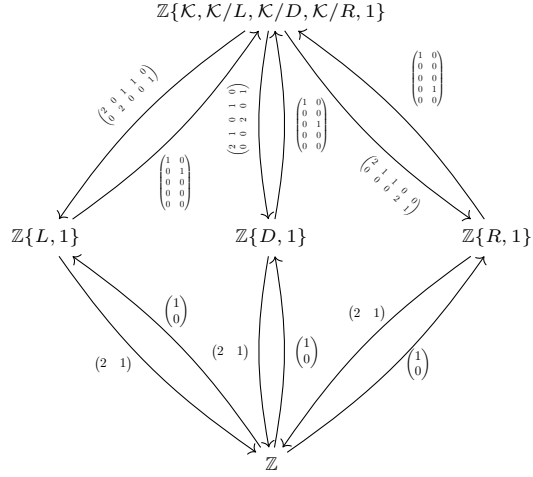

	To compute Bredon cohomology, we'll need to understand all of $\underline{A}_{\K/H} \coloneqq \uparrow_H^\K \downarrow_H^\K \underline{A}$ for $H \leq \K$. These Mackey functors are depicted in \cref{FreeKMackeyFunctors}. 

\begin{figure} 
\caption{Free $\K$-Mackey Functors} 
\label{FreeKMackeyFunctors}
\begin{tabular}{|| c | c ||}
\hline &\\[-1em]

$\underline{A}_{\K/e} \coloneqq \uparrow_e^\K \downarrow_e^\K \underline{A} \cong \uparrow_e^\K \Z$ & $\underline{A}_{\K/L} \coloneqq \uparrow_L^K \downarrow_L^K \underline{A} $ \\

\hline &\\[-1em]

\begin{tikzpicture}[scale=0.75]
\node (K) at (0,6) {\scriptsize $\Z[\K/\K]$};
\node (L) at (-3,3) {\scriptsize $\Z[\K/L]$};
\node (D) at (0,3) {\scriptsize $\Z[\K/D]$};
\node (R) at (3,3) {\scriptsize $\Z[\K/R]$};
\node (e) at (0,0) {\scriptsize $\Z[\K]$};

				
\draw[bend right=10,->] (L) to node[fill=white, rotate=-45, inner sep=1pt] {\tiny $(1+l)$} (e);
\draw[bend right=10,->] (D) to node[fill=white, rotate=90, inner sep=1pt] {\tiny $(1+d)$} (e);
\draw[bend right=10,->] (R) to node[fill=white, rotate=45, inner sep=1pt] {\tiny $(1+r)$} (e);

				
\draw[bend right=10,->] (e) to node[fill=white, inner sep=1pt, rotate=-45] {\tiny $q$} (L);
\draw[bend right=10,->] (e) to node[fill=white, inner sep=1pt] {\tiny $q$} (D);
\draw[bend right=10,->] (e) to node[fill=white, inner sep=1pt, rotate=45] {\tiny $q$} (R);

				
\draw[bend right=10,->] (K) to node[fill=white, inner sep=1pt] {\tiny $\Delta$} (L);
\draw[bend right=10,->] (K) to node[fill=white, inner sep=1pt] {\tiny $\Delta$} (D);
\draw[bend right=10,->] (K) to node[fill=white, inner sep=1pt] {\tiny $\Delta$} (R);

				
\draw[bend right=10,->] (L) to node[fill=white, inner sep=1pt] {\tiny $\nabla$} (K);
\draw[bend right=10,->] (D) to node[fill=white, inner sep=1pt]{\tiny $\nabla$} (K);
\draw[bend right=10,->] (R) to node[fill=white, inner sep=1pt]{\tiny $\nabla$} (K);
\end{tikzpicture}

& 

\begin{tikzpicture}[scale=0.75]
\node (K) at (0,6) {\scriptsize $A(L)$};
\node (L) at (-3,3) {\scriptsize $\Z[\K/L] \otimes A(L)$};
\node (D) at (0,3) {\scriptsize $A(e)$};
\node (R) at (3,3) {\scriptsize $A(e)$};
\node (e) at (0,0) {\scriptsize $\Z[\K/L] \otimes A(e)$};

				
\draw[bend right=10,->] (L) to node[fill=white, rotate=-45, inner sep=1pt] {\tiny $\text{id} \otimes res_e^L$} (e);
\draw[bend right=10,->] (D) to node[fill=white, inner sep=0.8pt] {\tiny $\Delta$} (e);
\draw[bend right=10,->] (R) to node[fill=white, inner sep=1pt] {\tiny $\Delta$} (e);

				
\draw[bend right=10,->] (e) to node[fill=white, rotate=-45, inner sep=1pt] {\tiny $\text{id} \otimes \tau_e^L$} (L);
\draw[bend right=10,->] (e) to node[fill=white, inner sep=0.8pt]{\tiny $\nabla$} (D);
\draw[bend right=10,->] (e) to node[fill=white, inner sep=1pt]{\tiny $\nabla$} (R);

				
\draw[bend right=10,->] (K) to node[fill=white, inner sep=1pt] {\tiny $\Delta$} (L);
\draw[bend right=10,->] (K) to node[fill=white, inner sep=1pt, rotate=90, above] {\tiny $res_e^L$} (D);
\draw[bend right=10,->] (K) to node[fill=white, inner sep=1pt, rotate=-45] {\tiny $res_e^L$} (R);

				
\draw[bend right=10,->] (L) to node[fill=white, inner sep=1pt] {\tiny $\nabla$} (K);
\draw[bend right=10,->] (D) to node[fill=white, inner sep=1pt, rotate=90, below]{\tiny $\tau_e^L$} (K);
\draw[bend right=10,->] (R) to node[fill=white, inner sep=1pt, rotate=-45]{\tiny $\tau_e^L$} (K);

\end{tikzpicture}

\\

\hline &\\[-1em]

$\underline{A}_{\K/D} \coloneqq \uparrow_D^K \downarrow_D^K \underline{A}$ & $\underline{A}_{\K/R} \coloneqq \uparrow_R^K \downarrow_R^K \underline{A}$

\\

\hline &\\[-1em]

\begin{tikzpicture}[scale=0.75]
\node (K) at (0,6) {\scriptsize $A(D)$};
\node (L) at (-3,3) {\scriptsize $A(e)$};
\node (D) at (0,3) {\scriptsize $\Z[\K/D] \otimes A(D)$};
\node (R) at (3,3) {\scriptsize $A(e)$};
\node (e) at (0,0) {\scriptsize $\Z[\K/D] \otimes A(e)$};

				
\draw[bend right=10,->] (L) to node[fill=white, inner sep=1pt] {\tiny $\Delta$} (e);
\draw[bend right=10,->] (D) to node[fill=white, rotate=90, inner sep=1pt, above] {\tiny $\text{id} \otimes res_e^D$} (e);
\draw[bend right=10,->] (R) to node[fill=white, inner sep=1pt] {\tiny $\Delta$} (e);

				
\draw[bend right=10,->] (e) to node[fill=white, inner sep=1pt] {\tiny $\nabla$} (L);
\draw[bend right=10,->] (e) to node[fill=white, rotate=90, inner sep=0.5pt, below]{\tiny $\text{id} \otimes \tau_e^D$} (D);
\draw[bend right=10,->] (e) to node[fill=white, inner sep=1pt]{\tiny $\nabla$} (R);

				
\draw[bend right=10,->] (K) to node[fill=white, inner sep=1pt, rotate=45] {\tiny $res_e^D$} (L);
\draw[bend right=10,->] (K) to node[fill=white, inner sep=1pt] {\tiny $\Delta$} (D);
\draw[bend right=10,->] (K) to node[fill=white, inner sep=1pt, rotate=-45] {\tiny $res_e^D$} (R);

				
\draw[bend right=10,->] (L) to node[fill=white, inner sep=1pt, rotate=45] {\tiny $\tau_e^D$} (K);
\draw[bend right=10,->] (D) to node[fill=white, inner sep=1pt]{\tiny $\nabla$} (K);
\draw[bend right=10,->] (R) to node[fill=white, inner sep=1pt, rotate=-45]{\tiny $\tau_e^D$} (K);

\end{tikzpicture}

&

\begin{tikzpicture}[scale=0.75]
\node (K) at (0,6) {\scriptsize $A(R)$};
\node (L) at (-3,3) {\scriptsize $A(e)$};
\node (D) at (0,3) {\scriptsize $A(e)$};
\node (R) at (3,3) {\scriptsize $\Z[\K/R] \otimes A(R)$};
\node (e) at (0,0) {\scriptsize $\Z[\K/R] \otimes A(e)$};

				
\draw[bend right=10,->] (L) to node[fill=white, inner sep=1pt] {\tiny $\Delta$} (e);
\draw[bend right=10,->] (D) to node[fill=white, inner sep=0.8pt] {\tiny $\Delta$} (e);
\draw[bend right=10,->] (R) to node[fill=white, inner sep=1pt, rotate=45] {\tiny $\text{id} \otimes res_e^R$} (e);

				
\draw[bend right=10,->] (e) to node[fill=white, inner sep=1pt] {\tiny $\nabla$} (L);
\draw[bend right=10,->] (e) to node[fill=white, inner sep=0.8pt]{\tiny $\nabla$} (D);
\draw[bend right=10,->] (e) to node[fill=white, inner sep=1pt, rotate=45]{\tiny $\text{id} \otimes \tau_e^R$} (R);

				
\draw[bend right=10,->] (K) to node[fill=white, inner sep=1pt, rotate=45] {\tiny $res_e^R$} (L);
\draw[bend right=10,->] (K) to node[fill=white, inner sep=1pt, rotate=90, above] {\tiny $res_e^R$} (D);
\draw[bend right=10,->] (K) to node[fill=white, inner sep=1pt] {\tiny $\Delta$} (R);

				
\draw[bend right=10,->] (L) to node[fill=white, inner sep=1pt, rotate=45] {\tiny $\tau_e^R$} (K);
\draw[bend right=10,->] (D) to node[fill=white, inner sep=1pt,rotate=90, below]{\tiny $\tau_e^R$} (K);
\draw[bend right=10,->] (R) to node[fill=white, inner sep=1pt]{\tiny $\nabla$} (K);
\end{tikzpicture}

\\

\hline 

\end{tabular}
\end{figure}
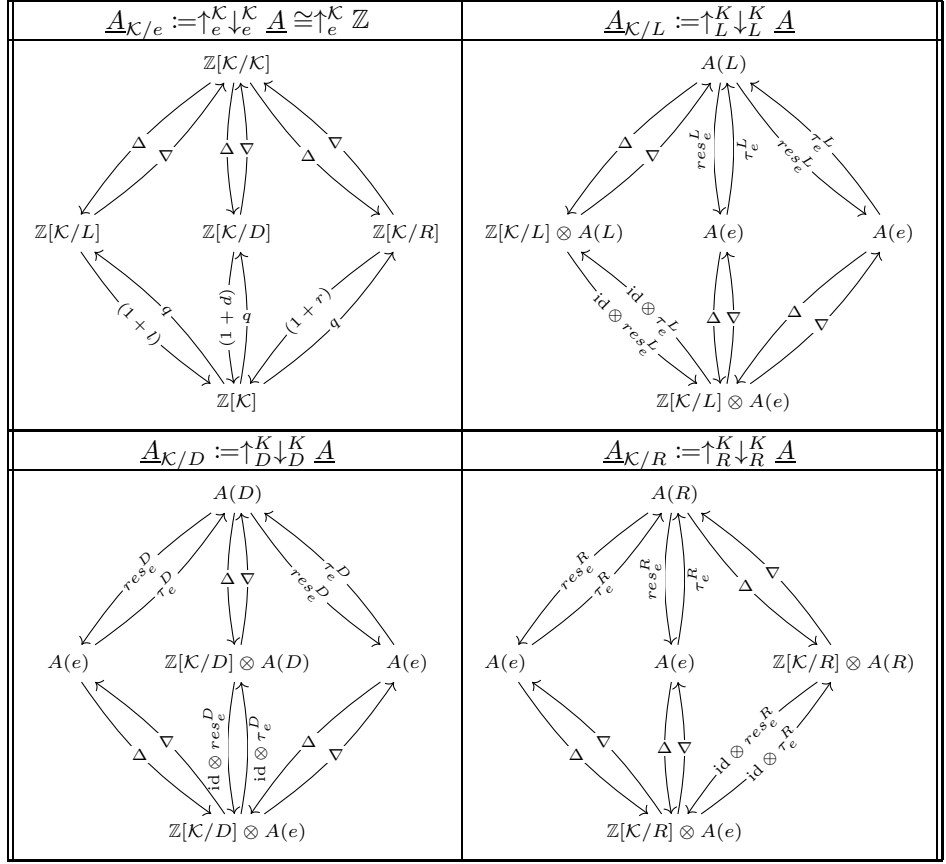

\subsection{Cellular Bredon (Co)homology}
	Throughout this section, let $X$ denote a based $\K$-CW complex and $\underline{M}$ be a Mackey functor for $\K$. By definition, a $\K$-CW complex is built out of cells of the form $\K/H \times D^n$, where $\K$ acts diagonally and $D_n$ is equipped with a trivial action. 
\begin{definition}
The Mackey functor valued Bredon homology will be the homology of the chain complex of Mackey functors 
	\[ \underline{C}^\K_i(X ; \underline{M}) \cong \bigoplus_{H_j} \underline{A}_{\K/H_j} \boxtimes \underline{M} \cong \bigoplus_{H_j} \uparrow_{H_j}^\K \downarrow_{H_j}^{\K} \underline{M}, \]
	where $H_j$ ranges over the $H_j$ appearing in the $i$-cells of $X$. The differentials are induced by the attaching maps in the $\K$-CW structure. 
\end{definition}

\begin{definition}
The Mackey functor valued Bredon cohomology is the cohomology of the cochain complex of Mackey functors 
\[ 
\underline{C}_\K^i(X ; \underline{M}) \cong \underline{\Hom} \left( \bigoplus_{H_j} \underline{A}_{\K/H_j}, \underline{M} \right).
\] 	
\end{definition}
	
Reduced Bredon cohomology with coefficients in a Mackey functor $\underline{M}$ is represented by the genuine $G$-spectrum $H\underline{M}$. These genuine $\K$-spectra are equivariant analogues of Eilenberg-Mac Lane spectra. In particular, they satisfy
	\[ \underline{\pi}_n(H\underline{M}) \cong 
	\begin{cases}
		\underline{M} & \text{if $n=0$} \\
		\underline{0} & \text{else}.  
	\end{cases}
	\]
With this in mind, $H\underline{M}$ is called an equivariant Eilenberg-Mac Lane $\K$-spectrum and we will refer to Bredon cohomology as ordinary equivariant cohomology. The homotopy of $H\underline{M}$ that isn't purely integer graded need not vanish, of course. In the case of $\underline{M} = \underline{A}$, this is precisely what will be computed here. 

\begin{example}
\label{sigma_Hhomology}
We'll compute the Bredon cohomology of $S^{k\sigma_H}$ for $k \geq 1$ and $H \in \{L,D,R\}$. We can build $S^{k\sigma_H}$ with two 0-cells of type $\K/\K$ (fixed cells) and a single cell of type $\K/H$ in each dimension $1 \leq i \leq k$. This gives a (reduced) cellular chain complex of Mackey functors as follows: 
	\[ \underline{A}_{\K/\K} \cong \underline{A} \leftarrow \underline{A}_{\K/H} \leftarrow \cdots \leftarrow \underline{A}_{\K/H} \leftarrow \underline{0} \leftarrow \cdots. \]
	By \cref{MapOutOfFreeFormula}, we can label a differential $\underline{A}_{\K/J} \rightarrow \underline{A}_{\K/J'}$ by what it is on level $\K/J$ as this will determine the entire map. Written in this fashion, we have
	\begin{center}
	\begin{tikzpicture}[scale=1]
	\node (00) at (-7,0) {$\underline{A}$};
	\node (10) at (-4.667,0) {$\underline{A}_{\K/H}$};
	\node (20) at (-2.333,0) {$\underline{A}_{\K/H}$};
	\node (30) at (0,0) {$\cdots$}; 
	\node (40) at (2.333,0) {$\underline{A}_{\mathcal{K}/H}$};
	\node (50) at (4.667,0) {$\underline{0}$}; 
	\node (60) at (7,0) {$\cdots$};
	\draw[->] (10) to node[above] {\tiny $\nabla$} (00);
	\draw[->] (20) to node[above] {\tiny $(\text{id}-\varepsilon) \otimes \text{id}$} (10);
	\draw[->] (30) to node[above] {\tiny $(\text{id}+\varepsilon) \otimes \text{id}$} (20);
	\draw[bend right=10, color=evendiff, ->] (40) to node[above] {\tiny $(\text{id}-\varepsilon) \otimes \text{id}$} (30);
	\draw[bend left=10, color=odddiff, ->] (40) to node[below] {\tiny $(\text{id}+\varepsilon) \otimes \text{id}$} (30); 
	\draw[->] (50) to (40);
	\draw[->] (60) to (50);
	\end{tikzpicture},
	\end{center}
	where $\varepsilon$ is the generator of $\K/H$. Further, we have denoted the first map by $\nabla$, the map at level $\K/H$, since this determines the map algebraically. However, it is often the case that one is only interested in computing the $(\K/\K)$-level of any homotopy Mackey functors. With this in mind, it is worth noting that the first map is given by the transfer $\tau_H^\K$ at the top level. The blue, top (red, bottom) arrow is the differential used when $k$ is even (odd). Let's consider $H=L$ and $k \geq 3$. The first map $\nabla$ is the following, where we are using that $\Z[\K/L] \otimes (-) \cong (-)^{\oplus 2}$: \\
	
	\textbf{Convention:} Throughout this paper, a map of $\K$-Mackey functors will appear as below. The three horizontal maps - top, middle, and bottom - in the middle of this picture, though not depicted this way, indicate the maps between the levels $\K/L$, $\K/D$, and $\K/R$, respectively. This is purely an aesthetic choice made to allow maps to be depicted without overlapping arrows.

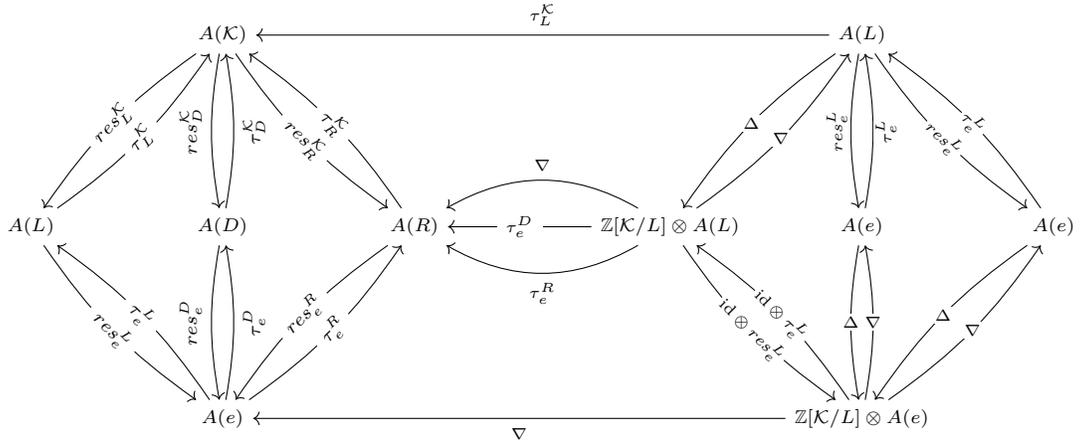
\begin{center}
{ 
\begin{tikzpicture}[scale=0.85]\label{mapexample}

				
\node (AK) at (-5,6) {\scriptsize $A(\K)$};
\node (AL) at (-8,3) {\scriptsize $A(L)$};
\node (AD) at (-5,3) {\scriptsize $A(D)$};
\node (AR) at (-2,3) {\scriptsize $A(R)$};
\node (Ae) at (-5,0) {\scriptsize $A(e)$};

				
\draw[bend right=10,->] (AL) to node[fill=white, rotate=-45] {\scriptsize $res_e^L$} (Ae);
\draw[bend right=10,->] (AD) to node[above, rotate=90] {\scriptsize  $res_e^D$} (Ae);
\draw[bend right=10,->] (AR) to node[fill=white, rotate=45] {\scriptsize $res_e^R$} (Ae);

				
\draw[bend right=10,->] (Ae) to node[fill=white, rotate=-45] {\scriptsize  $\tau_e^L$} (AL);
\draw[bend right=10,->] (Ae) to node[fill=white, rotate=45]{\scriptsize  $\tau_e^R$} (AR);
\draw[bend right=10,->] (Ae) to node[below, rotate=90]{\scriptsize  $\tau_e^D$} (AD);

				
\draw[bend right=10,->] (AK) to node[fill=white, rotate=45] {\scriptsize $res_L^\K$} (AL);
\draw[bend right=10,->] (AK) to node[above, rotate=90] {\scriptsize $res_D^\K$} (AD);
\draw[bend right=10,->] (AK) to node[fill=white, rotate=-45] {\scriptsize $res_R^\K$} (AR);

				
\draw[bend right=10,->] (AL) to node[fill=white, rotate=45] {\scriptsize $\tau_L^\K$} (AK);
\draw[bend right=10,->] (AD) to node[rotate=90, below]{\scriptsize $\tau_D^\K$} (AK);
\draw[bend right=10,->] (AR) to node[fill=white, rotate=-45]{\scriptsize $\tau_R^\K$} (AK);

				
\node (K) at (5,6) {\scriptsize $A(L)$};
\node (L) at (2,3) {\scriptsize $\Z[\K/L] \otimes A(L)$};
\node (D) at (5,3) {\scriptsize $A(e)$};
\node (R) at (8,3) {\scriptsize $A(e)$};
\node (e) at (5,0) {\scriptsize $\Z[\K/L] \otimes A(e)$};

				
\draw[bend right=10,->] (L) to node[fill=white, inner sep=1pt, rotate=-45] {\tiny $\text{id} \otimes res_e^L$} (e);
\draw[bend right=10,->] (D) to node[fill=white, inner sep=1pt] {\tiny $\Delta$} (e);
\draw[bend right=10,->] (R) to node[fill=white, inner sep=1pt] {\tiny $\Delta$} (e);

				
\draw[bend right=10,->] (e) to node[fill=white, inner sep=1pt, rotate=-45] {\tiny $\text{id} \otimes \tau_e^L$} (L);
\draw[bend right=10,->] (e) to node[fill=white, inner sep=1pt]{\tiny $\nabla$} (D);
\draw[bend right=10,->] (e) to node[fill=white, inner sep=1pt]{\tiny $\nabla$} (R);

				
\draw[bend right=10,->] (K) to node[fill=white, inner sep=1pt] {\tiny $\Delta$} (L);
\draw[bend right=10,->] (K) to node[fill=white, inner sep=1pt, rotate=90, above] {\tiny $res_e^L$} (D);
\draw[bend right=10,->] (K) to node[fill=white, inner sep=1pt, rotate=-45] {\tiny $res_e^L$} (R);

				
\draw[bend right=10,->] (L) to node[fill=white, inner sep=1pt] {\tiny $\nabla$} (K);
\draw[bend right=10,->] (D) to node[fill=white, inner sep=0.8pt, rotate=90, below]{\tiny $\tau_e^L$} (K);
\draw[bend right=10,->] (R) to node[fill=white, inner sep=1pt, rotate=-45]{\tiny $\tau_e^L$} (K);

				
\draw[->] (K) to node[above] {\tiny $\tau_L^\K$} (AK); 
\draw[bend right=30,->] (L) to node[above] {\tiny $\nabla$} (AR);
\draw[->] (L) to node[fill=white] {\tiny $\tau_e^D$} (AR);
\draw[bend left=30,->] (L) to node[below] {\tiny $\tau_e^R$} (AR);
\draw[->] (e) to node[below] {\tiny $\nabla$} (Ae); 

\end{tikzpicture}
}
\end{center}	
The second map, $(1-r) \otimes \text{id}$, is as follows: 	
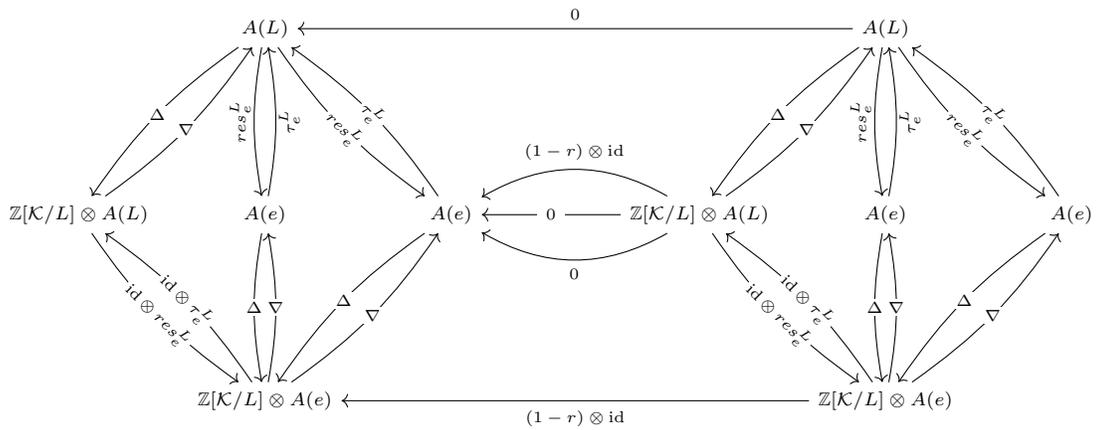
\begin{center}
{ 
\begin{tikzpicture}[scale=0.825]

				
\node (SK) at (5,6) {\scriptsize $A(L)$};
\node (SL) at (2,3) {\scriptsize $\Z[\K/L] \otimes A(L)$};
\node (SD) at (5,3) {\scriptsize $A(e)$};
\node (SR) at (8,3) {\scriptsize $A(e)$};
\node (Se) at (5,0) {\scriptsize $\Z[\K/L] \otimes A(e)$};

				
\draw[bend right=10,->] (SL) to node[fill=white,rotate=-45, inner sep=1pt] {\tiny $\text{id} \otimes res_e^L$} (Se);
\draw[bend right=10,->] (SD) to node[fill=white, inner sep=1pt] {\tiny $\Delta$} (Se);
\draw[bend right=10,->] (SR) to node[fill=white, inner sep=1pt] {\tiny $\Delta$} (Se);

				
\draw[bend right=10,->] (Se) to node[fill=white, rotate=-45, inner sep=1pt] {\tiny $\text{id} \otimes \tau_e^L$} (SL);
\draw[bend right=10,->] (Se) to node[fill=white, inner sep=1pt]{\tiny $\nabla$} (SD);
\draw[bend right=10,->] (Se) to node[fill=white, inner sep=1pt]{\tiny $\nabla$} (SR);

				
\draw[bend right=10,->] (SK) to node[fill=white, inner sep=1pt] {\tiny $\Delta$} (SL);
\draw[bend right=10,->] (SK) to node[above, rotate=90, inner sep=1pt] {\tiny $res_e^L$} (SD);
\draw[bend right=10,->] (SK) to node[rotate=-45, fill=white, inner sep=1pt] {\tiny $res_e^L$} (SR);

				
\draw[bend right=10,->] (SL) to node[fill=white, inner sep=1pt] {\tiny $\nabla$} (SK);
\draw[bend right=10,->] (SD) to node[rotate=90, below, inner sep=1pt]{\tiny $\tau_e^L$} (SK);
\draw[bend right=10,->] (SR) to node[fill=white, rotate=-45, inner sep=1pt]{\tiny $\tau_e^L$} (SK);

				
\node (TK) at (-5,6) {\scriptsize $A(L)$};
\node (TL) at (-8,3) {\scriptsize $\Z[\K/L] \otimes A(L)$};
\node (TD) at (-5,3) {\scriptsize $A(e)$};
\node (TR) at (-2,3) {\scriptsize $A(e)$};
\node (Te) at (-5,0) {\scriptsize $\Z[\K/L] \otimes A(e)$};

				
\draw[bend right=10,->] (TL) to node[fill=white, rotate=-45, inner sep=1pt] {\tiny $\text{id} \otimes res_e^L$} (Te);
\draw[bend right=10,->] (TD) to node[fill=white, inner sep=1pt] {\tiny $\Delta$} (Te);
\draw[bend right=10,->] (TR) to node[fill=white, inner sep=1pt] {\tiny $\Delta$} (Te);

				
\draw[bend right=10,->] (Te) to node[fill=white, rotate=-45, inner sep=1pt] {\tiny $\text{id} \otimes \tau_e^L$} (TL);
\draw[bend right=10,->] (Te) to node[fill=white, inner sep=1pt]{\tiny $\nabla$} (TD);
\draw[bend right=10,->] (Te) to node[fill=white, inner sep=1pt]{\tiny $\nabla$} (TR);

				
\draw[bend right=10,->] (TK) to node[fill=white, inner sep=1pt] {\tiny $\Delta$} (TL);
\draw[bend right=10,->] (TK) to node[above, rotate=90, inner sep=1pt] {\tiny $res_e^L$} (TD);
\draw[bend right=10,->] (TK) to node[fill=white, rotate=-45, inner sep=1pt] {\tiny $res_e^L$} (TR);

				
\draw[bend right=10,->] (TL) to node[fill=white, inner sep=1pt] {\tiny $\nabla$} (TK);
\draw[bend right=10,->] (TD) to node[below, rotate=90, inner sep=1pt]{\tiny $\tau_e^L$} (TK);
\draw[bend right=10,->] (TR) to node[fill=white, rotate=-45, inner sep=1pt]{\tiny $\tau_e^L$} (TK);

				
\draw[->] (SK) to node[above] {\tiny $0$} (TK); 
\draw[bend right=30,->] (SL) to node[above] {\tiny $(1-r) \otimes \text{id}$} (TR);
\draw[->] (SL) to node[fill=white] {\tiny $0$} (TR);
\draw[bend left=30,->] (SL) to node[below] {\tiny $0$} (TR);
\draw[->] (Se) to node[below] {\tiny $(1-r) \otimes \text{id}$} (Te); 

\end{tikzpicture}
}
\end{center}		
The third map, $(1+r) \otimes \text{id}$, is 
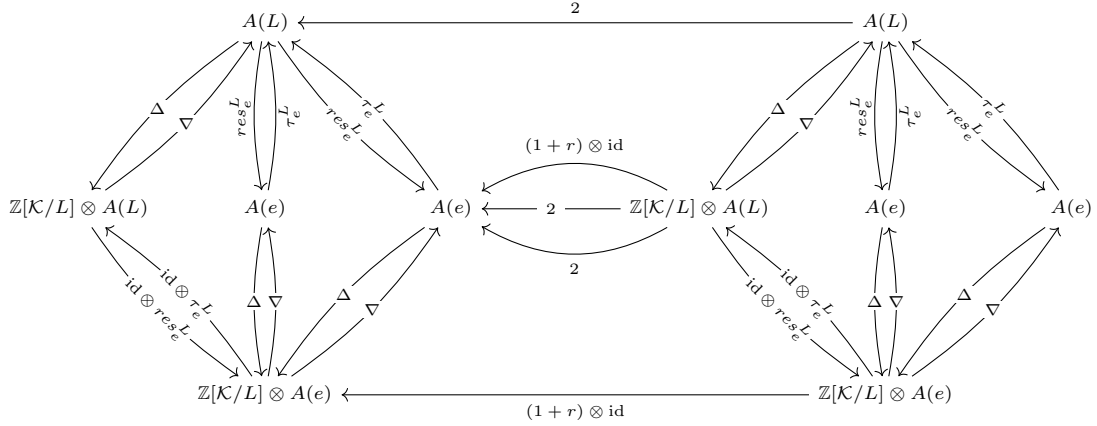
\begin{center}
{ 
\begin{tikzpicture}[scale=0.825]

				
\node (SK) at (5,6) {\scriptsize $A(L)$};
\node (SL) at (2,3) {\scriptsize $\Z[\K/L] \otimes A(L)$};
\node (SD) at (5,3) {\scriptsize $A(e)$};
\node (SR) at (8,3) {\scriptsize $A(e)$};
\node (Se) at (5,0) {\scriptsize $\Z[\K/L] \otimes A(e)$};

				
\draw[bend right=10,->] (SL) to node[fill=white, inner sep=1pt, rotate=-45] {\tiny $\text{id} \otimes res_e^L$} (Se);
\draw[bend right=10,->] (SD) to node[fill=white, inner sep=1pt] {\tiny $\Delta$} (Se);
\draw[bend right=10,->] (SR) to node[fill=white, inner sep=1pt] {\tiny $\Delta$} (Se);

				
\draw[bend right=10,->] (Se) to node[fill=white, inner sep=1pt, rotate=-45] {\tiny $\text{id} \otimes \tau_e^L$} (SL);
\draw[bend right=10,->] (Se) to node[fill=white, inner sep=1pt]{\tiny $\nabla$} (SD);
\draw[bend right=10,->] (Se) to node[fill=white, inner sep=1pt]{\tiny $\nabla$} (SR);

				
\draw[bend right=10,->] (SK) to node[fill=white, inner sep=1pt] {\tiny $\Delta$} (SL);
\draw[bend right=10,->] (SK) to node[inner sep=1pt, rotate=90, above] {\tiny $res_e^L$} (SD);
\draw[bend right=10,->] (SK) to node[fill=white, inner sep=1pt, rotate=-45] {\tiny $res_e^L$} (SR);

				
\draw[bend right=10,->] (SL) to node[fill=white, inner sep=1pt] {\tiny $\nabla$} (SK);
\draw[bend right=10,->] (SD) to node[rotate=90, below, inner sep=1pt]{\tiny $\tau_e^L$} (SK);
\draw[bend right=10,->] (SR) to node[fill=white, inner sep=1pt, rotate=-45]{\tiny $\tau_e^L$} (SK);

				
\node (TK) at (-5,6) {\scriptsize $A(L)$};
\node (TL) at (-8,3) {\scriptsize $\Z[\K/L] \otimes A(L)$};
\node (TD) at (-5,3) {\scriptsize $A(e)$};
\node (TR) at (-2,3) {\scriptsize $A(e)$};
\node (Te) at (-5,0) {\scriptsize $\Z[\K/L] \otimes A(e)$};

				
\draw[bend right=10,->] (TL) to node[fill=white, inner sep=1pt, rotate=-45] {\tiny $\text{id} \otimes res_e^L$} (Te);
\draw[bend right=10,->] (TD) to node[fill=white, inner sep=1pt] {\tiny $\Delta$} (Te);
\draw[bend right=10,->] (TR) to node[fill=white, inner sep=1pt] {\tiny $\Delta$} (Te);

				
\draw[bend right=10,->] (Te) to node[fill=white, inner sep=1pt, rotate=-45] {\tiny $\text{id} \otimes \tau_e^L$} (TL);
\draw[bend right=10,->] (Te) to node[fill=white, inner sep=1pt]{\tiny $\nabla$} (TD);
\draw[bend right=10,->] (Te) to node[fill=white, inner sep=1pt]{\tiny $\nabla$} (TR);

				
\draw[bend right=10,->] (TK) to node[fill=white, inner sep=1pt] {\tiny $\Delta$} (TL);
\draw[bend right=10,->] (TK) to node[inner sep=1pt, rotate=90, above] {\tiny $res_e^L$} (TD);
\draw[bend right=10,->] (TK) to node[fill=white, inner sep=1pt, rotate=-45] {\tiny $res_e^L$} (TR);

				
\draw[bend right=10,->] (TL) to node[fill=white, inner sep=1pt] {\tiny $\nabla$} (TK);
\draw[bend right=10,->] (TD) to node[inner sep=1pt, rotate=90, below]{\tiny $\tau_e^L$} (TK);
\draw[bend right=10,->] (TR) to node[fill=white, inner sep=1pt, rotate=-45]{\tiny $\tau_e^L$} (TK);

				
\draw[->] (SK) to node[above] {\tiny $2$} (TK); 
\draw[bend right=30,->] (SL) to node[above] {\tiny $(1+r) \otimes \text{id}$} (TR);
\draw[->] (SL) to node[fill=white] {\tiny $2$} (TR);
\draw[bend left=30,->] (SL) to node[below] {\tiny $2$} (TR);
\draw[->] (Se) to node[below] {\tiny $(1+r) \otimes \text{id}$} (Te); 

\end{tikzpicture}
}
\end{center}
From here, we can compute $\underline{\pi}^\K_{p-q\sigma_H}(H\underline{A}) \cong \widetilde{\underline{H}}_p^\K(S^{q\sigma_H}; \underline{A})$ for $p,q \geq 0$. Before recording this answer, however, let's compute the corresponding cochain complexes as well. They are as follows:  
	\begin{center}
	\begin{tikzpicture}[scale=1]
	\node (00) at (-7,0) {$\underline{A}$};
	\node (10) at (-4.667,0) {$\underline{A}_{\K/H}$};
	\node (20) at (-2.333,0) {$\underline{A}_{\K/H}$};
	\node (30) at (0,0) {$\cdots$}; 
	\node (40) at (2.333,0) {$\underline{A}_{\K/H}$};
	\node (50) at (4.667,0) {$\underline{0}$}; 
	\node (60) at (7,0) {$\cdots$};
	\draw[->] (00) to node[above] {\tiny $res_H^\K$} (10);
	\draw[->] (10) to node[above] {\tiny $(\text{id}-\varepsilon) \otimes \text{id}$} (20);
	\draw[->] (20) to node[above] {\tiny $(\text{id}+\varepsilon) \otimes \text{id}$} (30);
	\draw[bend left=10, color=evendiff, ->] (30) to node[above] {\tiny $(\text{id}-\varepsilon) \otimes \text{id}$} (40);
	\draw[bend right=10, color=odddiff, ->] (30) to node[below] {\tiny $(\text{id}+\varepsilon) \otimes \text{id}$} (40); 
	\draw[->] (40) to (50);
	\draw[->] (50) to (60);
	\end{tikzpicture}
	\end{center}
	
\begin{prop}
\label{OneRepHomotopyProp}
The homotopy groups $\underline{\pi}^\K_{p+q\sigma_H}(H\underline{A})(\K/\K)$ are as in \cref{OneRepHomotopy}.   
\end{prop}

\begin{figure}
\caption{$\underline{\pi}^\K_{p + q\sigma_H}(H\underline{A})(\K/\K)$}
\label{OneRepHomotopy}
\begin{tikzpicture}[scale=0.6]
      	    \draw[gray] (-10,-10) grid (10,10); 
            \draw[->] (0,-10) to (0,10); 
            \draw[->] (-10,0) to (10,0);

            \node at (9.5,0.5) [draw] {$p \cdot 1$};
            \node at (-0.6,9.5) [draw, rotate=90] {$q \cdot \sigma_H$};

            \foreach \n in {1,2,3,4,5} {
                \filldraw (-2*\n,2*\n) circle (3pt);
            }





            
            \foreach \x in {3,5,7,9} {
                \foreach \y in {\x,...,10} {
                \draw[fill=white!100] (-\x,\y) circle (3pt);
                }
            }

            \foreach \x in {2,4,6,8} {
                \filldraw (\x,-\x) circle (3pt);
                \foreach \y in {\x,...,9} {
                    \draw[fill=white!100] (\x,-\y-1) circle (3pt);
                }
            }

            \filldraw (10,-10) circle (3pt);

            \foreach \x in {-1,-2,...,-10} {
                \node at (0,\x) {$\whitesquare$};
            }

            \node at (0,0) {$\fillsquare$};

            \foreach \x in {1,2,...,10} {
                \node at (-0.5,-\x+0.25) {}; 
            }

            \foreach \x in {1,2,3,4} {
            \node at (2*\x+0.5,-2*\x+0.5) {}; 
            \node at (-2*\x-0.5,2*\x-0.5) {}; 
            }

            \foreach \x in {1,...,9} {
                \node at (0,\x) {$\widehat{\whitesquare}$};
            }

            \node at (0.3,1.25) {}; 

            \foreach \x in {1,...,8} {
                \node at (0.75, 1.25 + \x) {}; 
            }
        \end{tikzpicture}
        
        Key: $\bullet = A(H)$ \quad $\circ = A(H)/2A(H) \quad \blacksquare = A(\K) \quad \whitesquare = \text{coker}(\tau_H^\K) \quad \widehat{\whitesquare} = \ker(res_H^K)$ \\
    \end{figure}
\end{example}

\section{Double Complexes, Universal Coefficients, and a Kunneth Spectral Sequence}
\label{KunnethSection}

In this section, we'll give a brief introduction to the spectral sequence associated to a double complex and describe its relationship to both the Kunneth spectral sequence and the Universal Coefficients spectral sequence. We will demonstrate this relationship in the context of $C_2$-equivariant stable homotopy theory, and show that the Kunneth spectral sequence is far less efficient than the spectral sequence associated to a double complex for this task. To that end, recall the real representation ring of $C_2$: 
\[ RO(C_2) \cong \Z\{1, \sigma\}, \]
where $\sigma$ is the real line equipped with a $C_2$-action by multiplication by $-1$. We can quickly compute the homology of $S^\sigma$ with the same method as in \cref{sigma_Hhomology} to get the following: 

\begin{prop} \label{SigmaSusHAHomotopy}
The non-zero homotopy of the $\sigma$-suspension of $H\underline{A}$ is
\[
\underline{\pi}_*^{C_2}(\Sigma^\sigma H\underline{A}) \cong
\begin{cases}
\langle \Z \rangle & *=0 \\
\underline{f} & *=1.
\end{cases}
\]
\end{prop} 

To compute the reduced homology of $S^{2\sigma} \simeq S^\sigma \wedge S^\sigma$, observe that the reduced chain complex for the smash product $S^\sigma \wedge S^\sigma$ is the two-fold tensor power of the reduced chain complex for $S^\sigma$. In particular, it is the total complex of the double complex 
\[ \underline{D}_{i,j} \cong \widetilde{\underline{C}}_i(S^\sigma) \boxtimes \widetilde{\underline{C}}_j(S^\sigma). \]
From this, we get the spectral sequence associated to the double complex with $E^2$-page 
\[ E^2_{i,j} \cong \widetilde{\underline{H}}_i^{C_2}(S^{\sigma}; \widetilde{\underline{H}}_j^{C_2}(S^{\sigma})) \Rightarrow \widetilde{\underline{H}}_{i+j}^{C_2} (S^\sigma \wedge S^\sigma; \underline{A}) \cong \widetilde{\underline{H}}_{i+j}^{C_2}(S^{2\sigma}; \underline{A}) \]
as depicted in \cref{C2DoubleComplexSS}. 

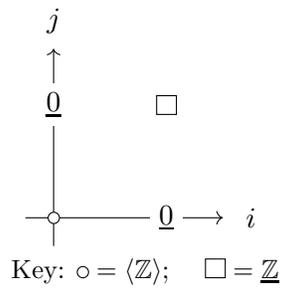
\begin{figure}
\caption{$E^2$-page of the Double Complex Spectral Sequence for $\widetilde{\underline{H}}_*^{C_2}(S^{2\sigma};\underline{A})$}
\label{C2DoubleComplexSS}
	\begin{tikzpicture}[scale=1.5,font=\tiny]
	
            	
	\draw[->] (0,-0.25) to (0,1.5); 
	\draw[->] (-0.25,0) to (1.5,0);
	
	\node at (1.75,0) {\large $i$};
	\node at (0,1.75) {\large $j$};
	
	
	\node at (0,0) {\large $\whitecirc$};
	\node[fill=white] at (0,1) {\large $\underline{0}$};
	\node[fill=white] at (1,0) {\large $\underline{0}$};
	\node at (1,1) {\large $\whitesquare$}; 
	\end{tikzpicture} \\
	Key: $\whitecirc = \langle \Z \rangle$; \quad $\whitesquare = \underline{\Z}$
\end{figure}
There isn't any room for differentials, so the spectral sequence collapses. We obtain a computation of the homology of $S^{2\sigma}$. 
\begin{prop}
The non-zero homotopy of the $2\sigma$-suspension of $H\underline{A}$ is given by
\[
\underline{\pi}_*^{C_2}(\Sigma^{2\sigma} H\underline{A}) \cong 
\begin{cases}
\langle \Z \rangle & *=0 \\
\underline{\Z} & *=2.
\end{cases}
\]
\end{prop}

	As in \cite{LewisMandell}, the Kunneth spectral sequence for computing $\widetilde{\underline{H}}_*^{C_2}(S^\sigma \wedge S^\sigma; \underline{A})$ has $E^2$-page given by 
\[ E_{p,q}^2 \cong \underline{\Tor}^{\underline{A}}_{p,q}(\widetilde{\underline{H}}_*^{C_2}(S^\sigma); \widetilde{\underline{H}}_*^{C_2}(S^\sigma; \underline{A})) \Rightarrow \widetilde{\underline{H}}_*^{C_2}(S^\sigma \wedge S^\sigma; \underline{A}) \cong \widetilde{\underline{H}}_*^{C_2}(S^{2\sigma}; \underline{A}),
\]
where 
\[
\underline{\Tor}^{\underline{A}}_{p,q}(\widetilde{\underline{H}}_*^{C_2}(S^\sigma); \widetilde{\underline{H}}_*^{C_2}(S^\sigma; \underline{A})) \cong \bigoplus_{s+t=q} \underline{\Tor}^{\underline{A}}_p(\widetilde{\underline{H}}_s^{C_2}(S^\sigma); \widetilde{\underline{H}}_t^{C_2}(S^\sigma; \underline{A})).
\]
The $E^2$-page is depicted in \cref{C2KunnethSS}.
\begin{figure}
\caption{$E^2$-page of the Kunneth Spectral Sequence for $\widetilde{\underline{H}}_*^{C_2}(S^{2\sigma};\underline{A})$}
\label{C2KunnethSS}
	\begin{tikzpicture}[scale=1.5,font=\tiny]
	
            	
	\draw[->] (0,-0.25) to (0,2.5); 
	\draw[->] (-0.25,0) to (7.5,0);
	
	\node at (7.75,0) {\large $p$};
	\node at (0,2.75) {\large $q$};
	
	
	\node at (0,0) {\large $\whitecirc$};
	\node[fill=white] at (1,0) {\large $\underline{0}$};
	\node[fill=white] at (2,0) {\large $\underline{0}$};
	\node (30) at (3,0) {\large $\fillcirc$}; 
	\node[fill=white] at (4,0) {\large $\underline{0}$};
	\node[fill=white] at (5,0) {\large $\underline{0}$};
	\node[fill=white] at (6,0) {\large $\underline{0}$};
	\node[fill=white] (70) at (7,0) {\tiny $\cdots$};
	
	\node[fill=white] at (0,1) {\large $\underline{0}$};
	\node[draw, fill=black, text=white, circle, inner sep=1pt] (11) at (1,1) {\tiny $2$};
	\node at (2,1) {\large $\underline{0}$};
	\node at (3,1) {\large $\underline{0}$};
	\node at (4,1) {\large $\underline{0}$};
	\node[draw, fill=black, text=white, circle, inner sep=1pt] (51) at (5,1) {\tiny $2$};
	\node at (6,1) {\large $\underline{0}$};
	\node at (7,1) {\tiny $\cdots$};
	
	\node[fill=white] at (0,2) {$\whitesquaredual$};
	\node at (1,2) {\large $\underline{0}$};
	\node at (2,2) {\large $\underline{0}$};
	\node (32) at (3,2) {\large $\fillcirc$}; 
	\node at (4,2) {\large $\underline{0}$};
	\node at (5,2) {\large $\underline{0}$};
	\node at (6,2) {\large $\underline{0}$};
	\node at (7,2) {\tiny $\cdots$};
	
	
	\draw[color=\amult, ->] (30) to (11) {};
	\draw[color=\amult, ->] (51) to (32) {};
	\draw[color=\amult, ->] (70) to (51) {};
	\end{tikzpicture} \\
	Key: $\whitecirc = \langle \Z \rangle$; \quad $\fillcirc = \langle \F_2 \rangle$; \quad $\begin{tikzpicture} \node[draw, fill=black, text=white, circle, inner sep=0.5pt] at (0,0) {\tiny $2$}; \end{tikzpicture} = \langle \F_2 \rangle^{\oplus 2}$; \quad $\whitesquaredual = \underline{\Z}^*$
\end{figure}

In \cref{C2KunnethSS}, the differentials are so that the only copy of $\langle \F_2 \rangle$ that survives the spectral sequence is in degree $(1,1)$. Further, there will be no room for differentials after the second page. Thus, the spectral sequence collapses and we are left with solving an extension problem for $\underline{\pi}_2^{C_2}(\Sigma^{2\sigma} H\underline{A})$. The extension problem is of the form 
\[ \underline{0} \rightarrow \whitesquaredual \rightarrow \underline{\pi}_2^{C_2}(\Sigma^{2\sigma} H\underline{A}) \rightarrow \fillcirc \rightarrow \underline{0}, \]
and the solution is $\underline{\Z}$ as we know from \cref{SigmaSusHAHomotopy}. 

We may also regrade the $E^2$-page of the Kunneth spectral sequence to be trigraded and use a universal coefficients spectral sequence to compute the $E^2$-page of the spectral sequence associated to a double complex. That is, we have
\[
E^2_{p,s,t} \cong \Tor_p^{\underline{A}}(\widetilde{\underline{H}}_s^{C_2}(S^\sigma; \widetilde{\underline{H}}_t^{C_2}(S^\sigma; \underline{A})) \Rightarrow \widetilde{\underline{H}}_{p+s}^{C_2}(S^\sigma; \widetilde{\underline{H}}_t^{C_2}(S^\sigma; \underline{A})). 
\]
On this page, the differentials have tridegree $(-2,1,0)$ so we really have a pair of bigraded spectral sequences indexed by the values of $t$. These spectral sequences for $t=0$ and $t=1$ are depicted in \cref{C2UnivCoeffSS0} and \cref{C2UnivCoeffSS1}, respectively. We can see that after the $E^2$-page, the spectral sequences will collapse. In the case of $t=0$, we are left with no extensions to solve. Though this is the case when $t=1$, we run into the same extension problem seen in the Kunneth spectral sequence. 

What we've seen is that there are two spectral sequences for computing homotopy of a smash product: the Kunneth spectral sequence and the spectral sequence associated to a double complex. The Kunneth spectral sequence is less efficient as the Mackey functors typically showing up as homology Mackey functors often have infinite homological dimension. See Theorem 2.1 of \cite{Greenlees}. However, the Kunneth spectral sequence ``factors" through the spectral sequence associated to a double complex via the Universal Coefficients spectral sequence, and it turns out the double complex spectral sequence is much easier to compute in this context.

\begin{figure}
\caption{$E^2_{p,s,0}$-page of the Universal Coefficients Spectral Sequence for $\widetilde{\underline{H}}_*^{C_2}(S^{2\sigma};\underline{A})$}
\label{C2UnivCoeffSS0}
	\begin{tikzpicture}[scale=1.5,font=\tiny]
	
            	
	\draw[->] (0,-0.25) to (0,1.5); 
	\draw[->] (-0.25,0) to (7.5,0);
	
	\node at (7.75,0) {\large $p$};
	\node at (0,1.75) {\large $q$};
	
	
	\node at (0,0) {\large $\whitecirc$};
	\node[fill=white] at (1,0) {\large $\underline{0}$};
	\node[fill=white] at (2,0) {\large $\underline{0}$};
	\node (30) at (3,0) {\large $\fillcirc$}; 
	\node[fill=white] at (4,0) {\large $\underline{0}$};
	\node[fill=white] at (5,0) {\large $\underline{0}$};
	\node[fill=white] at (6,0) {\large $\underline{0}$};
	\node[fill=white] (70) at (7,0) {\tiny $\cdots$};
	
	\node[fill=white] at (0,1) {\large $\underline{0}$};
	\node (11) at (1,1) {\large $\fillcirc$}; 
	\node at (2,1) {\large $\underline{0}$};
	\node at (3,1) {\large $\underline{0}$};
	\node at (4,1) {\large $\underline{0}$};
	\node (51) at (5,1) {\large $\fillcirc$}; 
	\node at (6,1) {\large $\underline{0}$};
	\node at (7,1) {\tiny $\cdots$};
	
	
	\draw[color=\amult, ->] (30) to (11) {};
	\draw[color=\amult, ->] (70) to (51) {};
	 
	\end{tikzpicture} \\
	Key: $\whitecirc = \langle \Z \rangle$; \quad $\fillcirc = \langle \F_2 \rangle$; \quad $\whitesquaredual = \underline{\Z}^*$
\end{figure}
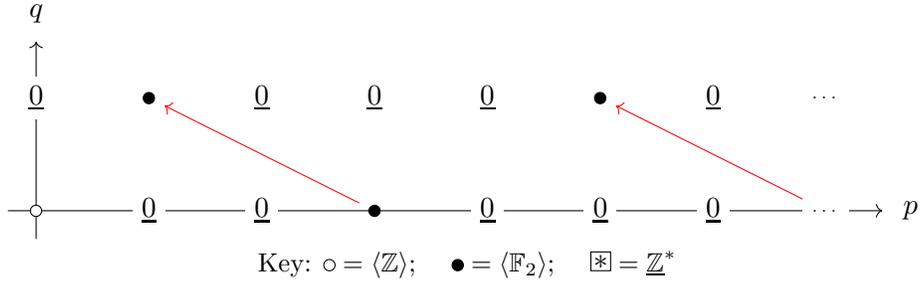

\begin{figure}
\caption{$E^2_{p,s,1}$-page of the Universal Coefficients Spectral Sequence for $\widetilde{\underline{H}}_*^{C_2}(S^{2\sigma};\underline{A})$}
\label{C2UnivCoeffSS1}
	\begin{tikzpicture}[scale=1.5,font=\tiny]
	
            	
	\draw[->] (0,-0.25) to (0,1.5); 
	\draw[->] (-0.25,0) to (7.5,0);
	
	\node at (7.75,0) {\large $p$};
	\node at (0,1.75) {\large $q$};
	
	
	\node[fill=white, inner sep=0.5pt] at (0,0) {\large $\underline{0}$};
	\node[fill=white] (10) at (1,0) {\large $\fillcirc$}; 
	\node[fill=white] at (2,0) {\large $\underline{0}$};
	\node[fill=white] at (3,0) {\large $\underline{0}$};
	\node[fill=white] at (4,0) {\large $\underline{0}$};
	\node[fill=white] (50) at (5,0) {\large $\fillcirc$}; 
	\node[fill=white] at (6,0) {\large $\underline{0}$};
	\node[fill=white] at (7,0) {\tiny $\cdots$};
	
	\node[fill=white] at (0,1) {$\whitesquaredual$};
	\node at (1,1) {\large $\underline{0}$};
	\node at (2,1) {\large $\underline{0}$};
	\node (31) at (3,1) {\large $\fillcirc$}; 
	\node at (4,1) {\large $\underline{0}$};
	\node at (5,1) {\large $\underline{0}$};
	\node at (6,1) {\large $\underline{0}$};
	\node at (7,1) {\tiny $\cdots$};
	
	
	\draw[color=\amult, ->] (50) to (31) {}; 
	\end{tikzpicture} \\
	Key: $\fillcirc = \langle \F_2 \rangle$; \quad $\whitesquaredual = \underline{\Z}^*$
\end{figure}

\section{A Comparison of the Positive Cones for $H\underline{A}$ and $H\underline{\Z}$} \label{ComparisonPositive}

In this section, we'll compute a large range of the ``positive cone" of $H\underline{A}$, i.e. $\underline{\pi}^\K_n(\Sigma^{k\overline{\rho}} H\underline{A})$ for $k \geq 0$ via the following short exact sequence: 
\[ \underline{I} \rightarrow \underline{A} \rightarrow \underline{\Z}, \label{coeffseq} \]
where $\underline{I}$ is the kernel of the augmentation map $\underline{A} \rightarrow \underline{\Z}$:

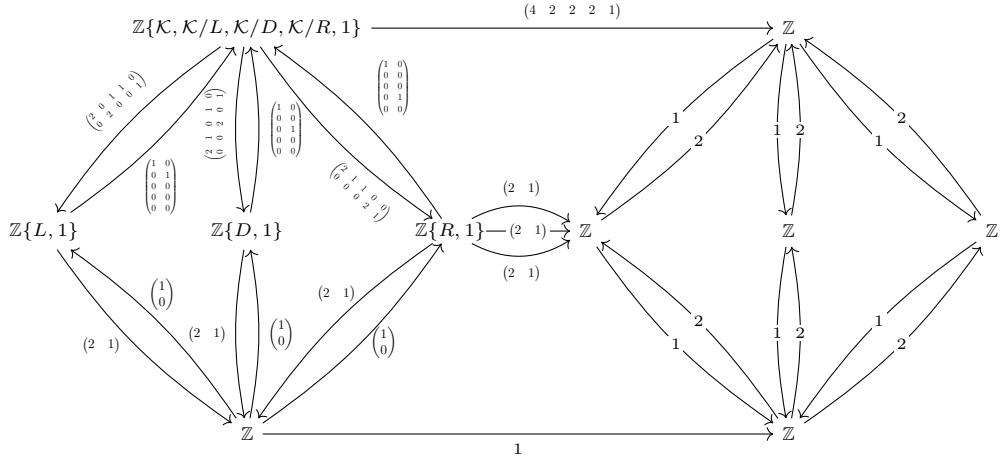
\begin{center}
{ 
\begin{tikzpicture}[scale=0.9]

				
\node (AK) at (-4,6) {\scriptsize $\Z\{\K,\K/L,\K/D,\K/R,1\}$};
\node (AL) at (-7,3) {\scriptsize $\Z\{L,1\}$};
\node (AD) at (-4,3) {\scriptsize $\Z\{D,1\}$};
\node[inner xsep=0pt, inner ysep=1pt] (AR) at (-1,3) {\scriptsize $\Z\{R,1\}$};
\node (Ae) at (-4,0) {\scriptsize $\Z$};

				
\draw[bend right=10,->] (AL) to node[left={0.5ex}] {\mylittlematrix{2 & 1}} (Ae);
\draw[bend right=10,->] (AD) to node[left] {\mylittlematrix{2 & 1}} (Ae);
\draw[bend right=10,->] (AR) to node[above={1ex}] {\mylittlematrix{2 & 1}} (Ae);

				
\draw[bend right=10,->] (Ae) to node[above={0.5ex}] {\mylittlematrix{1 \\ 0}} (AL);
\draw[bend right=10,->] (Ae) to node[right]{\mylittlematrix{1 \\ 0 }} (AR);
\draw[bend right=10,->] (Ae) to node[right]{\mylittlematrix{1 \\ 0 }} (AD);

				
\draw[bend right=10,->] (AK) to node[above={0.5ex}, rotate=45] {\mytinymatrix{2 & 0 & 1 & 1 & 0 \\ 0 & 2 & 0 & 0 & 1}} (AL);
\draw[bend right=10,->] (AK) to node[above={0.5ex}, rotate=90] {\mytinymatrix{2 & 1 & 0 & 1 & 0 \\ 0 & 0 & 2 & 0 & 1}} (AD);
\draw[bend right=10,->] (AK) to node[below right={0.5ex}, rotate=-45] {\mytinymatrix{2 & 1 & 1 & 0 & 0 \\ 0 & 0 & 0 & 2 & 1}} (AR);

				
\draw[bend right=10,->] (AL) to node[below={1ex}] {\mytinymatrix{1 & 0 \\ 0 & 1 \\ 0 & 0 \\ 0 & 0 \\ 0 & 0}} (AK);
\draw[bend right=10,->] (AD) to node[right={0.05ex}]{\mytinymatrix{1 & 0 \\ 0 & 0 \\0 & 1 \\ 0 & 0 \\ 0 & 0 }} (AK);
\draw[bend right=10,->] (AR) to node[above right]{\mytinymatrix{1 & 0 \\ 0 & 0 \\ 0 & 0 \\ 0 & 1 \\ 0 & 0 }} (AK);

				
\node (ZK) at (4,6) {\scriptsize$\Z$};
\node (ZL) at (1,3) {\scriptsize$\Z$};
\node (ZD) at (4,3) {\scriptsize$\Z$};
\node (ZR) at (7,3) {\scriptsize$\Z$};
\node (Ze) at (4,0) {\scriptsize$\Z$};

				
\draw[bend right=10,->] (ZL) to node[fill=white, inner sep=1pt] {\tiny $1$} (Ze);
\draw[bend right=10,->] (ZD) to node[fill=white, inner sep=1pt] {\tiny $1$} (Ze);
\draw[bend right=10,->] (ZR) to node[fill=white, inner sep=1pt] {\tiny $1$} (Ze);

				
\draw[bend right=10,->] (Ze) to node[fill=white, inner sep=1pt] {\tiny $2$} (ZL);
\draw[bend right=10,->] (Ze) to node[fill=white, inner sep=1pt]{\tiny $2$} (ZD);
\draw[bend right=10,->] (Ze) to node[fill=white, inner sep=1pt]{\tiny $2$} (ZR);

				
\draw[bend right=10,->] (ZK) to node[fill=white, inner sep=1pt] {\tiny $1$} (ZL);
\draw[bend right=10,->] (ZK) to node[fill=white, inner sep=1pt] {\tiny $1$} (ZD);
\draw[bend right=10,->] (ZK) to node[fill=white, inner sep=1pt] {\tiny $1$} (ZR);

				
\draw[bend right=10,->] (ZL) to node[fill=white, inner sep=1pt] {\tiny $2$} (ZK);
\draw[bend right=10,->] (ZD) to node[fill=white, inner sep=1pt]{\tiny $2$} (ZK);
\draw[bend right=10,->] (ZR) to node[fill=white, inner sep=1pt]{\tiny $2$} (ZK);

				
\draw[->] (AK) to node[above] {$\mylittlematrix{4 & 2 & 2 & 2 & 1}$} (ZK); 
\draw[bend left=30,->] (AR) to node[above] {$\mylittlematrix{2 & 1}$} (ZL);
\draw[->] (AR) to node[fill=white, inner sep=1pt] {$\mylittlematrix{2 & 1}$} (ZL);
\draw[bend right=30,->] (AR) to node[below] {$\mylittlematrix{2 & 1}$} (ZL);
\draw[->] (Ae) to node[below] {\tiny $1$} (Ze); 

\end{tikzpicture}
}
\end{center}
Explicitly, $\underline{I}$ is the Mackey functor
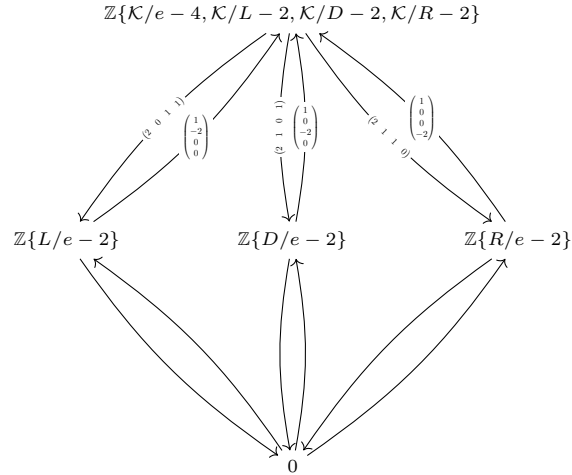
\begin{center}
{ 
\begin{tikzpicture}[scale=1]

				
\node (IK) at (0,6) {\scriptsize $\Z\{\K/e - 4, \K/L-2, \K/D-2, \K/R-2\}$};
\node (IL) at (-3,3) {\scriptsize $\Z\{L/e - 2\}$};
\node (ID) at (0,3) {\scriptsize $\Z\{D/e - 2\}$};
\node (IR) at (3,3) {\scriptsize $\Z\{R/e - 2\}$};
\node (Ie) at (0,0) {\scriptsize $0$};

				
\draw[bend right=10, ->] (IL) to (Ie);
\draw[bend right=10, ->] (ID) to (Ie);
\draw[bend right=10, ->] (IR) to (Ie);

	
\draw[bend right=10, ->] (Ie) to (IL);
\draw[bend right=10, ->] (Ie) to (ID);
\draw[bend right=10, ->] (Ie) to (IR);

				
\draw[bend right=10,->] (IK) to node[fill=white, inner sep=1pt, rotate=45] {\mytinymatrix{2 & 0 & 1 & 1}} (IL);
\draw[bend right=10,->] (IK) to node[fill=white, inner sep=1pt, rotate=90] {\mytinymatrix{2 & 1 & 0 & 1}} (ID);
\draw[bend right=10,->] (IK) to node[fill=white, inner sep=1pt, rotate=-45] {\mytinymatrix{2 & 1 & 1 & 0}} (IR);

				
\draw[bend right=10,->] (IL) to node[fill=white, inner sep=1pt] {$\mytinymatrix{1 \\ -2 \\ 0 \\ 0}$} (IK);
\draw[bend right=10,->] (ID) to node[fill=white, inner sep=1pt]{$\mytinymatrix{1 \\ 0 \\ -2 \\ 0}$} (IK);
\draw[bend right=10,->] (IR) to node[fill=white, inner sep=1pt]{$\mytinymatrix{1 \\ 0 \\ 0 \\ -2}$} (IK);

\end{tikzpicture}
}
\end{center}
Notice that $\underline{I}$ vanishes at the bottom! This will play a large part in the computation. 

From our short exact sequence of Mackey functors, we get a cofiber sequence 
\[ H\underline{I} \rightarrow H\underline{A} \rightarrow H\underline{\Z}. \]
We'll compute the coefficients of $\Sigma^{k\overline{\rho}} H\underline{I}$ in several steps, and this will allow us to deduce the coefficients of $\Sigma^{k\overline{\rho}} H\underline{A}$ in a large range quite easily.

\begin{remark}\label{RhoSphereCWStructure}
$S^{k\overline{\rho}}$ has a $\K$-CW structure in which the $j$-cells are all free cells for $j \geq k+1$. This is precisely the product $\K$-CW structure on $S^{k\overline{\rho}} \simeq S^{k\sigma_L} \wedge S^{k\sigma_D} \wedge S^{k\sigma_R}$. The key ingredient in this fact is that $\K/H \times \K/J \cong \K/e$  for $H \neq J \leq \K$ with $|H| = |J| = 2$.  
\end{remark}

\begin{lemma}
Let $n, k \geq 0$. Then, $\underline{\pi}_n^\K(S^{k\overline{\rho}} \wedge H\underline{I}) \cong \underline{0}$ for $n \geq k+1$. 
\end{lemma}

\begin{proof}
By \cref{RhoSphereCWStructure}, $S^{k\overline{\rho}}$ has exclusively free cells $n$-cells for $n \geq k+1$. Further, $\underline{A}_{\K/e} \boxtimes \underline{I} \cong \underline{0}$ since $\underline{I}$ vanishes at the underlying level. By definition, the $n^{th}$ chain Mackey functor vanishes.
\end{proof}

\begin{corollary}\label{PosConeComparison}
	For $n \geq k+2$, the long exact sequence in homotopy associated to the fiber sequence $H\underline{I} \rightarrow H\underline{A } \rightarrow H\underline{\Z}$ gives an isomorphism $\underline{\pi}_n^\K(\Sigma^{k\overline{\rho}} H\underline{A}) \cong \underline{\pi}_n^\K(\Sigma^{k\overline{\rho}} H\underline{\Z})$. 
\end{corollary}

\begin{proof}
	Applying $\underline{\pi}_n^\K(-)$ to the cofiber sequence
	\[ \Sigma^{k\overline{\rho}} H\underline{I} \rightarrow \Sigma^{k\overline{\rho}} H\underline{A} \rightarrow \Sigma^{k\overline{\rho}} H\underline{\Z} \]
	for $n \geq k+2$ gives 
	\[ \underline{0} \rightarrow \underline{\pi}_n^\K(\Sigma^{k\overline{\rho}} H\underline{A}) \rightarrow \underline{\pi}_n^\K(\Sigma^{k\overline{\rho}} H\underline{\Z}) \rightarrow \underline{0}. \]
\end{proof}

From here, we will go ahead and compute $\underline{\pi}_n^\K(\Sigma^{k\overline{\rho}} H\underline{A})$ for $n \leq k+1$. 

\begin{prop}\label{IPosConeProp1}
	The non-zero homotopy of the $\overline{\rho}$-suspension of $H\underline{I}$ is given by 
	\[
	\underline{\pi}_n^\K(\Sigma^{\overline{\rho}} H\underline{I})
	\begin{cases}
	\langle \Z \oplus \F_2^2 \rangle & n = 0 \\
	\phi_{LDR} \underline{f} & n = 1.
	\end{cases}
	\]
\end{prop}

\begin{proof}
	Analyzing the long exact sequence in homotopy Mackey functors given by the cofiber sequence 
	\[ K/L_+ \wedge H\underline{I} \rightarrow H\underline{I} \rightarrow \Sigma^{\sigma_L} H\underline{I} \]
gives the homotopy of $\Sigma^{\sigma_L} H\underline{I}$. We understand the integer graded homotopy of $H\underline{I}$ and by \cref{boxformula} and \cref{topologicalboxformula} we understand the left-most term. Using this as input for the long exact sequence in homotopy given by the cofiber sequence  
	\[ \K/D_+ \wedge \Sigma^{\sigma_L} H\underline{I} \rightarrow \Sigma^{\sigma_L} H\underline{I} \rightarrow \Sigma^{\sigma_L+\sigma_D} H\underline{I} \]
	gives the homotopy of $\Sigma^{\sigma_L + \sigma_D} H\underline{I}$. Lastly, using this as input for the long exact sequence in homotopy given by the cofiber sequence 
	\[ \K/R_+ \wedge \Sigma^{\sigma_L + \sigma_D} H\underline{I} \rightarrow \Sigma^{\sigma_L + \sigma_D} H\underline{I} \rightarrow \Sigma^{\overline{\rho}} H\underline{I} \]
	gives the homotopy of $\Sigma^{\overline{\rho}} H\underline{I}$. 
\end{proof}

With the homotopy of the first $\overline{\rho}$-suspension, we'll use a spectral sequence to compute the higher $\overline{\rho}$-suspensions. Recall that for $n \geq 0$, $\underline{\pi}_n(\Sigma^{2\overline{\rho}} H\underline{I})$ is isomorphic to $\underline{\tilde{H}}_n^\K(S^{2\overline{\rho}}; \underline{I})$, the ordinary equivariant homology of $S^{2\overline{\rho}}$ with coefficients in $\underline{I}$. We'll compute this via a cellular approach. Since $S^{2\overline{\rho}} \cong S^{\overline{\rho}} \wedge S^{\overline{\rho}}$, we can build the reduced cellular chain complex of the two-fold suspension as the total complex of the double complex $\underline{D}_{i,j} = (\widetilde{\underline{C}}_i(S^{\overline{\rho}}) \boxtimes \underline{I}) \boxtimes \widetilde{\underline{C}}_j(S^{\overline{\rho}})$. If we take homology of $\underline{D}_{i,j}$ with respect to the horizontal differentials, we obtain the $E^1$-page of a spectral sequence computing the reduced Bredon homology of $S^{2\overline{\rho}}$ with coefficients in $\underline{I}$. 

To get a more convenient description, note that for $j > 0$ we have 
\begin{eqnarray*}
\underline{D}_{i,j} &\cong & \widetilde{\underline{C}}_i(S^{\overline{\rho}}) \boxtimes \underline{I} \boxtimes \left( \bigoplus_{H_m} \underline{A}_{\K/H_m} \right) \\
&\cong & \bigoplus_{H_m} \left( \widetilde{\underline{C}}_i(S^{\overline{\rho}}) \boxtimes \underline{I}\right) \boxtimes \uparrow_{H_m}^\K \downarrow_{H_m}^\K \underline{A} \\
&\cong & \bigoplus_{H_m} \uparrow_{H_m}^\K \downarrow_{H_m}^\K (\widetilde{\underline{C}}_i(S^{\overline{\rho}}) \boxtimes \underline{I}), 
\end{eqnarray*}
where $H_m$ ranges through the stabilizers that show up in the collection of $j$-cells in $S^{\overline{\rho}}$. Since both induction and restriction are exact, they preserve homology with respect to the horizontal differentials. We get the following description of the $E^1$-page computing the homology of $S^{2\overline{\rho}}$:  
\[
E^1_{i,j} \cong \bigoplus_{H_m} \uparrow_{H_m}^\K \downarrow_{H_m}^\K \widetilde{\underline{H}}_i^\K(S^{\overline{\rho}};\underline{I}).
\]
The full page is depicted in \cref{E1Page}. 
\begin{figure}
\caption{$E^1_{i,j}-\text{page}$}
\label{E1Page}
	\begin{tikzpicture}[scale=1.5,font=\tiny]
	
            	
	\draw[->] (0,-0.25) to (0,2.5); 
	\draw[->] (-0.25,0) to (2.5,0);
	
	\node at (2.75,0) {\large $i$};
	\node at (0,2.75) {\large $j$};
	
			
	\node[below left={0.35ex}, draw, fill=white, inner sep=2pt, circle] at (0,0) {};
	\node[above right={0.35ex}, draw, fill=black, inner sep=2pt, circle, text=white] at (0,0) {$2$};
	
	\node[fill=white] (10) at (1,0) {\large $\phiLDRf$};
	
	\node (11) at (1,1) {\large $\widehat{\whitecirc}$}; 
	
	\node[fill=white] at (0,1) {\large $\underline{0}$};
	\node[fill=white] at (0,2) {\large $\underline{0}$};
	
	\node at (1,2) {\large $\underline{0}$};		
	
	\node[fill=white] at (2,0) {\large $\underline{0}$};
	\node at (2,1) {\large $\underline{0}$};
	\node at (2,2) {\large $\underline{0}$};
		
			
	\draw[orddiff, thick, ->] (11) to (10) {};
	
	\end{tikzpicture} \\
	Key: $\whitecirc = \langle \Z \rangle$; \quad $\phiLDRf = \phi_{LDR}^* \underline{f}$; \quad $\widehat{\whitecirc} = \bigoplus_H \uparrow_H^\K \langle \Z \rangle$; \quad $\begin{tikzpicture} \node[draw, fill=black, text=white, circle, inner sep=0.5pt] at (0,0) {\tiny $2$}; \end{tikzpicture} = \langle \F_2 \rangle^{\oplus 2}$
\end{figure}
Now, the only differential in this $E^1$-page vanishes at $\K/\K$. For all of the other levels, the maps are forced since we know the $C_2$-equivariant homotopy of $\downarrow_{H}^\K S^{\overline{\rho}} \cong S^{1 + 2\sigma}$ for any $\{e\} < H < \K$. The differential is as follows: 
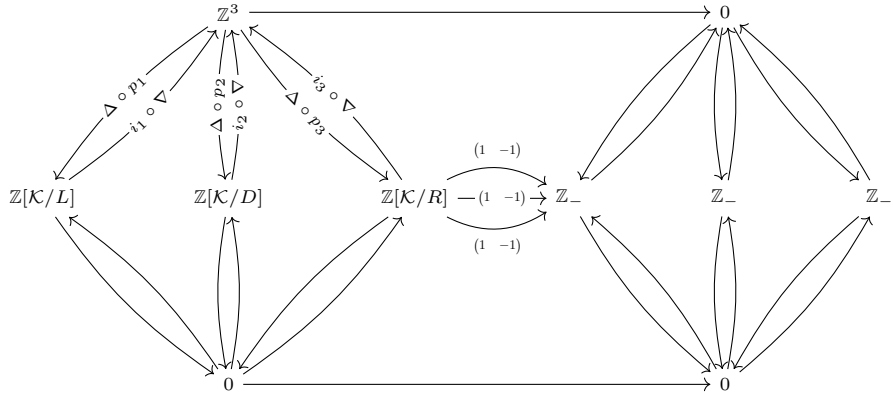
\begin{center}
{ 
\begin{tikzpicture}[scale=0.825]

				
\node (SK) at (-4,6) {\scriptsize $\Z^3$};
\node (SL) at (-7,3) {\scriptsize $\Z[\K/L]$};
\node (SD) at (-4,3) {\scriptsize $\Z[\K/D]$};
\node (SR) at (-1,3) {\scriptsize $\Z[\K/R]$};
\node (Se) at (-4,0) {\scriptsize $0$};

				
\draw[bend right=10,->] (SL) to (Se);
\draw[bend right=10,->] (SD) to (Se);
\draw[bend right=10,->] (SR) to (Se);

				
\draw[bend right=10,->] (Se) to (SL);
\draw[bend right=10,->] (Se) to (SD);
\draw[bend right=10,->] (Se) to (SR);

				
\draw[bend right=10,->] (SK) to node[rotate=45, fill=white, inner sep=1pt] {\tiny $\Delta \circ p_1$} (SL);
\draw[bend right=10,->] (SK) to node[rotate=90, fill=white, inner sep=1pt] {\tiny $\Delta \circ p_2$} (SD);
\draw[bend right=10,->] (SK) to node[rotate=-45, fill=white, inner sep=1pt] {\tiny $\Delta \circ p_3$} (SR);

				
\draw[bend right=10,->] (SL) to node[rotate=45, fill=white, inner sep=1pt] {\tiny $i_1 \circ \nabla$} (SK);
\draw[bend right=10,->] (SD) to node[rotate=90, fill=white, inner sep=1pt] {\tiny $i_2 \circ \nabla$} (SK);
\draw[bend right=10,->] (SR) to node[rotate=-45, fill=white, inner sep=1pt] {\tiny $i_3 \circ \nabla$} (SK);

				
\node (TK) at (4,6) {\scriptsize $0$};
\node (TL) at (1.5,3) {\scriptsize $\Z_-$};
\node (TD) at (4,3) {\scriptsize $\Z_-$};
\node (TR) at (6.5,3) {\scriptsize $\Z_-$};
\node (Te) at (4,0) {\scriptsize $0$};

				
\draw[bend right=10,->] (TL) to (Te);
\draw[bend right=10,->] (TD) to (Te);
\draw[bend right=10,->] (TR) to (Te);

				
\draw[bend right=10,->] (Te) to (TL);
\draw[bend right=10,->] (Te) to (TD);
\draw[bend right=10,->] (Te) to (TR);

				
\draw[bend right=10,->] (TK) to (TL);
\draw[bend right=10,->] (TK) to (TD);
\draw[bend right=10,->] (TK) to (TR);

				
\draw[bend right=10,->] (TL) to (TK);
\draw[bend right=10,->] (TD) to (TK);
\draw[bend right=10,->] (TR) to (TK);

				
\draw[->] (SK) to (TK); 
\draw[bend left=30,->] (SR) to node[above] {$\mylittlematrix{1 & -1}$} (TL);
\draw[->] (SR) to node[fill=white, inner sep=1pt] {$\mylittlematrix{1 & -1}$} (TL);
\draw[bend right=30,->] (SR) to node[below] {$\mylittlematrix{1 & -1}$} (TL);
\draw[->] (Se) to (Te); 

\end{tikzpicture}
}
\end{center}
This allows us to compute the second page: 
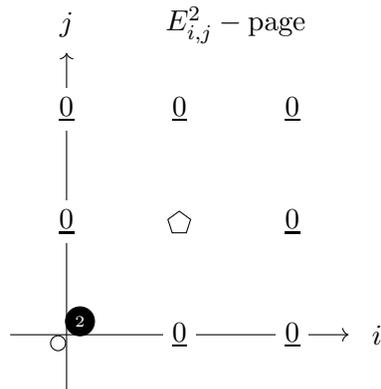
\begin{center}
	\begin{tikzpicture}[scale=1.5,font=\tiny]
	
            	
	\draw[->] (0,-0.5) to (0,2.5); 
	\draw[->] (-0.5,0) to (2.5,0);
	
	\node at (2.75,0) {\large $i$};
	\node at (0,2.75) {\large $j$};
	
	\node at (1.5,2.75) {\large $E^2_{i,j}-\text{page}$};
	
			
	\node[below left={0.35ex}, draw, fill=white, inner sep=2pt, circle] at (0,0) {};
	\node[above right={0.35ex}, draw, fill=black, inner sep=2pt, circle, text=white] at (0,0) {$2$};
	
	\node (11) at (1,1) {\large $\whitepent$}; 
	
	\node[fill=white] at (0,1) {\large $\underline{0}$};
	\node[fill=white] at (0,2) {\large $\underline{0}$};
	
	\node[fill=white] at (1,0) {\large $\underline{0}$};
	\node at (1,2) {\large $\underline{0}$};		
	
	\node[fill=white] at (2,0) {\large $\underline{0}$};
	\node at (2,1) {\large $\underline{0}$};
	\node at (2,2) {\large $\underline{0}$};
	\end{tikzpicture} \\
	Key: $\whitecirc = \langle \Z \rangle$; \quad $\whitepent = \phi_{LDR}^* \underline{\Z}$; \quad $\begin{tikzpicture} \node[draw, fill=black, text=white, circle, inner sep=0.5pt] at (0,0) {\tiny $2$}; \end{tikzpicture} = \langle \F_2 \rangle^{\oplus 2}$
\end{center}
There's no room for differentials possible on this page, so the spectral sequence collapses. We're able to read the homotopy off as follows: 
\begin{prop}\label{IPosConeProp2}
The non-zero homotopy of the $2\overline{\rho}$-suspension of $H\underline{I}$ is as follows: 
\[
\underline{\pi}_n^\K(\Sigma^{2\overline{\rho}}H\underline{I}) \cong
\begin{cases}
	\langle \Z \oplus \F_2^2 \rangle & n=0 \\
	\phi_{LDR}^* \underline{\Z} & n=2. 
\end{cases}
\]
\end{prop}

To learn about the next suspension, we'll repeat this process. That is, we view $S^{3\overline{\rho}} \simeq S^{2\overline{\rho}} \wedge S^{\overline{\rho}}$ and set up the corresponding double complex. On the $E^1$-page, we see $\underline{\pi}_n(\Sigma^{2\overline{\rho}} H\underline{I})$ along the 0th row. The full page is as follows: 

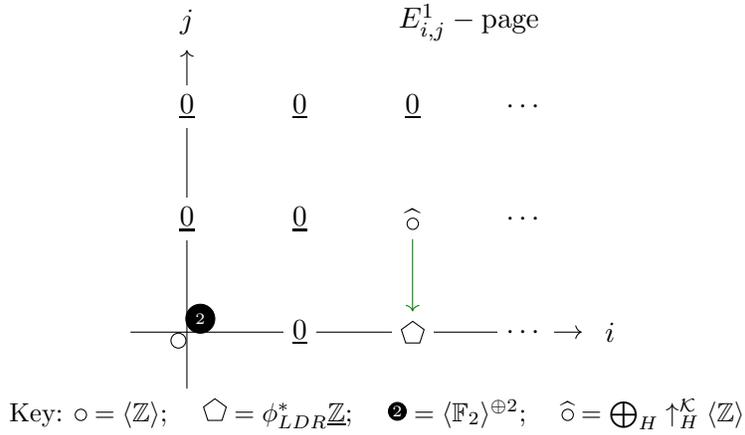
\begin{center}
	\begin{tikzpicture}[scale=1.5,font=\tiny]
	
            	 
	\draw[->] (0,-0.5) to (0,2.5); 
	\draw[->] (-0.5,0) to (3.5,0);
	
	\node at (3.75,0) {\large $i$};
	\node at (0,2.75) {\large $j$};
	
	\node at (2.5,2.75) {\large $E^1_{i,j}-\text{page}$};
	
			
	\node[below left={0.35ex}, draw, fill=white, inner sep=2pt, circle] at (0,0) {};
	\node[above right={0.35ex}, draw, fill=black, inner sep=2pt, circle, text=white] at (0,0) {$2$};
	
	\node[fill=white] (20) at (2,0) {\large $\whitepent$}; 
	\node (21) at (2,1) {\large $\widehat{\whitecirc}$};
	
	\node[fill=white] at (0,1) {\large $\underline{0}$};
	\node[fill=white] at (0,2) {\large $\underline{0}$};
	
	\node[fill=white] at (1,0) {\large $\underline{0}$};
	\node at (1,1) {\large $\underline{0}$};	
	\node at (1,2) {\large $\underline{0}$};		
	
	\node at (2,2) {\large $\underline{0}$};
	
	\node[fill=white] at (3,0) {\large $\cdots$};
	
	\foreach \n in {1,...,2}{ 
		\node at (3,\n) {\large $\cdots$};
	}
	
			
	\draw[orddiff, ->] (21) to (20);
	\end{tikzpicture} \\
	Key: $\whitecirc = \langle \Z \rangle$; \quad $\whitepent = \phi_{LDR}^* \underline{\Z}$; \quad $\begin{tikzpicture} \node[draw, fill=black, text=white, circle, inner sep=0.5pt] at (0,0) {\tiny $2$}; \end{tikzpicture} = \langle \F_2 \rangle^{\oplus 2}$; \quad $\widehat{\whitecirc} = \bigoplus_H \uparrow_H^\K \langle \Z \rangle$
	
\end{center}
For the same reasoning, we can understand the differential here: 
\begin{center}
{ 
\begin{tikzpicture}[scale=0.825]

				
\node (SK) at (-4,6) {\scriptsize $\Z^3$};
\node (SL) at (-7,3) {\scriptsize $\Z[\K/L]$};
\node (SD) at (-4,3) {\scriptsize $\Z[\K/D]$};
\node (SR) at (-1,3) {\scriptsize $\Z[\K/R]$};
\node (Se) at (-4,0) {\scriptsize $0$};

				
\draw[bend right=10,->] (SL) to (Se);
\draw[bend right=10,->] (SD) to (Se);
\draw[bend right=10,->] (SR) to (Se);

				
\draw[bend right=10,->] (Se) to (SL);
\draw[bend right=10,->] (Se) to (SD);
\draw[bend right=10,->] (Se) to (SR);

				
\draw[bend right=10,->] (SK) to node[rotate=45, fill=white, inner sep=1pt] {\tiny $\Delta \circ p_1$} (SL);
\draw[bend right=10,->] (SK) to node[rotate=90, fill=white, inner sep=1pt] {\tiny $\Delta \circ p_2$} (SD);
\draw[bend right=10,->] (SK) to node[rotate=-45, fill=white, inner sep=1pt] {\tiny $\Delta \circ p_3$} (SR);

				
\draw[bend right=10,->] (SL) to node[rotate=45, fill=white, inner sep=1pt] {\tiny $i_1 \circ \nabla$} (SK);
\draw[bend right=10,->] (SD) to node[rotate=90, fill=white, inner sep=1pt] {\tiny $i_2 \circ \nabla$} (SK);
\draw[bend right=10,->] (SR) to node[rotate=-45, fill=white, inner sep=1pt] {\tiny $i_3 \circ \nabla$} (SK);

				
\node (TK) at (4,6) {\scriptsize $\Z^3$};
\node (TL) at (1,3) {\scriptsize $\Z$};
\node (TD) at (4,3) {\scriptsize $\Z$};
\node (TR) at (7,3) {\scriptsize $\Z$};
\node (Te) at (4,0) {\scriptsize $0$};

				
\draw[bend right=10,->] (TL) to (Te);
\draw[bend right=10,->] (TD) to (Te);
\draw[bend right=10,->] (TR) to (Te);

				
\draw[bend right=10,->] (Te) to (TL);
\draw[bend right=10,->] (Te) to (TD);
\draw[bend right=10,->] (Te) to (TR);

				
\draw[bend right=10,->] (TK) to node[rotate=45, fill=white, inner sep=1pt] {\tiny $p_1$} (TL);
\draw[bend right=10,->] (TK) to node[rotate=90, fill=white, inner sep=1pt] {\tiny $p_2$} (TD);
\draw[bend right=10,->] (TK) to node[rotate=-45, fill=white, inner sep=1pt] {\tiny $p_3$} (TR);

				
\draw[bend right=10,->] (TL) to node[rotate=45, fill=white, inner sep=1pt] {\tiny $2i_1$} (TK);
\draw[bend right=10,->] (TD) to node[rotate=90, fill=white, inner sep=1pt] {\tiny $2i_2$} (TK);
\draw[bend right=10,->] (TR) to node[rotate=-45, fill=white, inner sep=1pt] {\tiny $2i_3$} (TK);

				
\draw[->] (SK) to node[above] {\tiny $2$} (TK); 
\draw[bend left=30,->] (SR) to node[above] {\tiny $\nabla$} (TL);
\draw[->] (SR) to node[fill=white, inner sep=1pt] {\tiny $\nabla$} (TL);
\draw[bend right=30,->] (SR) to node[below] {\tiny $\nabla$} (TL);
\draw[->] (Se) to (Te); 

\end{tikzpicture}
}
\end{center}
With this, we can compute the second page: 
\begin{center}
	\begin{tikzpicture}[scale=1.5,font=\tiny]
	
            	 
	\draw[->] (0,-0.5) to (0,2.5); 
	\draw[->] (-0.5,0) to (3.5,0);
	
	\node at (3.75,0) {\large $i$};
	\node at (0,2.75) {\large $j$};
	
	\node at (2.5,2.75) {\large $E^2_{i,j}-\text{page}$};
	
			
	\node[below left={0.35ex}, draw, fill=white, inner sep=2pt, circle] at (0,0) {};
	\node[above right={0.35ex}, draw, fill=black, inner sep=2pt, circle, text=white] at (0,0) {$2$};
	
	\node[draw, fill=black, inner sep=2pt, circle, text=white] (20) at (2,0) {$3$};
	\node[fill=white] (21) at (2,1) {\large $\phiLDRf$}; 
	
			
	\node[fill=white] at (0,1) {\large $\underline{0}$};
	\node[fill=white] at (0,2) {\large $\underline{0}$};
	
	\node[fill=white] at (1,0) {\large $\underline{0}$};
	\node at (1,1) {\large $\underline{0}$};	
	\node at (1,2) {\large $\underline{0}$};	
	
	\node at (2,2) {\large $\underline{0}$};
	
	\node[fill=white] at (3,0) {\large $\cdots$};
	
	\foreach \n in {1,...,2}{ 
		\node at (3,\n) {\large $\cdots$};
	}

	\end{tikzpicture} \\
	Key: $\whitecirc = \langle \Z \rangle$; \quad $\begin{tikzpicture} \node[draw, fill=black, text=white, circle, inner sep=0.5pt] at (0,0) {\tiny $n$}; \end{tikzpicture} = \langle \F_2 \rangle^{\oplus n}$; \quad $\phiLDRf = \phi_{LDR}^* \underline{f}$
\end{center}
Again, there is no room for differentials and the spectral sequence collapses. So, we can read the homotopy off of the $E^2 = E^\infty$-page. 
\begin{prop}\label{IPosConeProp3}
The non-zero homotopy of the $3\overline{\rho}$-suspension of $H\underline{I}$ is as follows: 
\[
\underline{\pi}_n^\K(\Sigma^{3\overline{\rho}}H\underline{I}) \cong
\begin{cases}
	\langle \Z \oplus \F_2^2 \rangle & n=0 \\
	\langle \F_2^3 \rangle & n=2 \\
	\phi_{LDR}^* \underline{f} & n=3. 
\end{cases}
\]
\end{prop}

We can continue in this fashion to compute $\underline{\pi}^\K_n(\Sigma^{k\overline{\rho}} H\underline{I})$ for $k \geq 4$. The only necessary thing to check is that the differentials repeat what we've seen already. 

\begin{theorem}\label{IPosConeThm}
	The non-zero homotopy of $\Sigma^{k\overline{\rho}} H\underline{I}$ for $k \geq 4$ is as follows:
	\[
	\underline{\pi}_n^\K(\Sigma^{k\overline{\rho}} H\underline{I}) \cong 
	\begin{cases}
	\langle \Z \oplus \F_2^2 \rangle & n=0 \\	
	\langle \F_2^3 \rangle & 1 < n < k; n \text{ even} \\
	\phi_{LDR} \underline{\Z} & n = k; \text{k even} \\
	\phi_{LDR} \underline{f} & n=k; \text{k odd}. 
	\end{cases}
	\]
\end{theorem}
The data of \cref{IPosConeThm} will be referred to as the ``positive cone" for $H\underline{I}$ and is displayed in the fourth quadrant of \cref{fig:piK4HI}. We'll see that $\Sigma^{k\overline{\rho}} H\underline{A}$ and $\Sigma^{k\overline{\rho}} H\underline{\Z}$ agree in a very slightly larger range ($n \geq k+1$), but this method can't see this as $H\underline{I}$ doesn't vanish in the appropriate range. In fact, the connecting morphism in the long exact sequence vanishes. 

\section{A Comparison of the Negative Cones for $H\underline{A}$ and $H\underline{\Z}$} \label{ComparisonNegative}

Our focus in this section will be to compute the ``negative cone" of $H\underline{A}$, i.e. $\underline{\pi}_n(\Sigma^{-k\overline{\rho}} H\underline{A})$ for $k \geq 0$ and $n \leq 0$. Like the last section, we'll give a comparison to the negative cone for $H\underline{\Z}$ via the cofiber sequence 
\[ \Sigma^{-k\overline{\rho}} H\underline{I} \rightarrow \Sigma^{-k\overline{\rho}} H\underline{A} \rightarrow \Sigma^{-k\overline{\rho}} H\underline{\Z}. \]

\begin{lemma}\label{NegConeComparison}
The long exact sequence associated to the fiber sequence $H\underline{I} \rightarrow H\underline{A} \rightarrow H\underline{\Z}$ gives an isomorphism 
\[ \underline{\pi}_n^\K(\Sigma^{-k\overline{\rho}} H\underline{A}) \cong \underline{\pi}_n^\K(\Sigma^{-k\overline{\rho}} H\underline{\Z}) \]
for $n \leq -(k+2)$. 
\end{lemma}

\begin{proof}
Similar to the proof of \cref{PosConeComparison}, the $n^{th}$ cochain Mackey functor will vanish for $n \leq -(k+2)$.
\end{proof}

For computing the homotopy in low degrees, we start with a $-\overline{\rho}$-suspension as in the last section:

\begin{prop}
	The non-zero homotopy of the $-\overline{\rho}$-suspension of $H\underline{I}$ is given by 
	\[
	\underline{\pi}_n^\K(\Sigma^{-\overline{\rho}} H\underline{I}) \cong 
	\begin{cases}
		\langle \Z \rangle & n=0 \\
		\underline{E} & n=-1.
	\end{cases}
	\]
\end{prop}

\begin{proof}
	The situation is dual to \cref{IPosConeProp1} in that we are computing ordinary equivariant cohomology as opposed to homology. So, we use the dual cofiber sequence 
	\[ S^{-\sigma_L} \rightarrow \mathbb{S} \rightarrow \K/L_+. \]
Smashing with $H\underline{I}$ gives another cofiber sequence: 
	\[ \Sigma^{-\sigma_L} H\underline{I} \rightarrow H\underline{I} \rightarrow \K/L_+ \wedge H\underline{I}, \]
	and this is enough to compute the integer graded homotopy of $\Sigma^{-\sigma_L} H\underline{I}$. Using this as input, analyzing the long exact sequence given by the cofiber sequence 
	\[ \Sigma^{-\sigma_L - \sigma_D} H\underline{I} \rightarrow \Sigma^{-\sigma_L} H\underline{I} \rightarrow \K/D_+ \wedge \Sigma^{-\sigma_L} H\underline{I} \]
	gives the homotopy of $\Sigma^{-\sigma_L-\sigma_D} H\underline{I}$. Lastly, we input this into the cofiber sequence 
	\[ \Sigma^{-\overline{\rho}} H\underline{I} \rightarrow \Sigma^{-\sigma_L-\sigma_D} H\underline{I} \rightarrow \K/R_+ \wedge \Sigma^{-\sigma_L-\sigma_D} H\underline{I}. \]
	This gives the homotopy of the $-\overline{\rho}$-suspension of $H\underline{I}$. 
\end{proof}

Now that we've recorded this answer, we begin a series of spectral sequences used in the same manner as to compute the positive cone. In particular, we'll see the cohomology of $S^{\overline{\rho}}$ in the first row and differentials going up. The $E_1$-page of the spectral sequence computing the homotopy of the $-2\overline{\rho}$-suspension of $H\underline{I}$ is as follows: 

\begin{center}
	\begin{tikzpicture}[scale=1.5,font=\tiny]
	
            	
	\draw[->] (0,-0.5) to (0,2.5); 
	\draw[->] (-0.5,0) to (2.5,0);
	
	\node at (2.75,0) {\large $i$};
	\node at (0,2.75) {\large $j$};
	
	\node at (1.5,2.75) {\large $E^1_{i,j}-\text{page}$};
	
			
	\node[draw, fill=white, inner sep=2pt, circle] at (0,0) {};
	
	\node[fill=white] (10) at (1,0) {$\bardot$};
	
	\node (11) at (1,1) {\large $\widehat{\whitecirc}$}; 
	
	\node[fill=white] at (0,1) {\large $\underline{0}$};
	\node[fill=white] at (0,2) {\large $\underline{0}$};
	
	\node at (1,2) {\large $\underline{0}$};
	
	\node[fill=white] at (2,0) {\large $\underline{0}$};
	\node at (2,1) {\large $\underline{0}$};
	\node at (2,2) {\large $\underline{0}$};
		
			
	\draw[orddiff, thick, ->] (10) to (11) {};
	\end{tikzpicture} \\
	Key: $\whitecirc = \langle \Z \rangle$; \quad $\bardot = \underline{E}$ \quad $\widehat{\whitecirc} = \bigoplus_H \uparrow_H^\K \langle \Z \rangle$
\end{center}
As with the positive cone, we'll see that the differential is forced by knowing the $C_2$-equivariant restriction of this spectral sequence. Explicitly, the differential is the following: 
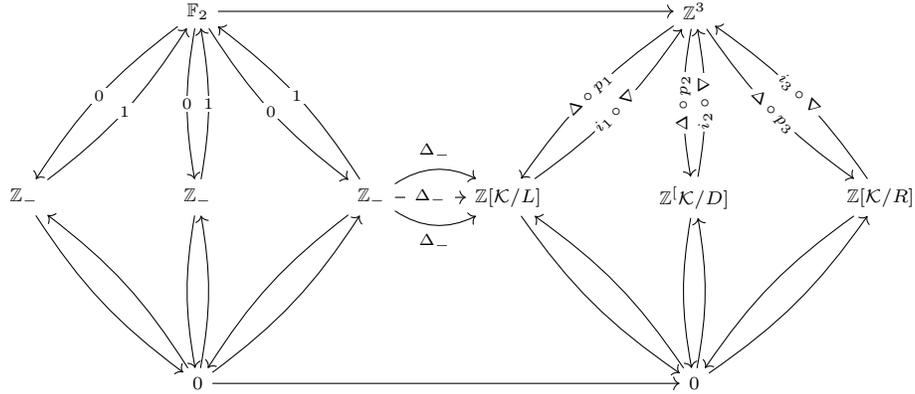
\begin{center}
{ 
\begin{tikzpicture}[scale=0.825]

				
\node (SK) at (-4,6) {\scriptsize $\F_2$};
\node (SL) at (-6.8,3) {\scriptsize $\Z_-$};
\node (SD) at (-4,3) {\scriptsize $\Z_-$};
\node (SR) at (-1.2,3) {\scriptsize $\Z_-$};
\node (Se) at (-4,0) {\scriptsize $0$};

				
\draw[bend right=10,->] (SL) to (Se);
\draw[bend right=10,->] (SD) to (Se);
\draw[bend right=10,->] (SR) to (Se);

				
\draw[bend right=10,->] (Se) to (SL);
\draw[bend right=10,->] (Se) to (SD);
\draw[bend right=10,->] (Se) to (SR);

				
\draw[bend right=10,->] (SK) to node[fill=white, inner sep=1pt] {\tiny $0$} (SL);
\draw[bend right=10,->] (SK) to node[fill=white, inner sep=1pt] {\tiny $0$} (SD);
\draw[bend right=10,->] (SK) to node[fill=white, inner sep=1pt] {\tiny $0$} (SR);

				
\draw[bend right=10,->] (SL) to node[fill=white, inner sep=1pt] {\tiny $1$} (SK);
\draw[bend right=10,->] (SD) to node[fill=white, inner sep=1pt] {\tiny $1$} (SK);
\draw[bend right=10,->] (SR) to node[fill=white, inner sep=1pt] {\tiny $1$} (SK);

				
\node (TK) at (4,6) {\scriptsize $\Z^3$};
\node (TL) at (1,3) {\scriptsize $\Z[\K/L]$};
\node (TD) at (4,3) {\scriptsize $\Z^[\K/D]$};
\node (TR) at (7,3) {\scriptsize $\Z[\K/R]$};
\node (Te) at (4,0) {\scriptsize $0$};

				
\draw[bend right=10,->] (TL) to (Te);
\draw[bend right=10,->] (TD) to (Te);
\draw[bend right=10,->] (TR) to (Te);

				
\draw[bend right=10,->] (Te) to (TL);
\draw[bend right=10,->] (Te) to (TD);
\draw[bend right=10,->] (Te) to (TR);

				
\draw[bend right=10,->] (TK) to node[rotate=45, fill=white, inner sep=1pt] {\tiny $\Delta \circ p_1$} (TL);
\draw[bend right=10,->] (TK) to node[rotate=90, fill=white, inner sep=1pt] {\tiny $\Delta \circ p_2$} (TD);
\draw[bend right=10,->] (TK) to node[rotate=-45, fill=white, inner sep=1pt] {\tiny $\Delta \circ p_3$} (TR);

				
\draw[bend right=10,->] (TL) to node[rotate=45, fill=white, inner sep=1pt] {\tiny $i_1 \circ \nabla$} (TK);
\draw[bend right=10,->] (TD) to node[rotate=90, fill=white, inner sep=1pt] {\tiny $i_2 \circ \nabla$} (TK);
\draw[bend right=10,->] (TR) to node[rotate=-45, fill=white, inner sep=1pt] {\tiny $i_3 \circ \nabla$} (TK);

				
\draw[->] (SK) to (TK); 
\draw[bend left=30,->] (SR) to node[above] {\tiny $\Delta_-$} (TL);
\draw[->] (SR) to node[fill=white] {\tiny $\Delta_-$} (TL);
\draw[bend right=30,->] (SR) to node[below] {\tiny $\Delta_-$} (TL);
\draw[->] (Se) to (Te); 

\end{tikzpicture}
}
\end{center}
This allows us to compute the second page of the spectral sequence: 
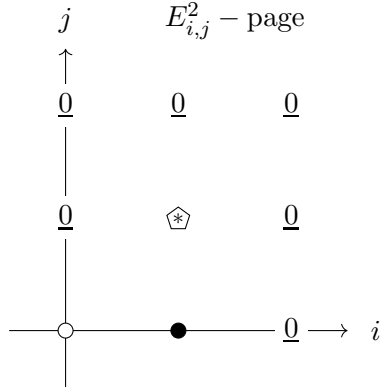
\begin{center}
	\begin{tikzpicture}[scale=1.5,font=\tiny]
	
            	
	\draw[->] (0,-0.5) to (0,2.5); 
	\draw[->] (-0.5,0) to (2.5,0);
	
	\node at (2.75,0) {\large $i$};
	\node at (0,2.75) {\large $j$};
	
	\node at (1.5,2.75) {\large $E^2_{i,j}-\text{page}$};
	
			
	\node[draw, fill=white, inner sep=2pt, circle] at (0,0) {};
	\node[draw, circle, fill=black, inner sep=2pt] at (1,0) {};
	
	\node (11) at (1,1) {\large $\whitepentdual$}; 
	
	\node[fill=white] at (0,1) {\large $\underline{0}$};
	\node[fill=white] at (0,2) {\large $\underline{0}$};
	
	\node at (1,2) {\large $\underline{0}$};	
	
	\node[fill=white] at (2,0) {\large $\underline{0}$};
	\node at (2,1) {\large $\underline{0}$};
	\node at (2,2) {\large $\underline{0}$};
		
	\end{tikzpicture} \\
	Key: $\whitecirc = \langle \Z \rangle$; \quad $\fillcirc = \langle \F_2 \rangle$; \quad $\whitepentdual = \phi_{LDR}^* \underline{\Z}^*$
\end{center}
Since there isn't room for any differentials, the spectral sequence collapses. 
\begin{prop}
The non-zero homotopy of the $-2\overline{\rho}$-suspension of $H\underline{I}$ is as follows:
	\[
	\underline{\pi}_n^\K(\Sigma^{-2\overline{\rho}} H\underline{I}) \cong
	\begin{cases}
		\langle \Z \rangle & n=0 \\
		\langle \F_2 \rangle & n=-1 \\
		\phi_{LDR}^* \underline{\Z}^* & n=-2. 
	\end{cases}
	\]
\end{prop}

Using this as input, we obtain the following $E_1$-page of a spectral sequence computing the homotopy of the $-3\overline{\rho}$-suspension of $H\underline{I}$: 

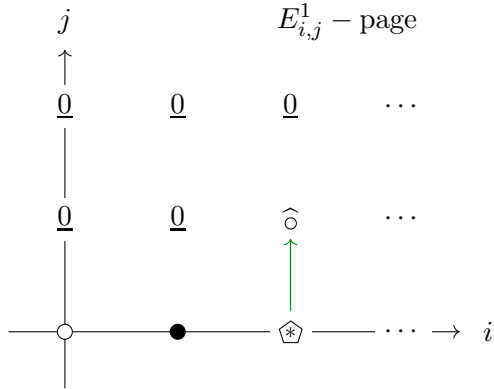
\begin{center}
	\begin{tikzpicture}[scale=1.5,font=\tiny]
	
            	 
	\draw[->] (0,-0.5) to (0,2.5); 
	\draw[->] (-0.5,0) to (3.5,0);
	
	\node at (3.75,0) {\large $i$};
	\node at (0,2.75) {\large $j$};
	
	\node at (2.5,2.75) {\large $E^1_{i,j}-\text{page}$};
	
			
	\node[draw, fill=white, inner sep=2pt, circle] at (0,0) {};
	\node[draw, fill=black, inner sep=2pt, circle] at (1,0) {};
	\node[fill=white] (20) at (2,0) {\large $\whitepentdual$}; 
	\node (21) at (2,1) {\large $\widehat{\whitecirc}$};
	
	\node[fill=white] at (0,1) {\large $\underline{0}$};
	\node[fill=white] at (0,2) {\large $\underline{0}$};
	
	\node at (1,1) {\large $\underline{0}$};	
	\node at (1,2) {\large $\underline{0}$};	
	
	\node at (2,2) {\large $\underline{0}$};
	
	\node[fill=white] at (3,0) {\large $\cdots$};
	
	\foreach \n in {1,...,2}{ 
		\node at (3,\n) {\large $\cdots$};
	}
	
			
	\draw[orddiff, ->] (20) to (21);
	\end{tikzpicture} \\
	Key: $\whitecirc = \langle \Z \rangle$; \quad $\fillcirc = \langle \F_2 \rangle$; \quad $\whitepentdual = \phi_{LDR}^* \underline{\Z}$; \quad $\widehat{\whitecirc} = \bigoplus_H \uparrow_H^\K \langle \Z \rangle$
\end{center}
For the same reasoning, we can understand the differential here: 
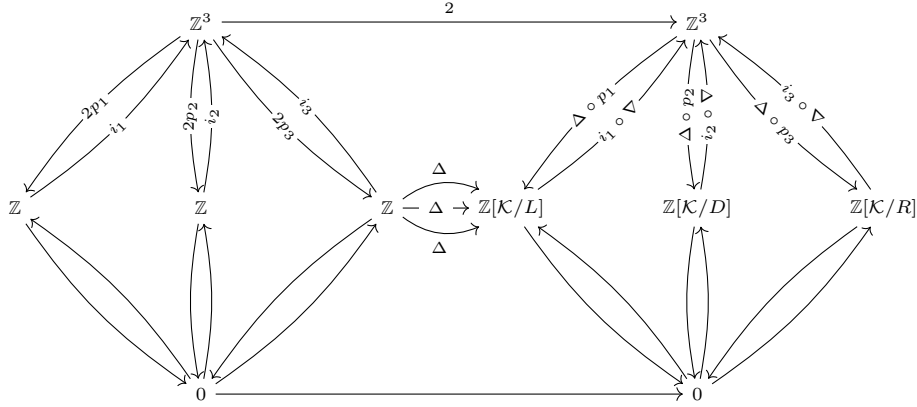
\begin{center}
{ 
\begin{tikzpicture}[scale=0.825]

				
\node (SK) at (-4,6) {\scriptsize $\Z^3$};
\node (SL) at (-7,3) {\scriptsize $\Z$};
\node (SD) at (-4,3) {\scriptsize $\Z$};
\node (SR) at (-1,3) {\scriptsize $\Z$};
\node (Se) at (-4,0) {\scriptsize $0$};

				
\draw[bend right=10,->] (SL) to (Se);
\draw[bend right=10,->] (SD) to (Se);
\draw[bend right=10,->] (SR) to (Se);

				
\draw[bend right=10,->] (Se) to (SL);
\draw[bend right=10,->] (Se) to (SD);
\draw[bend right=10,->] (Se) to (SR);

				
\draw[bend right=10,->] (SK) to node[rotate=45, fill=white, inner sep=1pt] {\tiny $2p_1$} (SL);
\draw[bend right=10,->] (SK) to node[rotate=90, fill=white, inner sep=1pt] {\tiny $2p_2$} (SD);
\draw[bend right=10,->] (SK) to node[rotate=-45, fill=white, inner sep=1pt] {\tiny $2p_3$} (SR);

				
\draw[bend right=10,->] (SL) to node[rotate=45, fill=white, inner sep=1pt] {\tiny $i_1$} (SK);
\draw[bend right=10,->] (SD) to node[rotate=90, fill=white, inner sep=1pt] {\tiny $i_2$} (SK);
\draw[bend right=10,->] (SR) to node[rotate=-45, fill=white, inner sep=1pt] {\tiny $i_3$} (SK);

				
\node (TK) at (4,6) {\scriptsize $\Z^3$};
\node (TL) at (1,3) {\scriptsize $\Z[\K/L]$};
\node (TD) at (4,3) {\scriptsize $\Z[\K/D]$};
\node (TR) at (7,3) {\scriptsize $\Z[\K/R]$};
\node (Te) at (4,0) {\scriptsize $0$};

				
\draw[bend right=10,->] (TL) to (Te);
\draw[bend right=10,->] (TD) to (Te);
\draw[bend right=10,->] (TR) to (Te);

				
\draw[bend right=10,->] (Te) to (TL);
\draw[bend right=10,->] (Te) to (TD);
\draw[bend right=10,->] (Te) to (TR);

				
\draw[bend right=10,->] (TK) to node[rotate=45, fill=white, inner sep=1pt] {\tiny $\Delta \circ p_1$} (TL);
\draw[bend right=10,->] (TK) to node[rotate=90, fill=white, inner sep=1pt] {\tiny $\Delta \circ p_2$} (TD);
\draw[bend right=10,->] (TK) to node[rotate=-45, fill=white, inner sep=1pt] {\tiny $\Delta \circ p_3$} (TR);

				
\draw[bend right=10,->] (TL) to node[rotate=45, fill=white, inner sep=1pt] {\tiny $i_1 \circ \nabla$} (TK);
\draw[bend right=10,->] (TD) to node[rotate=90, fill=white, inner sep=1pt] {\tiny $i_2 \circ \nabla$} (TK);
\draw[bend right=10,->] (TR) to node[rotate=-45, fill=white, inner sep=1pt] {\tiny $i_3 \circ \nabla$} (TK);

				
\draw[->] (SK) to node[above] {\tiny $2$} (TK); 
\draw[bend left=30,->] (SR) to node[above] {\tiny $\Delta$} (TL);
\draw[->] (SR) to node[fill=white] {\tiny $\Delta$} (TL);
\draw[bend right=30,->] (SR) to node[below] {\tiny $\Delta$} (TL);
\draw[->] (Se) to (Te); 

\end{tikzpicture}
}
\end{center}
With this, we can compute the second page: 
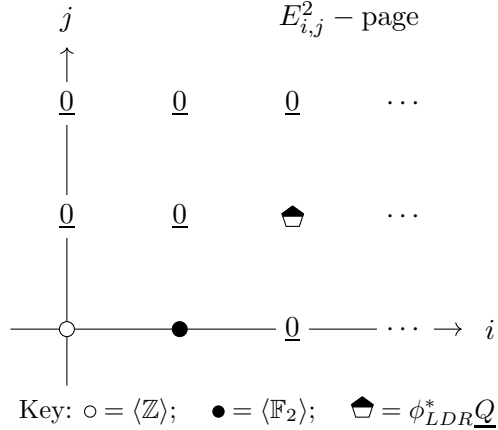
\begin{center}
	\begin{tikzpicture}[scale=1.5,font=\tiny]
	
            	 
	\draw[->] (0,-0.5) to (0,2.5); 
	\draw[->] (-0.5,0) to (3.5,0);
	
	\node at (3.75,0) {\large $i$};
	\node at (0,2.75) {\large $j$};
	
	\node at (2.5,2.75) {\large $E^2_{i,j}-\text{page}$};
	
			
	\node[draw, fill=white, inner sep=2pt, circle] at (0,0) {};
	\node[draw, fill=black, inner sep=2pt, circle] at (1,0) {};
	
	\node[fill=white] (21) at (2,1) {$\phiLDRQ$}; 
	
			
	\node[fill=white] at (0,1) {\large $\underline{0}$};
	\node[fill=white] at (0,2) {\large $\underline{0}$};
	
	\node at (1,1) {\large $\underline{0}$};	
	\node at (1,2) {\large $\underline{0}$};	
	
	\node[fill=white] at (2,0) {\large $\underline{0}$};
	\node at (2,2) {\large $\underline{0}$};
	
	\node[fill=white] at (3,0) {\large $\cdots$};
	
	\foreach \n in {1,...,2}{ 
		\node at (3,\n) {\large $\cdots$};
	}
	
	\end{tikzpicture} \\
	Key: $\whitecirc = \langle \Z \rangle$; \quad $\fillcirc = \langle \F_2 \rangle$; \quad $\phiLDRQ = \phi_{LDR}^* \underline{Q}$
\end{center}
Again, the spectral sequence collapses. 
\begin{prop}
	The non-zero homotopy of the $-3\overline{\rho}$-suspension of $H\underline{I}$ is
	\[
	\underline{\pi}_n^\K(\Sigma^{-3\overline{\rho}} H\underline{I}) \cong
	\begin{cases}
	\langle \Z \rangle & n=0 \\
	\langle \F_2 \rangle & n=-1 \\
	\phi_{LDR}^* \underline{Q} & n=-3. 
	\end{cases}
	\]
\end{prop}

We can continue inductively in this fashion knowing what the differentials will be at any restriction to a proper subgroup and obtain the rest of the negative cone. 

\begin{theorem}\label{INegConeThm}
 The non-zero homotopy of $\Sigma^{-k\overline{\rho}} H\underline{I}$ for $k \geq 4$ is as follows: 
 \[
 \underline{\pi}_n^\K(\Sigma^{-k\overline{\rho}} H\underline{I}) \cong 
 \begin{cases}
 \langle \Z \rangle & n=0 \\
 \langle \F_2 \rangle & n=-1 \\
 \langle \F_2 \rangle^3 & -k < n < -1; k \text{ odd} \\
 \phi_{LDR}^* \underline{\Z}^* & n=-k; k \text{ even} \\
 \phi_{LDR}^* \underline{Q} & n=-k; n \text{ odd}. 
 \end{cases}
 \]
\end{theorem}

\begin{proof}
	This follows by continuing inductively on $k$ via these spectral sequences. 
\end{proof}
Like with the positive cone, there will be a slightly larger range in which $\Sigma^{-k\overline{\rho}} H\underline{A}$ and $\Sigma^{-k\overline{\rho}} H\underline{\Z}$ agree, but it will be because the connecting morphism in the long exact sequence induced by the short exact sequence of coefficients vanishes. The data of \cref{INegConeThm} will be referred to as the ``negative cone" for $H\underline{I}$ and is depicted in quadrant two of \cref{fig:piK4HI}. 

\section{The Positive Cone for $H\underline{A}$} \label{HAPositiveCone}

Given the results in the previous sections, the only portion of the positive cone of $H\underline{A}$ left to compute is $\underline{\pi}^\K_n(\Sigma^{k\overline{\rho}} H\underline{A})$ for $k \geq 0$ and $0 \leq n \leq k+1$. This will mostly follow by the same approach that we used for $H\underline{I}$, i.e. using cofiber sequences to compute the first suspension and spectral sequences to compute further suspensions. However, we will run into some extension problems and resolve them using the multiplicative structure of $\underline{\pi}^\K_{\star}(H\underline{A})$ as well as the long exact sequence in homotopy given by the cofiber sequence $H\underline{I} \rightarrow H\underline{A} \rightarrow H\underline{\Z}$. Recall that the $\star$ wildcard indicates an $RO(\K)$-graded homotopy. 

\begin{prop}
	The non-zero homotopy of the $\overline{\rho}$-suspension of $H\underline{A}$ is: 
	\[
	\underline{\pi}_n^\K(\Sigma^{\overline{\rho}} H\underline{A}) \cong
	\begin{cases}
		\langle \Z \rangle & n=0 \\
		\phi_{LDR}^* \underline{f} & n=1 \\
		\underline{\Z} & n=3. \\
	\end{cases}
	\]
\end{prop}

\begin{proof}
	In \cref{OneRepHomotopyProp}, we computed $\underline{\pi}_n^\K(\Sigma^{\sigma_L} H\underline{A})$ at level $\K/\K$. Using this data to analyze the long exact sequence given by the cofiber sequence 
	\[ \K/D_+ \wedge \Sigma^{\sigma_L} H\underline{A} \rightarrow \Sigma^{\sigma_L} H\underline{A} \rightarrow \Sigma^{\sigma_L + \sigma_D} H\underline{A} \]
	yields the homotopy of the $\Sigma^{(\sigma_L + \sigma_D)}H\underline{A}$. Then, we can analyze the long exact sequence given by the cofiber sequence 
	\[ \K/R_+ \wedge \Sigma^{\sigma_L + \sigma_D} H\underline{A} \rightarrow \Sigma^{\sigma_L + \sigma_D} H\underline{A} \rightarrow \Sigma^{\overline{\rho}} H\underline{A} \]
	to learn the homotopy of the $\overline{\rho}$-suspension.  
\end{proof}

Similar to when we were computing the positive cone of $H\underline{I}$, we'll use a small spectral sequence argument to compute consecutive $\overline{\rho}$-suspensions. The $E^1$-page of the spectral sequence computing the $2\overline{\rho}$-suspension is as follows: 

\begin{center}
	\begin{tikzpicture}[scale=1.5,font=\tiny]
	
            	
	\draw[->] (0,-0.5) to (0,3.5); 
	\draw[->] (-0.5,0) to (3.5,0);
	
	\node at (3.75,0) {\large $i$};
	\node at (0,3.75) {\large $j$};
	
	\node at (1.5,3.75) {\large $E^1_{i,j}-\text{page}$};
	
			
	\node[draw, fill=white, inner sep=2pt, circle] at (0,0) {};
	
	\node[fill=white] (10) at (1,0) {$\phiLDRf$};
	\node[draw, fill=white, inner sep = 4pt, regular polygon sides=4, 
 minimum width=0pt] (30) at (3,0) {};
	
	\node (11) at (1,1) {\large $\widehat{\whitecirc}$};
	
	\node (31) at (3,1) {\large $\widehat{\whitesquare}$};
	
	\node (32) at (3,2) {\large $\overline{\whitesquare}^{\oplus^3}$};
	\node (33) at (3,3) {\large $\overline{\whitesquare}^{\oplus^2}$};
	 
				
	\node[fill=white] at (0,1) {\large $\underline{0}$};
	\node[fill=white] at (0,2) {\large $\underline{0}$};
	\node[fill=white] at (0,3) {\large $\underline{0}$};
	
	\node at (1,2) {\large $\underline{0}$};		
	\node at (1,3) {\large $\underline{0}$};
	
	\node[fill=white] at (2,0) {\large $\underline{0}$};
	\node at (2,1) {\large $\underline{0}$};
	\node at (2,2) {\large $\underline{0}$};
	\node at (2,3) {\large $\underline{0}$};
		
			
	\draw[orddiff, thick, ->] (11) to (10) {};
	\draw[orddiff, thick, ->] (31) to (30) {};
	\draw[orddiff, thick, ->] (32) to (31) {};
	\draw[orddiff, thick, ->] (33) to (32) {};
	
	\end{tikzpicture} \\
	Key: $\whitesquare = \underline{\Z}$; \quad $\overline{\whitesquare} = \uparrow_e^\K \Z$; \quad $\widehat{\whitesquare} = \bigoplus_H \uparrow_H^\K \underline{\Z}$; \quad $\whitecirc = \langle \Z \rangle$; \quad $\widehat{\whitecirc} = \bigoplus_H \uparrow_H^\K \langle \Z \rangle$; \quad $\phiLDRf = \phi_{LDR}^* \underline{f}$
\end{center}
Given \cref{PosConeComparison}, we need only understand the differentials affecting homotopy in degrees $n \leq 3$. In particular, we need only compute the differentials out of $(1,1)$ and $(3,1)$. As we've seen, it's enough to understand these differentials on the levels of proper subgroups, which is at most a $C_2$-equivariant calculation. Resolving these differentials gives us the following $E^2$-page, where we'll omit a description of terms outside the range of interest. 
\begin{center}
	\begin{tikzpicture}[scale=1.5,font=\tiny]
	
            	
	\draw[->] (0,-0.5) to (0,3.5); 
	\draw[->] (-0.5,0) to (3.5,0);
	
	\node at (3.75,0) {\large $i$};
	\node at (0,3.75) {\large $j$};
	
	\node at (1.5,3.75) {\large $E^2_{i,j}-\text{page}$};
	
			
	\node[draw, fill=white, inner sep=2pt, circle] at (0,0) {};
	
	\node at (1,1) {\large $\whitepent$};
	\node[draw, fill=black, inner sep=2pt, circle]  at (3,0) {}; 
	\node at (3,1) {\large $E^2_{3,1}$};
	\node at (3,2) {\large $E^2_{3,2}$};
	\node at (3,3) {\large $E^2_{3,3}$};
	
	\node[fill=white] at (0,1) {\large $\underline{0}$};
	\node[fill=white] at (0,2) {\large $\underline{0}$};
	\node[fill=white] at (0,3) {\large $\underline{0}$};
	
	\node[fill=white] at (1,0) {\large $\underline{0}$};
	\node at (1,2) {\large $\underline{0}$};		
	\node at (1,3) {\large $\underline{0}$};
	
	\node[fill=white] at (2,0) {\large $\underline{0}$};
	\node at (2,1) {\large $\underline{0}$};
	\node at (2,2) {\large $\underline{0}$};
	\node at (2,3) {\large $\underline{0}$};
	
	\end{tikzpicture} \\
	Key: $\whitecirc = \langle \Z \rangle$; \quad $\whitepent = \phi_{LDR}^* \underline{\Z}$; \quad $\fillcirc = \langle \F_2 \rangle$
\end{center}
There isn't any room for differentials, so we can read the homotopy off as follows: 
\begin{prop}
	The non-zero homotopy of the $2\overline{\rho}$-suspension of $H\underline{A}$ is
	\[
	\underline{\pi}_n^\K(\Sigma^{2\overline{\rho}} H\underline{A}) \cong
	\begin{cases}
		\langle \Z \rangle & n=0 \\
		\phi_{LDR}^* \Z & n=2 \\
		\langle \F_2 \rangle & n=3 \\
		\underline{\pi}_n(\Sigma^{2\overline{\rho}} H\underline{\Z}) & n \geq 4
	\end{cases}
	\]
\end{prop}

For further suspensions, we continue using the spectral sequence argument, but we will begin seeing extension problems at this point. The $E^1$-page of the spectral sequence computing the $3\overline{\rho}$-suspension is as follows: 

\begin{center}
	\begin{tikzpicture}[scale=1.5,font=\tiny]
	
            	 
	\draw[->] (0,-0.5) to (0,2.5); 
	\draw[->] (-0.5,0) to (5.5,0);
	
	\node at (5.75,0) {\large $i$};
	\node at (0,2.75) {\large $j$};
	
	\node at (3,2.75) {\large $E^1_{i,j}-\text{page}$};
	
			
	\node[draw, fill=white, inner sep=2pt, circle] at (0,0) {};
	\node[fill=white] (20) at (2,0) {\large $\whitepentdual$}; 
	\node[draw, fill=black, inner sep=2pt, circle] at (3,0) {};
	\node[draw, fill=black, inner sep=3pt, trapezium] (40) at(4,0) {};

	\node (21) at (2,1) {\large $\widehat{\whitecirc}$};
	
	\node (41) at (4,1) {\large $\widehat{\fillcirc}$};
	
	\node[fill=white] at (0,1) {\large $\underline{0}$};
	\node[fill=white] at (0,2) {\large $\underline{0}$};
	
	\node[fill=white] at (1,0) {\large $\underline{0}$};
	\node at (1,1) {\large $\underline{0}$};	
	\node at (1,2) {\large $\underline{0}$};	
	
	\node at (2,2) {\large $\underline{0}$};
	
	\node at (3,1) {\large $\underline{0}$};
	\node at (3,2) {\large $\underline{0}$};
	
	\node at (4,2) {\large $\underline{0}$};
	
	\node[fill=white] at (5,0) {\large $\cdots$};
	
	\foreach \n in {1,...,2}{ 
		\node at (5,\n) {\large $\cdots$};
	}
	
			
	\draw[orddiff, ->] (21) to (20);
	\draw[orddiff, ->] (41) to (40);
	\end{tikzpicture} \\
	Key: $\whitecirc = \langle \Z \rangle$; \quad $\widehat{\whitecirc} = \bigoplus_H \uparrow_H^\K \langle \Z \rangle$; \quad $\widehat{\fillcirc} = \bigoplus_H \uparrow_H^\K \langle \F_2 \rangle$; \quad $\whitepentdual = \phi_{LDR}^* \underline{\Z}^*$; \quad $\filltrap = \underline{mg}$
\end{center}
After resolving the differentials, we obtain the following: 
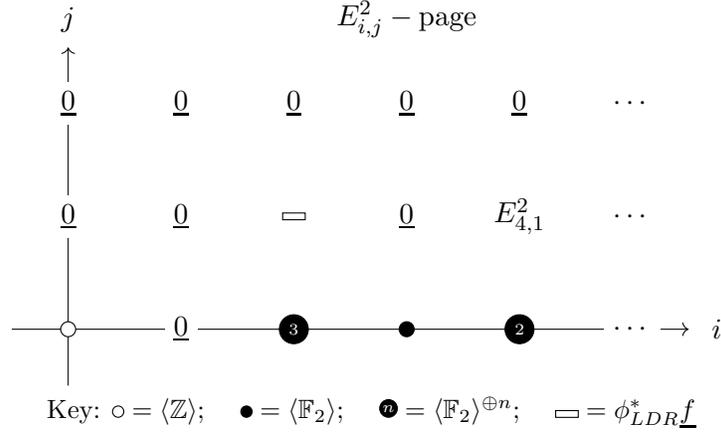
\begin{center}
	\begin{tikzpicture}[scale=1.5,font=\tiny]
	
            	 
	\draw[->] (0,-0.5) to (0,2.5); 
	\draw[->] (-0.5,0) to (5.5,0);
	
	\node at (5.75,0) {\large $i$};
	\node at (0,2.75) {\large $j$};
	
	\node at (3,2.75) {\large $E^2_{i,j}-\text{page}$};
	
			
	\node[draw, fill=white, inner sep=2pt, circle] at (0,0) {};
	\node[draw, fill=black, inner sep=2pt, circle, text=white] at (2,0) {$3$}; 
	\node[draw, fill=black, inner sep=2pt, circle] at (3,0) {};
	\node[draw, fill=black, inner sep=2pt, circle, text=white] at (4,0) {$2$};

	\node at (2,1) {$\phiLDRf$};
	
	\node at (4,1) {\large $E^2_{4,1}$};
	
	\node[fill=white] at (0,1) {\large $\underline{0}$};
	\node[fill=white] at (0,2) {\large $\underline{0}$};
	
	\node[fill=white] at (1,0) {\large $\underline{0}$};
	\node at (1,1) {\large $\underline{0}$};	
	\node at (1,2) {\large $\underline{0}$};	
	
	\node at (2,2) {\large $\underline{0}$};
	
	\node at (3,1) {\large $\underline{0}$};
	\node at (3,2) {\large $\underline{0}$};
	
	\node at (4,2) {\large $\underline{0}$};
	
	\node[fill=white] at (5,0) {\large $\cdots$};
	
	\foreach \n in {1,...,2}{ 
		\node at (5,\n) {\large $\cdots$};
	}
	
	\end{tikzpicture} \\
	Key: $\whitecirc = \langle \Z \rangle$; \quad $\fillcirc = \langle \F_2 \rangle$; \quad $\begin{tikzpicture} \node[draw, fill=black, text=white, circle, inner sep=0.5pt] at (0,0) {\tiny $n$}; \end{tikzpicture} = \langle \F_2 \rangle^{\oplus n}$; \quad $\phiLDRf = \phi_{LDR}^* \underline{f}$
\end{center}
There isn't room for differentials, so the spectral sequence collapses. However, we're left with an extension problem for $\underline{\pi}^\K_3(\Sigma^{3\overline{\rho}} H\underline{A})$. We will refer to the following extension problem as the ``main extension problem" for $\underline{\pi}^\K_3(\Sigma^{3\overline{\rho}} H\underline{A})$: 
\[ \underline{0} \rightarrow \langle \F_2 \rangle \rightarrow \underline{\pi}^\K_3(\Sigma^{3\overline{\rho}} H\underline{A}) \rightarrow  \phi_{LDR}^* \underline{f} \rightarrow \underline{0}. \]
From the main extension problem, we can see that levelwise the exact sequence is split, and the restrictions of the extension must be zero. However, we cannot deduce the transfers of the extension from this information. We'll see that the extension is trivial once we analyze another extension problem. Recall our comparison (i.e. cofiber sequence) 
\[ H\underline{I} \rightarrow H\underline{A} \rightarrow H\underline{\Z}. \]
The long exact sequence in homotopy gives us another extension problem for $\underline{\pi}^\K_3(\Sigma^{3\overline{\rho}} H\underline{A})$, which we'll refer to as the ``coefficients extension problem". This extension takes the following form: 
\[ \underline{0} \rightarrow \phi_{LDR}^* \underline{f} \rightarrow \underline{\pi}^\K_3(\Sigma^{3\overline{\rho}} H\underline{A}) \rightarrow \phi_{LDR}^* \underline{\F_2} \oplus \langle \F_2 \rangle \rightarrow \langle \F_2^3 \rangle. \]
This sequence is enough to force the transfers in $\underline{\pi}^\K_3(\Sigma^{3\overline{\rho}} H\underline{A})$ to be 0, hence the extension is a trivial one. We obtain the homotopy of the $3\overline{\rho}$-suspension: 

\begin{prop}
	The non-zero homotopy of the $3\overline{\rho}$-suspension of $H\underline{A}$ is as follows: 
	\[
	\underline{\pi}_n^\K(\Sigma^{3\overline{\rho}} H\underline{A}) \cong
	\begin{cases}
	\langle \Z \rangle & n=0 \\
	\langle \F_2^3 \rangle & n=2 \\
	\phi_{LDR}^* \underline{f} \oplus \langle \F_2 \rangle & n=3 \\
	\langle \F_2^2 \rangle & n=4 \\
	\underline{\pi}_n(\Sigma^{3\overline{\rho}} H\underline{\Z}) & n \geq 5
	\end{cases}
	\]
\end{prop}

As we continue on, we will run into another extension problem that cannot be resolved in the same way. We have the following $E^1$-page for computing $\underline{\pi}_n^\K(\Sigma^{4\overline{\rho}} H\underline{A})$ in the appropriate range: 

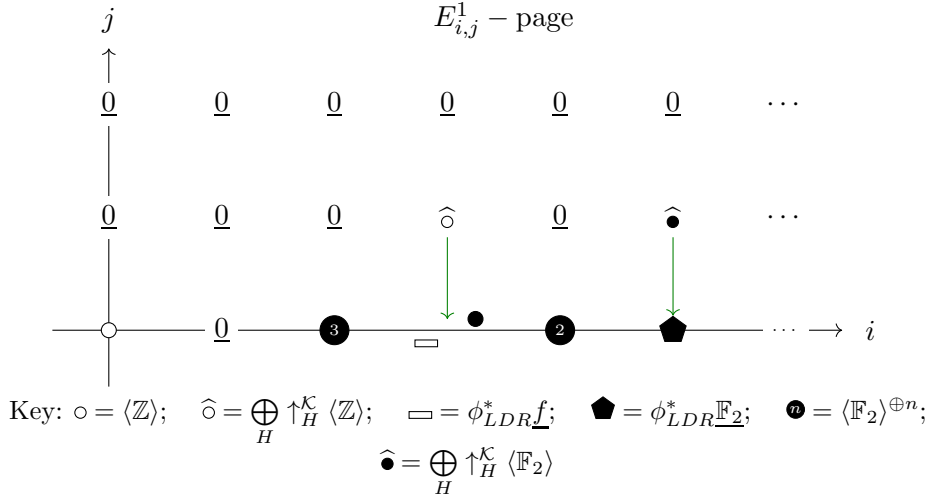
\begin{center}
	\begin{tikzpicture}[scale=1.5,font=\tiny]
	
            	 
	\draw[->] (0,-0.5) to (0,2.5); 
	\draw[->] (-0.5,0) to (6.5,0);
	
	\node at (6.75,0) {\large $i$};
	\node at (0,2.75) {\large $j$};
	
	\node at (3.5,2.75) {\large $E^1_{i,j}-\text{page}$};
	
			
	\node[draw, fill=white, inner sep=2pt, circle] at (0,0) {};
	
	\node[draw, fill=black, inner sep=2pt, circle, text=white] at (2,0) {$3$}; 
	
	\node[below left] at (3,0) {$\phiLDRf$};
	
	\node[draw, circle, fill=black, inner sep=2pt] at (3.25,0.1) {};
	
	\node[draw, fill=black, inner sep=2pt, circle, text=white] at(4,0) {$2$};
	\node[draw, fill=black, inner sep=3pt, regular polygon, regular polygon sides=5] (50) at (5,0) {};

	\node (31) at (3,1) {\large $\widehat{\whitecirc}$};
	
	\node (51) at (5,1) {\large $\widehat{\fillcirc}$};
	
	\node[fill=white] at (0,1) {\large $\underline{0}$};
	\node[fill=white] at (0,2) {\large $\underline{0}$};
	
	\node[fill=white] at (1,0) {\large $\underline{0}$};
	\node at (1,1) {\large $\underline{0}$};	
	\node at (1,2) {\large $\underline{0}$};	
	
	\node at (2,1) {\large $\underline{0}$};
	\node at (2,2) {\large $\underline{0}$};
	
	\node at (3,2) {\large $\underline{0}$};
	
	\node at (4,1) {\large $\underline{0}$};
	\node at (4,2) {\large $\underline{0}$};
	
	\node at (5,2) {\large $\underline{0}$};
	
	\foreach \n in {1,...,2}{ 
		\node at (6,\n) {\large $\cdots$};
	}
	
	\node[fill=white] at (6,0) {$\cdots$};	
	
			
	\draw[orddiff, ->] (31) to (30);
	\draw[orddiff, ->] (51) to (50);
	
	\end{tikzpicture} \\
	Key: $\whitecirc = \langle \Z \rangle$; \quad $\widehat{\whitecirc} = \bigoplus\limits_H \uparrow_H^\K \langle \Z \rangle$; \quad $\phiLDRf = \phi_{LDR}^* \underline{f}$; \quad $\fillpent = \phi_{LDR}^* \underline{\F_2}$; \quad $\begin{tikzpicture} \node[draw, fill=black, text=white, circle, inner sep=0.5pt] at (0,0) {\tiny $n$}; \end{tikzpicture} = \langle \F_2 \rangle^{\oplus n}$; \quad $\widehat{\fillcirc} = \bigoplus\limits_H \uparrow_H^\K \langle \F_2 \rangle$
\end{center}
Resolving the differentials yields the following $E^2$-page: 
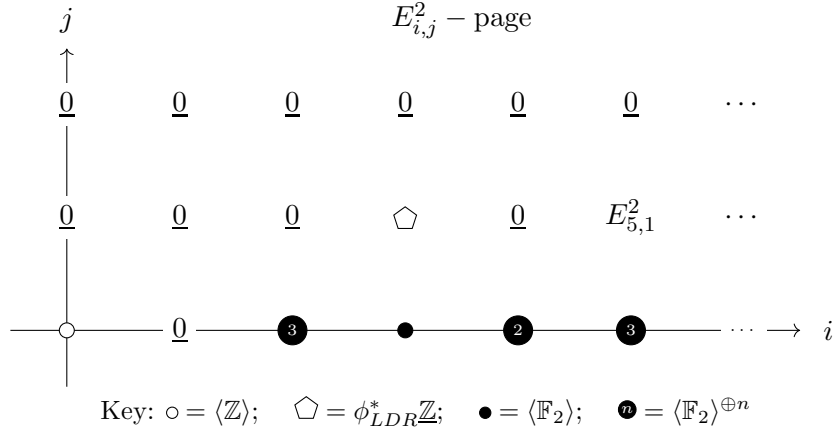
\begin{center}
	\begin{tikzpicture}[scale=1.5,font=\tiny]
	
            	 
	\draw[->] (0,-0.5) to (0,2.5); 
	\draw[->] (-0.5,0) to (6.5,0);
	
	\node at (6.75,0) {\large $i$};
	\node at (0,2.75) {\large $j$};
	
	\node at (3.5,2.75) {\large $E^2_{i,j}-\text{page}$};
	
			
	\node[draw, fill=white, inner sep=2pt, circle] at (0,0) {};
	
	\node[draw, fill=black, inner sep=2pt, circle, text=white] at (2,0) {$3$}; 
	
	\node[draw, circle, fill=black, inner sep=2pt] at (3,0) {};
	
	\node[draw, fill=black, inner sep=2pt, circle, text=white] at(4,0) {$2$};
	
	\node[draw, fill=black, inner sep=2pt, circle, text=white] (50) at (5,0) {$3$};

	\node (31) at (3,1) {\large $\whitepent$};
	
	\node (51) at (5,1) {\large $E^2_{5,1}$};
	
	\node[fill=white] at (0,1) {\large $\underline{0}$};
	\node[fill=white] at (0,2) {\large $\underline{0}$};
	
	\node[fill=white] at (1,0) {\large $\underline{0}$};
	\node at (1,1) {\large $\underline{0}$};	
	\node at (1,2) {\large $\underline{0}$};	
	
	\node at (2,1) {\large $\underline{0}$};
	\node at (2,2) {\large $\underline{0}$};
	
	\node at (3,2) {\large $\underline{0}$};
	
	\node at (4,1) {\large $\underline{0}$};
	\node at (4,2) {\large $\underline{0}$};
	
	\node at (5,2) {\large $\underline{0}$};
	
	\foreach \n in {1,...,2}{ 
		\node at (6,\n) {\large $\cdots$};
	}
	
	\node[fill=white] at (6,0) {$\cdots$};	
	
	\end{tikzpicture} \\
	Key: $\whitecirc = \langle \Z \rangle$; \quad $\whitepent = \phi_{LDR}^* \underline{\Z}$; \quad $\fillcirc = \langle \F_2 \rangle$; \quad $\begin{tikzpicture} \node[draw, fill=black, text=white, circle, inner sep=0.5pt] at (0,0) {\tiny $n$}; \end{tikzpicture} = \langle \F_2 \rangle^{\oplus n}$
\end{center}
Again, there isn't any room for differentials. So, the spectral sequence collapses but we still get an extension problem for $\underline{\pi}^\K_4(\Sigma^{4\overline{\rho}} H\underline{A})$. We'll refer to this extension problem again as the ``main extension problem", and it takes the following form: 
\[ \underline{0} \rightarrow \langle \F_2^2 \rangle \rightarrow \underline{\pi}_4^\K(\Sigma^{4\overline{\rho}} H\underline{A}) \rightarrow \phi_{LDR}^* \underline{\Z} \rightarrow \underline{0}. \]
This extension problem is enough to force a levelwise split exact sequence and the restrictions in $\underline{\pi}_4^\K(\Sigma^{4\overline{\rho}}H\underline{A})$. We are stuck, again, with understanding the transfers. We will see that the sequence is split as a sequence of Mackey functors, but the proof will come later with the multiplicative structure at play here.

\begin{proposition}
	There is an isomorphism
	\[ \underline{\pi}_4^\K(\Sigma^{4\overline{\rho}} H\underline{A}) \cong \phi_{LDR}^* \underline{\Z} \oplus \langle \F_2^2 \rangle. \]
\end{proposition}

\begin{proof}
See \cref{extensionresolution}.
\end{proof}
We read off the homotopy from the $E^2$-page as follows: 
\begin{theorem}
	The non-zero homotopy of the $4\overline{\rho}$-suspension of $H\underline{A}$ is 
	\[
	\underline{\pi}_n^\K(\Sigma^{4\overline{\rho}} H\underline{A}) \cong
	\begin{cases}
		\langle \Z \rangle & n=0 \\
		\langle \F_2^3 \rangle & n=2 \\
		\langle \F_2 \rangle & n=3 \\
		\phi_{LDR}^* \underline{\Z} \oplus \langle \F_2^2 \rangle & n=4 \\
		\langle \F_2^3 \rangle & n=5 \\
		\underline{\pi}^\K_n(\Sigma^{4\overline{\rho}} H\underline{\Z}) & n \geq 6
	\end{cases}
	\]
\end{theorem}

Now, consecutive suspensions will be computed by collapsing spectral sequences, each with one extension problem (in the appropriate range). Further, we will always be able to resolve the extension.

\begin{theorem}
	Let $k \geq 5$. The extension problems in the spectral sequences computing $\underline{\pi}_n(\Sigma^{k\overline{\rho}} H\underline{A})$ in $n \leq k+2$ will occur at $n=k$ and resolve to the following extensions:  
	\[ \underline{\pi}^\K_k(\Sigma^{k\overline{\rho}} H\underline{A}) \cong 
	\begin{cases}
		\phi_{LDR}^* \underline{f} \oplus \langle \F_2^{k-2} \rangle & k \text{ odd} \\
		\phi_{LDR}^* \underline{\Z} \oplus \langle \F_2^{k-2} \rangle & k \text{ even}
	\end{cases}
	\]
\end{theorem}
Putting this all together gives the positive cone for $H\underline{A}$. 
\begin{theorem}\label{APosConeThm}
	The positive cone for $H\underline{A}$ has the following form for $k \geq 5$:
	\[ 
	\underline{\pi}_n^\K(\Sigma^{k\overline{\rho}} H\underline{A}) \cong 
	\begin{cases}
		\langle \Z \rangle & n=0 \\
		\underline{0} & n=1 \\
		\langle \F_2 \rangle^{n+1} & 2 \leq n \leq k-1, n \text{ even} \\		
		\langle \F_2 \rangle^{n-2} & 2 \leq n \leq k-1, n \text{ odd} \\
		\phi_{LDR}^* \underline{f} \oplus \langle \F_2 \rangle^{k-2} & n=k \text{ odd} \\
		\phi_{LDR}^* \underline{\Z} \oplus \langle \F_2 \rangle^{k-2} & n=k \text{ even} \\
		\langle \F_2 \rangle^{k-2} & n=k+1 \\
		\underline{\pi}^\K_n(\Sigma^{k\overline{\rho}} H\underline{\Z}) & n \geq k+2
	\end{cases}
	\]
\end{theorem}

\section{The Negative Cone for $H\underline{A}$} \label{HANegativeCone}
	
	In this section, we'll compute the negative cone for $H\underline{A}$ away from where it agrees with the negative cone for $H\underline{\Z}$. Like the positive cone, this computation will be done using cofiber sequences to compute the first desuspension and spectral sequences to compute further desuspensions. Similarly, we will run into extension problems at $\underline{\pi}^\K_{-n}(\Sigma^{-n\overline{\rho}} H\underline{A})$ for $n > 0$. These will be solved with another comparison, this time with the short exact sequence 
	\[ \underline{0} \rightarrow \underline{\Z}^* \rightarrow \underline{A} \rightarrow \underline{J}. \]
	Here, the map $\underline{\Z}^* \rightarrow \underline{A}$ is given by sending $1 \in \underline{\Z}^*(\K/e)$ to $1 \in \underline{A}(\K/e)$ and $\underline{J}$ is the cokernel of this map. This gives a cofiber sequence 
	\[ H\underline{\Z}^* \rightarrow H\underline{A} \rightarrow H\underline{J}. \]
	Understanding $\underline{\pi}_{n+k\overline{\rho}}^\K(H\underline{J})$ is done through the same method as for $H\underline{I}$, so the homotopy will be depicted in \cref{fig:piK4HJ}.   
	
	\begin{prop}
		The non-zero homotopy of the $-\overline{\rho}$-suspension of $H\underline{A}$ is given by 
	\[
	\underline{\pi}_n^\K(\Sigma^{-\overline{\rho}} H\underline{A}) \cong 
	\begin{cases}
		\langle \Z \rangle & n=0 \\
		\underline{E} & n=-1 \\
		\underline{\Z}^* & n=-3. 
	\end{cases}
	\] 
	\end{prop}		
	
	\begin{proof}
Smashing the dual cofiber sequence $S^{-\sigma_L} \rightarrow \mathbb{S} \rightarrow \K/L_+$ 
with $H\underline{A}$ yields a cofiber sequence
\[ \Sigma^{-\sigma_L} H\underline{A} \rightarrow H\underline{A} \rightarrow \K/L_+ \wedge H\underline{A}. \] 
Analyzing the long exact sequence in homotopy induced by this cofiber sequence yields the homotopy of the $-\sigma_L$-suspension of $H\underline{A}$. Using this as input in the long exact sequence given by the cofiber sequence 
	\[ \Sigma^{-\sigma_L-\sigma_D} H\underline{A} \rightarrow \Sigma^{-\sigma_L} H\underline{A} \rightarrow \K/D_+ \wedge \Sigma^{-\sigma_L} H\underline{A} \]
gives the homotopy of the $-(\sigma_L+\sigma_D)$-suspension. Lastly, inputting this into the long exact sequence given by the cofiber sequence 
	\[ \Sigma^{-\overline{\rho}} H\underline{A} \rightarrow \Sigma^{-\sigma_L-\sigma_D} H\underline{A} \rightarrow \K/R_+ \wedge \Sigma^{-\sigma_L-\sigma_D} H\underline{A} \]
gives the homotopy of the $-\overline{\rho}$-suspension. 
	\end{proof}
	
	Now, we turn to spectral sequences to compute the $-2\overline{\rho}$-suspension. The $E_1$-page is as follows: 
	
	\begin{center}
	\begin{tikzpicture}[scale=1.5,font=\tiny]
	
            	
	\draw[->] (0,-0.5) to (0,3.5); 
	\draw[->] (-0.5,0) to (3.5,0);
	
	\node at (3.75,0) {\large $i$};
	\node at (0,3.75) {\large $j$};
	
	\node at (1.5,3.75) {\large $E_1^{i,j}-\text{page}$};
	
			
	\node[draw, fill=white, inner sep=2pt, circle] at (0,0) {};
	
	\node[fill=white] (10) at (1,0) {$\bardot$};
	
	\node[fill=white] (30) at (3,0) {\large $\whitesquaredual$};
	\node (31) at (3,1) {\large $\widehat{\whitesquaredual}$};
	\node (32) at (3,2) {\large $\overline{\whitesquare}^{\oplus 3}$};
	\node (33) at (3,3) {\large $\overline{\whitesquare}^{\oplus 2}$};
	
	\node (11) at (1,1) {\large $\widehat{\whitecirc}$}; 
	
	\node[fill=white] at (0,1) {\large $\underline{0}$};
	\node[fill=white] at (0,2) {\large $\underline{0}$};
	\node[fill=white] at (0,3) {\large $\underline{0}$};
	
	\node at (1,2) {\large $\underline{0}$};		
	\node at (1,3) {\large $\underline{0}$};
	
	\node[fill=white] at (2,0) {\large $\underline{0}$};
	\node at (2,1) {\large $\underline{0}$};
	\node at (2,2) {\large $\underline{0}$};
	\node at (2,3) {\large $\underline{0}$};
		
			
	\draw[orddiff, thick, ->] (10) to (11) {};
	\draw[orddiff, thick, ->] (30) to (31) {};
	\draw[orddiff, thick, ->] (31) to (32) {};
	\draw[orddiff, thick, ->] (32) to (33) {};
	\end{tikzpicture} \\
	Key: $\whitecirc = \langle \Z \rangle$; \quad $\widehat{\whitecirc} = \bigoplus\limits_H \uparrow_H^\K \langle \Z \rangle$; \quad $\whitesquaredual = \underline{\Z}^*$; \quad $\widehat{\whitesquaredual} = \bigoplus\limits_H \uparrow_H^\K \underline{\Z}^*$; \quad $\overline{\whitesquare} = \uparrow_e^\K \Z$; \quad $\bardot = \underline{E}$
\end{center}
	As with all of the previous computations, it is enough to understand the underlying $C_2$-equivariant spectral sequences to resolve the differentials in the range we are interested in. This gives the following $E_2$-page:
	
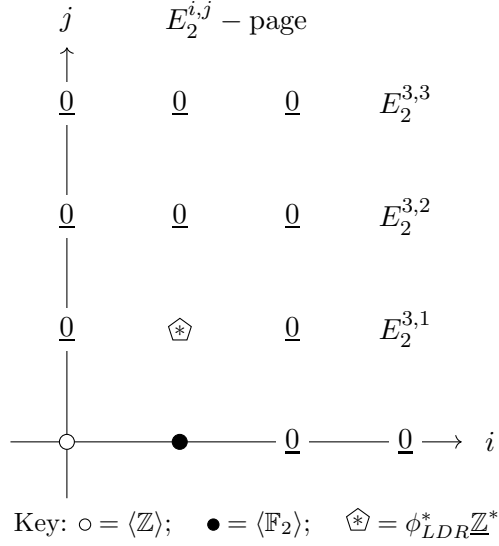
\begin{center}
	\begin{tikzpicture}[scale=1.5,font=\tiny]
	
            	
	\draw[->] (0,-0.5) to (0,3.5); 
	\draw[->] (-0.5,0) to (3.5,0);
	
	\node at (3.75,0) {\large $i$};
	\node at (0,3.75) {\large $j$};
	
	\node at (1.5,3.75) {\large $E_2^{i,j}-\text{page}$};
	
			
	\node[draw, fill=white, inner sep=2pt, circle] at (0,0) {};
	
	\node[draw, circle, fill=black, inner sep=2pt] at (1,0) {};
	
	\node[fill=white] at (3,0) {\large $\underline{0}$};
	\node[fill=white] at (3,1) {\large $E_2^{3,1}$};
	\node[fill=white] at (3,2) {\large $E_2^{3,2}$};
	\node[fill=white] at (3,3) {\large $E_2^{3,3}$};
	
	\node at (1,1) {\large $\whitepentdual$}; 
	
	\node[fill=white] at (0,1) {\large $\underline{0}$};
	\node[fill=white] at (0,2) {\large $\underline{0}$};
	\node[fill=white] at (0,3) {\large $\underline{0}$};
	
	\node at (1,2) {\large $\underline{0}$};		
	\node at (1,3) {\large $\underline{0}$};
	
	\node[fill=white] at (2,0) {\large $\underline{0}$};
	\node at (2,1) {\large $\underline{0}$};
	\node at (2,2) {\large $\underline{0}$};
	\node at (2,3) {\large $\underline{0}$};
		
	\end{tikzpicture} \\
	Key: $\whitecirc = \langle \Z \rangle$; \quad $\fillcirc = \langle \F_2 \rangle$; \quad $\whitepentdual = \phi_{LDR}^* \underline{\Z}^*$
\end{center}
	
	\begin{prop}
		The homotopy of the $-2\overline{\rho}$-suspension of $H\underline{A}$ is 
		\[ 
		\underline{\pi}_n^\K(\Sigma^{-2\overline{\rho}} H\underline{A}) \cong 
		\begin{cases}
		\langle \Z \rangle & n=0 \\
		\langle \F_2 \rangle & n=-1 \\
		\phi_{LDR}^* \underline{\Z}^* & n=-2 \\
		\underline{\pi}_n(\Sigma^{-2\overline{\rho}} H\underline{\Z}) & n \leq -4.
		\end{cases}
		\]
	\end{prop}
	
Continuing on for the $-3\overline{\rho}$-suspension, we have the following $E_1$-page: 

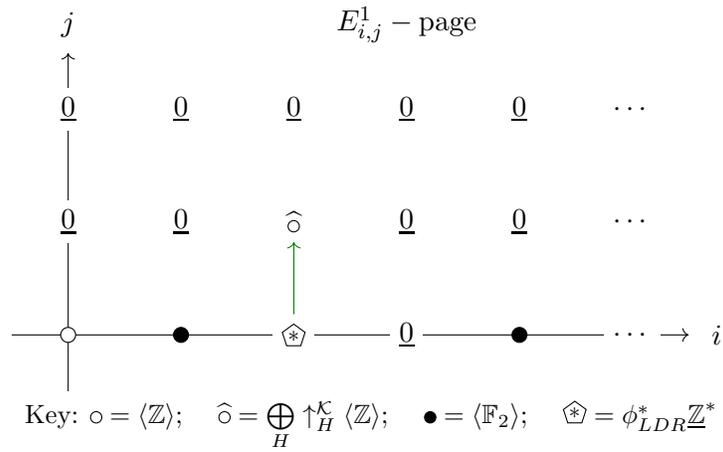
\begin{center}
	\begin{tikzpicture}[scale=1.5,font=\tiny]
	
            	 
	\draw[->] (0,-0.5) to (0,2.5); 
	\draw[->] (-0.5,0) to (5.5,0);
	
	\node at (5.75,0) {\large $i$};
	\node at (0,2.75) {\large $j$};
	
	\node at (3,2.75) {\large $E^1_{i,j}-\text{page}$};
	
			
	\node[draw, fill=white, inner sep=2pt, circle] at (0,0) {}; 
	\node[draw, fill=black, inner sep=2pt, circle] at (1,0) {};
	\node[fill=white] (20) at(2,0) {\large $\whitepentdual$};
	\node[draw, fill=black, inner sep=2pt, circle] at (4,0) {};
	
	\node (21) at (2,1) {\large $\widehat{\whitecirc}$};
	
	\node[fill=white] at (0,1) {\large $\underline{0}$};
	\node[fill=white] at (0,2) {\large $\underline{0}$};
	
	\node at (1,1) {\large $\underline{0}$};	
	\node at (1,2) {\large $\underline{0}$};	
	
	\node at (2,2) {\large $\underline{0}$};
	
	\node[fill=white] at (3,0) {\large $\underline{0}$};
	\node at (3,1) {\large $\underline{0}$};
	\node at (3,2) {\large $\underline{0}$};
	
	\node at (4,1) {\large $\underline{0}$};
	\node at (4,2) {\large $\underline{0}$};
	
	\node[fill=white] at (5,0) {\large $\cdots$};
	
	\foreach \n in {1,...,2}{ 
		\node at (5,\n) {\large $\cdots$};
	}
	
			
	\draw[orddiff, ->] (20) to (21);
	\end{tikzpicture} \\
	Key: $\whitecirc = \langle \Z \rangle$; \quad $\widehat{\whitecirc} = \bigoplus\limits_H \uparrow_H^\K \langle \Z \rangle$; \quad $\fillcirc = \langle \F_2 \rangle$; \quad $\whitepentdual = \phi_{LDR}^* \underline{\Z}^*$
\end{center}
Resolving the differentials gives the following $E_2$-page: 
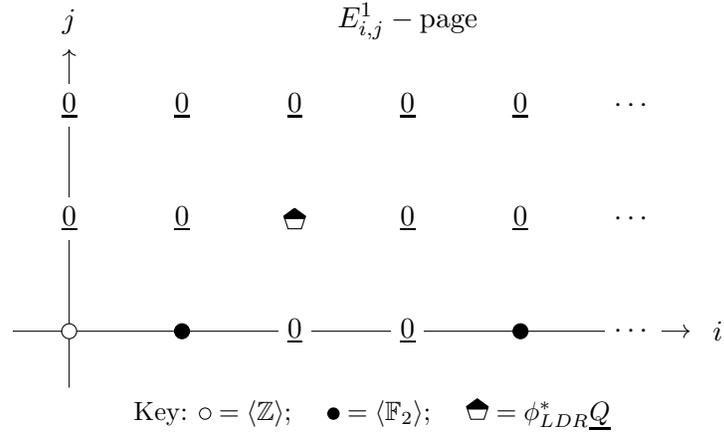
\begin{center}
\begin{tikzpicture}[scale=1.5,font=\tiny]
	
            	 
	\draw[->] (0,-0.5) to (0,2.5); 
	\draw[->] (-0.5,0) to (5.5,0);
	
	\node at (5.75,0) {\large $i$};
	\node at (0,2.75) {\large $j$};
	
	\node at (3,2.75) {\large $E^1_{i,j}-\text{page}$};
	
			
	\node[draw, fill=white, inner sep=2pt, circle] at (0,0) {}; 
	\node[draw, fill=black, inner sep=2pt, circle] at (1,0) {};
	\node[draw, fill=black, inner sep=2pt, circle] at (4,0) {};
	
	\node at (2,1) {$\phiLDRQ$};
	
	\node[fill=white] at (0,1) {\large $\underline{0}$};
	\node[fill=white] at (0,2) {\large $\underline{0}$};
	
	\node at (1,1) {\large $\underline{0}$};	
	\node at (1,2) {\large $\underline{0}$};	
	
	\node[fill=white] at (2,0) {\large $\underline{0}$};
	\node at (2,2) {\large $\underline{0}$};
	
	\node[fill=white] at (3,0) {\large $\underline{0}$};
	\node at (3,1) {\large $\underline{0}$};
	\node at (3,2) {\large $\underline{0}$};
	
	\node at (4,1) {\large $\underline{0}$};
	\node at (4,2) {\large $\underline{0}$};
	
	\node[fill=white] at (5,0) {\large $\cdots$};
	
	\foreach \n in {1,...,2}{ 
		\node at (5,\n) {\large $\cdots$};
	}
	
\end{tikzpicture} \\
Key: $\whitecirc = \langle \Z \rangle$; \quad $\fillcirc = \langle \F_2 \rangle$; \quad $\phiLDRQ = \phi_{LDR}^* \underline{Q}$
\end{center}	
There's no room for differentials, so the spectral collapses. Further, there isn't an extension problem to resolve, so we can read the homotopy off: 
	
\begin{prop}
The non-zero homotopy of the $-3\overline{\rho}$-suspension of $H\underline{A}$ is
	\[
	\underline{\pi}_n^\K(\Sigma^{-3\overline{\rho}} H\underline{A}) \cong
	\begin{cases}
		\langle \Z \rangle & n=0 \\
		\langle \F_2 \rangle & n=-1 \\
		\phi_{LDR}^* \underline{Q} & n=-3 \\
		\langle \F_2 \rangle & n=-4 \\
		\underline{\pi}_n(\Sigma^{-3\overline{\rho}} H\underline{\Z}) & n \leq -5.
	\end{cases}
	\]
	\end{prop}
	
	For the next suspension, we have the following $E_1$-page: 
	
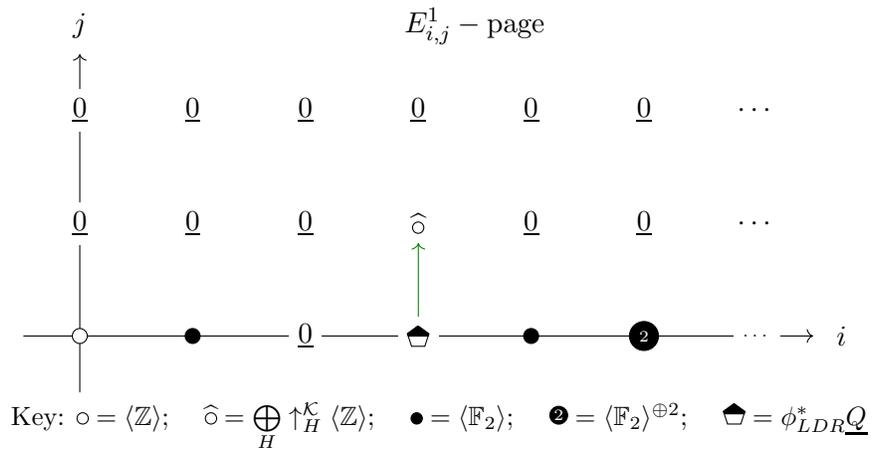
\begin{center}
	\begin{tikzpicture}[scale=1.5,font=\tiny]
	
            	 
	\draw[->] (0,-0.5) to (0,2.5); 
	\draw[->] (-0.5,0) to (6.5,0);
	
	\node at (6.75,0) {\large $i$};
	\node at (0,2.75) {\large $j$};
	
	\node at (3.5,2.75) {\large $E^1_{i,j}-\text{page}$};
	
			
	\node[draw, fill=white, inner sep=2pt, circle] at (0,0) {};
	
	\node[draw, fill=black, inner sep=2pt, circle] at (1,0) {}; 
	
	\node[fill=white] at (3,0) {$\phiLDRQ$};
	
	\node[draw, fill=black, inner sep=2pt, circle] at(4,0) {};
	
	\node[draw, fill=black, inner sep=2pt, circle, text=white] at (5,0) {$2$};

	\node (31) at (3,1) {\large $\widehat{\whitecirc}$};
	
	\node[fill=white] at (0,1) {\large $\underline{0}$};
	\node[fill=white] at (0,2) {\large $\underline{0}$};
	
	\node at (1,1) {\large $\underline{0}$};	
	\node at (1,2) {\large $\underline{0}$};	
	
	\node[fill=white] at (2,0) {\large $\underline{0}$};
	\node at (2,1) {\large $\underline{0}$};
	\node at (2,2) {\large $\underline{0}$};
	
	\node at (3,2) {\large $\underline{0}$};
	
	\node at (4,1) {\large $\underline{0}$};
	\node at (4,2) {\large $\underline{0}$};
	
	\node at (5,1) {\large $\underline{0}$};
	\node at (5,2) {\large $\underline{0}$};
	
	\foreach \n in {1,...,2}{ 
		\node at (6,\n) {\large $\cdots$};
	}
	
	\node[fill=white] at (6,0) {$\cdots$};	
	
			
	\draw[orddiff, ->] (30) to (31);
	
\end{tikzpicture} \\
Key: $\whitecirc = \langle \Z \rangle$; \quad $\widehat{\whitecirc} = \bigoplus\limits_H \uparrow_H^\K \langle \Z \rangle$; \quad $\fillcirc = \langle \F_2 \rangle$; \quad $\begin{tikzpicture} \node[draw, fill=black, text=white, circle, inner sep=0.5pt] at (0,0) {\tiny $2$}; \end{tikzpicture} = \langle \F_2 \rangle^{\oplus 2}$; \quad $\phiLDRQ = \phi_{LDR}^* \underline{Q}$
\end{center}
	Resolving the differential gives the following $E_2$-page:
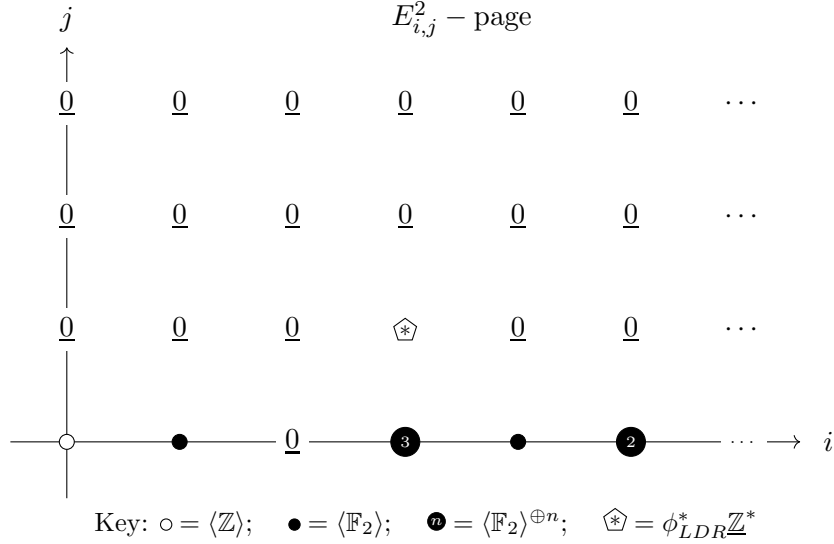
\begin{center}
\begin{tikzpicture}[scale=1.5,font=\tiny]
	
            	 
	\draw[->] (0,-0.5) to (0,3.5); 
	\draw[->] (-0.5,0) to (6.5,0);
	
	\node at (6.75,0) {\large $i$};
	\node at (0,3.75) {\large $j$};
	
	\node at (3.5,3.75) {\large $E^2_{i,j}-\text{page}$};
	
			
	\node[draw, fill=white, inner sep=2pt, circle] at (0,0) {};
	
	\node[draw, fill=black, inner sep=2pt, circle] at (1,0) {}; 
	
	\node[draw, fill=black, circle, inner sep=2pt, text=white] at (3,0) {$3$};
	
	\node[draw, fill=black, inner sep=2pt, circle] at(4,0) {};
	
	\node[draw, fill=black, inner sep=2pt, circle, text=white] at (5,0) {$2$};

	\node (31) at (3,1) {\large $\whitepentdual$};
	
	\node[fill=white] at (0,1) {\large $\underline{0}$};
	\node[fill=white] at (0,2) {\large $\underline{0}$};
	\node[fill=white] at (0,3) {\large $\underline{0}$};
	
	\node at (1,1) {\large $\underline{0}$};	
	\node at (1,2) {\large $\underline{0}$};		
	\node at (1,3) {\large $\underline{0}$};
	
	\node[fill=white] at (2,0) {\large $\underline{0}$};
	\node at (2,1) {\large $\underline{0}$};
	\node at (2,2) {\large $\underline{0}$};
	\node at (2,3) {\large $\underline{0}$};
	
	\node at (3,2) {\large $\underline{0}$};
	\node at (3,3) {\large $\underline{0}$};
	
	\node at (4,1) {\large $\underline{0}$};
	\node at (4,2) {\large $\underline{0}$};
	\node at (4,3) {\large $\underline{0}$};
	
	\node at (5,1) {\large $\underline{0}$};
	\node at (5,2) {\large $\underline{0}$};
	\node at (5,3) {\large $\underline{0}$};
	
	\foreach \n in {1,...,3}{ 
		\node at (6,\n) {\large $\cdots$};
	}
	
	\node[fill=white] at (6,0) {$\cdots$};	
	
\end{tikzpicture} \\
Key: $\whitecirc = \langle \Z \rangle$; \quad $\fillcirc = \langle \F_2 \rangle$; \quad $\begin{tikzpicture} \node[draw, fill=black, text=white, circle, inner sep=0.5pt] at (0,0) {\tiny $n$}; \end{tikzpicture} = \langle \F_2 \rangle^{\oplus n}$; \quad $\whitepentdual = \phi_{LDR}^* \underline{\Z}^*$
\end{center}
Here, the spectral sequence collapses as there isn't room for differentials, but we're left with an extension problem for $\underline{\pi}^\K_{-4}$. The ``main extension problem" is of the form
	\[ \underline{0} \rightarrow \phi_{LDR}^* \underline{\Z}^* \rightarrow \underline{\pi}^\K_{-4}(\Sigma^{-4\overline{\rho}} H\underline{A}) \rightarrow \langle \F_2 \rangle \rightarrow \underline{0}. \]
This extension problem does not determine the value of $\underline{\pi}_{-4}^\K(\Sigma^{-4\overline{\rho}} H\underline{A})$ at level $\K/\K$, the restrictions, or the transfers of $\underline{\pi}_{-4}^\K(\Sigma^{-4\overline{\rho}} H\underline{A})$. Further, the long exact sequence in cohomology induced by the short exact sequence of coefficients
\[ \underline{0} \rightarrow \underline{I} \rightarrow \underline{A} \rightarrow \underline{\Z} \rightarrow \underline{0}
\]
gives the same extension problem. Consequently, we will prove the extension is split using another short exact sequence of coefficients. 
\begin{theorem}
	There is an isomorphism 
	\[ \underline{\pi}_{-4}^\K(\Sigma^{-4\overline{\rho}} H\underline{A}) \cong \phi_{LDR}^* \underline{\Z}^* \oplus \underline{g}. \]
\end{theorem} 

\begin{proof}
Recall the short exact sequence 
\[ \underline{0} \rightarrow \underline{\Z}^* \rightarrow \underline{A} \rightarrow \underline{J} \rightarrow \underline{0} \]
of coefficients and the induced cofiber sequence 
\[ H\underline{\Z}^* \rightarrow H\underline{A} \rightarrow H\underline{J}. \]
The long exact sequence in homotopy tells us that there is at least one copy of $\underline{g}$ including into $\underline{\pi}_{-4}^\K(\Sigma^{-4\overline{\rho}} H\underline{A})$, and this forces
\[ \underline{\pi}_{-4}^\K(\Sigma^{-4\overline{\rho}} H\underline{A})(\K/\K) \cong \phi_{LDR}^* \underline{\Z}^*(\K/\K) \oplus \F_2.
\]
The ``main extension problem" forces the restrictions and transfers, and we see that the extension is a split extension.
\end{proof}

This finishes the spectral sequence computation: 
\begin{prop}
	The non-zero homotopy of the $-4\overline{\rho}$-suspension of $H\underline{A}$ is
	\[
	\underline{\pi}_n^{\K}(\Sigma^{-4\overline{\rho}} H\underline{A}) \cong 
	\begin{cases}
		\langle \Z \rangle & n=0 \\
		\langle \F_2 \rangle & n=-1 \\
		\langle \F_2^3 \rangle & n=-3 \\
		\phi_{LDR}^* \underline{\Z}^* \oplus \langle \F_2 \rangle & n=-4 \\
		\langle \F_2^2 \rangle & n=-5 \\
		\underline{\pi}^\K_n(\Sigma^{-4\overline{\rho}} H\underline{\Z}) & n\leq -6.
	\end{cases}
	\]
\end{prop}
	
	Continuing inductively in this fashion yields the rest of the negative cone, where we continue to use the short exact sequence 
	\[ \underline{0} \rightarrow \underline{\Z}^* \rightarrow \underline{A} \rightarrow \underline{J} \rightarrow \underline{0} \]
	to resolve the extension problems appearing in the spectral sequences for computing further desuspensions.  
	
	\begin{theorem}\label{ANegConeThm}
	For $k \geq 5$, the homotopy of the $-k\overline{\rho}$-suspension of $H\underline{A}$ in negative degrees is: 
	\[
	\underline{\pi}_{-n}^\K(\Sigma^{-k\overline{\rho}} H\underline{A}) \cong
	\begin{cases}
	\langle \Z \rangle & n=0 \\
	\langle \F_2 \rangle^{\oplus n} & 2 \leq n \leq k-1, n \text{ odd} \\
	\langle \F_2 \rangle^{\oplus n-3} & 2 \leq n \leq k-1, n \text{ even} \\
	\phi_{LDR}^* \underline{Q} \oplus \langle \F_2 \rangle^{\oplus k-3} & n=k, n \text{ odd} \\
	\phi_{LDR}^* \underline{\Z}^* \oplus \langle \F_2 \rangle^{\oplus k-3} & n=k, n \text{ even} \\
	\langle \F_2 \rangle^{\oplus k-2} & n=k+1 \\
	\underline{\pi}^\K_{-n}(\Sigma^{-k\overline{\rho}} H\underline{\Z}) & n \geq k+2.
	\end{cases}
	\]
	\end{theorem}
	This completes the computation of the negative cone for $H\underline{A}$.
	
\section{Remarks on the Multiplicative Structure} \label{RmksOnMultStructure}

In this last section, we'll see enough of the multiplicative structure to resolve the extension problems in the positive cone. First, recall that $\underline{A}$ is a commutative monoid in $\Mack_\K$ since it is the unit of the symmetric monoidal structure. As $H(-): \Mack_G \rightarrow Ho\Sp_G$ is a lax monoidal functor, it preserves monoids. 

\begin{definition}
	A commutative Green functor $\underline{R}$ for the group $G$ is a commutative monoid in $\Mack_G$. Levelwise, commutative Green functors are commutative rings. For $J \leq H \leq G$, the restriction maps are ring maps and the transfer maps are module maps, where $\underline{R}(J)$ is made into an $\underline{R}(H)$-module via the restriction map. This is known as ``Frobenius reciprocity". 
\end{definition}

Given that $H\underline{A}$ is a commutative monoid in $Ho\Sp_G$, the homotopy Mackey functors will form an $RO(G)$-graded Green functor. This is equivalent to a Green functor valued in $RO(G)$-graded rings, in which the restriction maps are maps of $RO(G)$-graded rings and the transfer maps satisfy an analogous Frobenius reciprocity condition. As mentioned, $\underline{A}$ is the unit for the symmetric monoidal product, and this implies that $\underline{M}$ is a module over $\underline{A}$ for any $\underline{M} \in \Mack_\K$. Further, we have that $H\underline{M}$ is a module over $H\underline{A}$ in the equivariant stable homotopy category. For example, $H\underline{\F_2}$ is a module over $H\underline{A}$ whose $RO(\K)$-graded homotopy is discussed in \cite{HausmannSchwede}. Similarly, $H\underline{\Z}$ is a module over $H\underline{A}$. Before going on, let's name some distinguished elements in $\underline{\pi}_\star(H\underline{A})$.

\begin{definition} Let $\sphere_\K$ denote the $\K$-equivariant sphere spectrum. \\
\begin{enumerate}[i)]
	\item The Euler class $a_{\overline{\rho}} \in \underline{\pi}_{-\overline{\rho}}^\K(H\underline{A})$ associated to $\overline{\rho}$ is the Hurewicz image of the class represented by the equivariant inclusion of fixed points $S^0 \rightarrow S^{\overline{\rho}}$ in $\underline{\pi}_{-\overline{\rho}}(\sphere_\K)$. 
	\item Recall that $\underline{\pi}^\K_{2-2\overline{\rho}}(H\underline{A}) \cong \phi_{LDR}^* \underline{\Z}$. Let $x_H \in \underline{\pi}_{2-2\overline{\rho}}^\K(H\underline{A})(\K/\K)$ be such that $res_H^\K x_H = a_\sigma^4$ and $\res_J^\K x_H = 0$ for $J \neq H$ an order two subgroup. Since the restrictions in $\phi_{LDR}^* \underline{\Z}$ are isomorphisms, this is well-defined. 
	\item Recall that $\underline{\pi}_{3-\overline{\rho}}^\K(H\underline{A}) \cong \underline{\Z}$. Let $u \in \underline{\pi}_{3-\overline{\rho}}^\K(H\underline{A})(\K/\K)\cong \Z$ be an element that restricts to 1 at the underlying level $(\K/e)$. This is known as an "equivariant orientation class", and the restrictions of $u$ to any non-trivial, proper subgroup will correspond to the orientation class for $C_2$-equivariant $H\underline{A}$. 
\end{enumerate}
\end{definition}

\begin{remark} \label{extensionresolution}
(Extension Resolution) Recall that we have extension problems for $\underline{\pi}_n^\K(\Sigma^{n\overline{\rho}} H\underline{A})$ when $n \geq 4$ is even, but that we know the levelwise values and restrictions of this Mackey functor. To understand the transfers, we simply use Frobenius reciprocity.  

For example, let's consider the main extension problem for $\underline{\pi}_4^\K(\Sigma^{4\overline{\rho}} H\underline{A})$: 
\[ \underline{0} \rightarrow \underline{g}^2 \rightarrow \underline{\pi}_4^\K(\Sigma^{4\overline{\rho}} H\underline{A}) \rightarrow \phi_{LDR}^* \underline{\Z} \rightarrow \underline{0}. \]
Again, we know that levelwise the sequence splits and the restrictions in $\underline{\pi}_4^\K$ are forced. In Lewis diagrams, the middle three terms in the sequence look like   
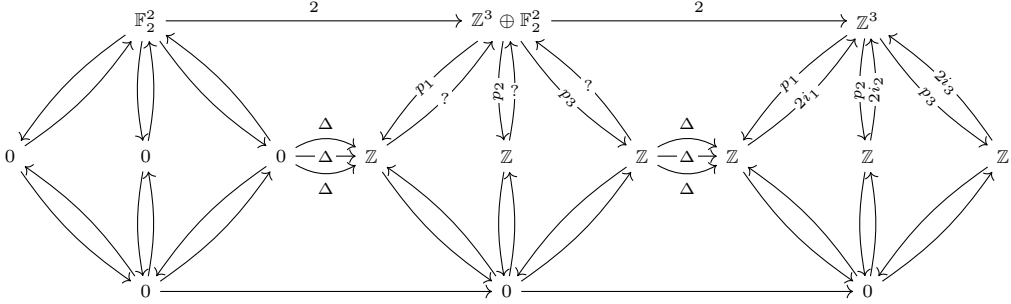
\begin{center}
{ 
\begin{tikzpicture}[scale=0.6]

				
\node (1K) at (-8,6) {\scriptsize $\F_2^2$};
\node (1L) at (-11,3) {\scriptsize $0$};
\node (1D) at (-8,3) {\scriptsize $0$};
\node (1R) at (-5,3) {\scriptsize $0$};
\node (1e) at (-8,0) {\scriptsize $0$};

				
\draw[bend right=10,->] (1L) to (1e);
\draw[bend right=10,->] (1D) to (1e);
\draw[bend right=10,->] (1R) to (1e);

				
\draw[bend right=10,->] (1e) to (1L);
\draw[bend right=10,->] (1e) to (1D);
\draw[bend right=10,->] (1e) to (1R);

				
\draw[bend right=10,->] (1K) to (1L);
\draw[bend right=10,->] (1K) to (1D);
\draw[bend right=10,->] (1K) to (1R);

				
\draw[bend right=10,->] (1L) to (1K);
\draw[bend right=10,->] (1D) to (1K);
\draw[bend right=10,->] (1R) to (1K);

				
\node (2K) at (0,6) {\scriptsize $\Z^3 \oplus \F_2^2$};
\node (2L) at (-3,3) {\scriptsize $\Z$};
\node (2D) at (0,3) {\scriptsize $\Z$};
\node (2R) at (3,3) {\scriptsize $\Z$};
\node (2e) at (0,0) {\scriptsize $0$};

				
\draw[bend right=10,->] (2L) to (2e);
\draw[bend right=10,->] (2D) to (2e);
\draw[bend right=10,->] (2R) to (2e);

				
\draw[bend right=10,->] (2e) to (2L);
\draw[bend right=10,->] (2e) to (2D);
\draw[bend right=10,->] (2e) to (2R);

				
\draw[bend right=10,->] (2K) to node[rotate=45, fill=white, inner sep=0.8pt] {\tiny $p_1$} (2L);
\draw[bend right=10,->] (2K) to node[rotate=90, fill=white, inner sep=0.8pt] {\tiny $p_2$} (2D);
\draw[bend right=10,->] (2K) to node[rotate=-45, fill=white, inner sep=0.8pt] {\tiny $p_3$} (2R);

				
\draw[bend right=10,->] (2L) to node[fill=white, inner sep=0.8pt] {\tiny $?$} (2K);
\draw[bend right=10,->] (2D) to node[fill=white, inner sep=0.8pt] {\tiny $?$} (2K);
\draw[bend right=10,->] (2R) to node[fill=white, inner sep=0.8pt] {\tiny $?$} (2K);

				
\node (3K) at (8,6) {\scriptsize $\Z^3$};
\node (3L) at (5,3) {\scriptsize $\Z$};
\node (3D) at (8,3) {\scriptsize $\Z$};
\node (3R) at (11,3) {\scriptsize $\Z$};
\node (3e) at (8,0) {\scriptsize $0$};

				
\draw[bend right=10,->] (3L) to (3e);
\draw[bend right=10,->] (3D) to (3e);
\draw[bend right=10,->] (3R) to (3e);

				
\draw[bend right=10,->] (3e) to (3L);
\draw[bend right=10,->] (3e) to (3D);
\draw[bend right=10,->] (3e) to (3R);

				
\draw[bend right=10,->] (3K) to node[rotate=45, fill=white, inner sep=1pt] {\tiny $p_1$} (3L);
\draw[bend right=10,->] (3K) to node[rotate=90, fill=white, inner sep=1pt] {\tiny $p_2$} (3D);
\draw[bend right=10,->] (3K) to node[rotate=-45, fill=white, inner sep=1pt] {\tiny $p_3$} (3R);

				
\draw[bend right=10,->] (3L) to node[rotate=45, fill=white, inner sep=1pt] {\tiny $2i_1$} (3K);
\draw[bend right=10,->] (3D) to node[rotate=90, fill=white, inner sep=1pt] {\tiny $2i_2$} (3K);
\draw[bend right=10,->] (3R) to node[rotate=-45, fill=white, inner sep=1pt] {\tiny $2i_3$} (3K);

				
\draw[->] (1K) to node[above] {\tiny $2$} (2K); 
\draw[bend left=30,->] (1R) to node[above] {\tiny $\Delta$} (2L);
\draw[->] (1R) to node[fill=white, inner sep=0.8pt] {\tiny $\Delta$} (2L);
\draw[bend right=30,->] (1R) to node[below] {\tiny $\Delta$} (2L);
\draw[->] (1e) to (2e); 

				
\draw[->] (2K) to node[above] {\tiny $2$} (3K); 
\draw[bend left=30,->] (2R) to node[above] {\tiny $\Delta$} (3L);
\draw[->] (2R) to node[fill=white, inner sep=0.8pt] {\tiny $\Delta$} (3L);
\draw[bend right=30,->] (2R) to node[below] {\tiny $\Delta$} (3L);
\draw[->] (2e) to (3e); 
\end{tikzpicture}
}
\end{center}

Frobenius reciprocity tells us that 
\[ \tau_H^\K(res_H^\K x_H^2) = x_H\tau_H^\K(res_H^\K x_H) = 2x_H^2. \]  
Since $res_H^\K x_H^2$ generates $\underline{\pi}_4^\K(\Sigma^{4\overline{\rho}} H\underline{A})(\K/H)$, we know the transfers! This proves that the ``main extension problem" for $\underline{\pi}_4^\K(\Sigma^{4\overline{\rho}}H\underline{A})$ is a split extension. To learn the transfers in degrees higher than four, note that the main extension problem will always give the same levelwise values and restrictions. Further, we have 
\[ \tau_H^\K(\res_H^\K x_H^{2j}) = x_H^{2j-2} \tau_H^\K(res_H^\K x_H^2) = 2x_H^{2j}.  
\]
\end{remark}

Let's define some more notable elements in $\underline{\pi}_\star^\K(H\underline{A})$. For $H=L,D,R$, we have Euler classes $a_{\sigma_H} \in \underline{\pi}_{-\sigma_H}^\K(H\underline{A})$ corresponding to the equivariant inclusions
\[ S^0 \hookrightarrow S^{\sigma_L}. \]
While these elements aren't depicted in Figures 8.1-8.4 for grading reasons, the product 
\[ a_{\overline{\rho}} = \prod\limits_H a_{\sigma_H} \]
certainly is. Analogous to the $C_2$-equivariant setting, we can understand the multiplicative characteristics of $a_{\sigma_H}$ (torsion and divisibility) via the cofiber sequence
\[ \K/H_+ \rightarrow S^0 \rightarrow S^{\sigma_H} \]
and its dual
\[ S^{-\sigma_H} \rightarrow S^0 \rightarrow \K/H_+. \]

\begin{proposition}
	In $\underline{\pi}_\star^\K(X)$, $a_{\sigma_H}$-torsion \textbf{is} the image of the transfer from $H$.
\end{proposition}

\begin{proof}
Applying $\underline{\pi}_\star^\K((-) \wedge X)(\K/\K)$ to the first cofiber sequence gives a long exact sequence 
\[ \cdots \rightarrow \underline{\pi}_\star^\K(\K/H_+ \wedge X)(\K/\K) \rightarrow \underline{\pi}_\star^\K(X)(\K/\K) \rightarrow \underline{\pi}_{\star - \sigma_H}^\K(X)(\K/\K) \rightarrow \cdots. \]
Up to isomorphism, this is the sequence 
\[ \cdots \rightarrow \underline{\pi}_\star^\K(X)(\K/H) \xrightarrow{tr} \underline{\pi}_\star^\K(X)(\K/\K) \xrightarrow{a_{\sigma_H}} \underline{\pi}_{\star - \sigma_H}^\K(X)(\K/\K) \rightarrow \cdots, \] 
where the first map is the transfer map in the Mackey functor and the map labeled $a_{\sigma_H}$ is multiplication by $a_{\sigma_H}$. 
\end{proof}

\begin{proposition}
	In $\underline{\pi}_\star^\K(X)$, the kernel of restriction to $H$ \textbf{is} the set of $a_{\sigma_H}$-divisible elements.  
\end{proposition}

\begin{proof}
The long exact sequence in homotopy given by the dual cofiber sequence is 
\[ \cdots \rightarrow \underline{\pi}_{\star+\sigma_H^\K}(X)(\K/\K) \rightarrow \underline{\pi}_\star^\K(X)(\K/\K) \rightarrow \underline{\pi}_{\star}^\K(\K/H_+ \wedge X)(\K/\K) \rightarrow \cdots. \]
Up to isomorphism, we have 
\[ \cdots \rightarrow \underline{\pi}_{\star+\sigma_H^\K}(X)(\K/\K) \xrightarrow{a_H} \underline{\pi}_\star^\K(X)(\K/\K) \xrightarrow{res_H^\K} \underline{\pi}_\star^\K(X)(\K/H) \rightarrow \cdots. \]
\end{proof}

We would like similar statements for $a_{\overline{\rho}}$ but don't have ones quite as strong. Since $a_{\overline{\rho}} = a_{\sigma_L}a_{\sigma_D}a_{\sigma_R}$, torsion elements for any of the factors of $a_{\overline{\rho}}$ are torsion for $a_{\overline{\rho}}$. The converse, however, is not necessarily true. 
\begin{remark} \label{torsionslogan}
	Elements in the image of the transfers from $H=L,D,R$ are $a_{\overline{\rho}}$-torsion, but not all $a_{\overline{\rho}}$-torsion elements are hit by some transfer. 
\end{remark}
We have a family of examples in the negative cone for $H\underline{A}$ that demonstrate \cref{torsionslogan}. Consider $(x,y) = (-1,1),(-4,2),(-7,3),(-10,4),...$ in \cref{fig:piK4HA}. The classes here aren't hit by transfers because the Mackey functors are inflated, but must certainly by $a_{\overline{\rho}}$-torsion because the target group (Mackey functor) is zero. It's also worth noting that all but one of the members of this family of examples lies in the negative cone of $H\underline{\Z}$, not just $H\underline{A}$.

A better statement can be made for divisibility since an $a_{\overline{\rho}}$-divisible element is certainly divisible by $a_{\sigma_H}$ for all of $H=L,D,R$. Thus, such an element is in the kernel of all of the restriction maps. The converse does not hold, so unfortunately we only have a test for non-divisibility. 
\begin{remark}
Elements that are $a_{\overline{\rho}}$-divisible are in the kernels of the restrictions to $H \in \{L,D,R\}$, but not all such elements are $a_{\overline{\rho}}$-divisible. 
\end{remark}

We can give some depiction of the multiplicative structure of $\underline{\pi}_n^\K(\Sigma^{k\overline{\rho}} H\underline{A})(\K/\K)$ at this point, and it's worth noting that we completely understand the multiplicative structure at all of the proper subgroup levels as well, see for example \cite{Sikora}. In \cref{fig:multpiK4HA}, there are several multiplicative generators worth collecting. The generators are colored so that arrows of the same color as some generator indicate a non-zero multiplication by this generator. Alternatively, the bi-degree of an element will determine the multiplicative arrows. We'll collect the elements we've seen so far in \cref{MultGenerators}.

\begin{table}
\caption{Distinguished Classes in $\underline{\pi}^\K_{x+y\overline{\rho}}(H\underline{A})$}
\label{MultGenerators}
\begin{tabular}{|| c | c | c ||}
\hline
Element & Bidegree $(x+y\overline{\rho})$ & Additive Generator? \\
\hline
$a_{\overline{\rho}}$ & (0,-1) & Yes \\
\hline
$u$ & (3,-1) & Yes \\
\hline
$\frac{\K - \K/L-\K/D-\K/R+2\K/\K}{a_{\overline{\rho}}}$ & (0,1) & Yes \\
\hline
$x_H$ & (2,-2) & Additive Basis \\
\hline
$\frac{\K}{u}$ & (-3,1) & Yes \\
\hline
$z_H \coloneqq \frac{\K/e-2\K/H}{x_H}$ & (-2,2) & Additive Basis \\
\hline
\end{tabular}
\end{table}

\begin{figure}
\caption{Some Mult. Structure of $\underline{\pi}^\K_{x + y\overline\rho}(H\underline{A})$}
\label{fig:multpiK4HA}
\begin{center}
	\begin{tikzpicture}[scale=0.625,font=\tiny]
	
	\draw[gray] (-12,-6) grid (12,6); 
	\draw[->] (0,-6) to (0,6.25); 
	\draw[->] (-12,0) to (12,0);
	
	\node at (11.5,0.5) [draw] {$x \cdot 1$};
	\node at (0.5,5.5) [draw, rotate=90] {$y \cdot \overline{\rho}$};       
	
            	           
    \node[fill=white, draw, circle, inner sep=0.5pt] (00) at (0,0) {\normalsize $\underline{A}$};
            	           
    \begin{scope}[every node/.style={draw,fill=white,circle,inner sep=1.25pt}]
		\node (0-1) at (0,-1){};
		\node (0-2) at (0,-2){};
		\node (0-3) at (0,-3){};
		\node (0-4) at (0,-4){};
		\node (0-5) at (0,-5){};
		\node (0-6) at (0,-6){};
	\end{scope}
	            
	\node (1-10) at (1,-1) {$\phiLDRf$};

	\node[regular polygon, fill=white, draw, regular polygon sides=4, 
 minimum width=0pt, inner sep = 0.5ex,] (3-1) at (3,-1) {};
	
	
	\node[regular polygon, fill=white, draw, regular polygon sides=5, 
 minimum width=0pt, inner sep = 0.5ex,] (2-2) at (2,-2) {};
	
	\begin{scope}[every node/.style={draw,fill=black,circle,inner sep=1.25pt}]
		\node (3-2) at (3,-2) {};
	\end{scope}
	
	\node[trapezium, fill=black, draw, inner sep=2.5pt,scale=1] (4-2) at (4,-2) {};
	
	\node[regular polygon, fill=white, draw, regular polygon sides=4, 
 minimum width=0pt, inner sep = 0.5ex,] (6-2) at (6,-2) {};
	

	\node[draw,fill=black,circle,inner sep=1.25pt, text=white] (2-3) at (2,-3) {$3$}; 
	
	\node[below left={-0.5ex}] at (3,-3) {$\phiLDRf$};
	
	\begin{scope}[every node/.style={draw,fill=black,circle,inner sep=1.25pt}]
		\node[above right={0.35ex}] (3-3) at (3,-3) {};
	\end{scope}
				
	\node[draw,fill=black,circle,inner sep=1.25pt, text=white] (4-3) at (4,-3) {$2$};
	
	\node[regular polygon, fill=black, draw, regular polygon sides=5, 
 minimum width=0pt, inner sep = 0.5ex,] (5-3) at (5,-3) {};
 
 	\node[draw, fill=black, circle, inner sep=1.25pt] (6-3) at (6,-3) {};
 	
 	\node[trapezium, fill=black, draw, inner sep=2.5pt,scale=1] (7-3) at (7,-3) {};
 	
 	\node[regular polygon, fill=white, draw, regular polygon sides=4, 
 minimum width=0pt, inner sep = 0.5ex,] (9-3) at (9,-3) {};
 
	
	\node[draw,fill=black,circle,inner sep=1.25pt, text=white] (2-4) at (2,-4) {$3$}; 	
	
	\node[draw,fill=black,circle,inner sep=1.25pt, text=white] (3-4) at (3,-4) {};
	
	\node[below left={0.35ex}, regular polygon, fill=white, draw, regular polygon sides=5, 
 minimum width=0pt, inner sep = 0.5ex,] (4-4')at (4,-4) {};
 
 	\node[above right={0.35ex}, draw,fill=black,circle,inner sep=1.25pt, text=white] (4-4) at (4,-4) {$2$}; 
 	
 	\node[draw,fill=black,circle,inner sep=1.25pt, text=white] (5-4) at (5,-4) {$3$}; 
 	
 	\node[below left={0.35ex}, regular polygon, fill=black, draw, regular polygon sides=5, 
 minimum width=0pt, inner sep = 0.5ex,] at (6,-4) {};
 
 	\node[above right={0.35ex}, draw,fill=black,circle,inner sep=1.25pt, text=white] (6-4) at (6,-4) {};
 	
 	\node[draw,fill=black,circle,inner sep=1.25pt, text=white] (7-4) at (7,-4) {$2$};
 	
 	\node[regular polygon, fill=black, draw, regular polygon sides=5, 
 minimum width=0pt, inner sep = 0.5ex,] (8-4) at (8,-4) {};
	
	\node[draw,fill=black,circle,inner sep=1.25pt, text=white] (9-4) at (9,-4) {};
	
	\node[trapezium, fill=black, draw, inner sep=2.5pt,scale=1] (10-4) at (10,-4) {};
	
	\node[regular polygon, fill=white, draw, regular polygon sides=4, 
 minimum width=0pt, inner sep = 0.5ex,] (12-4) at (12,-4) {};
				
			    			
	\node[draw,fill=black,circle,inner sep=1.25pt, text=white] (2-5) at (2,-5) {$3$}; 
	
	\node[draw,fill=black,circle,inner sep=1.25pt, text=white] (3-5) at (3,-5) {};
	
	\node[draw,fill=black,circle,inner sep=1.25pt, text=white] (4-5) at (4,-5) {$5$};  	
	
	\node[below left={-0.5ex}] at (5,-5) {\small $\phiLDRf$};
	
	\begin{scope}[every node/.style={draw,fill=black,circle,inner sep=1.25pt}]
		\node[above right={0.35ex}, text=white] (5-5) at (5,-5) {$3$};
	\end{scope}
	
	\node[draw,fill=black,circle,inner sep=1.25pt, text=white] (6-5) at (6,-5) {$4$};
	
	\node[below left={0.35ex}, regular polygon, fill=black, draw, regular polygon sides=5, 
 minimum width=0pt, inner sep = 0.5ex,] at (7,-5) {};
 
 	\node[above right={0.35ex}, draw,fill=black,circle,inner sep=1.25pt, text=white] (7-5) at (7,-5) {$2$}; 
 	
 	\node[draw,fill=black,circle,inner sep=1.25pt, text=white] (8-5) at (8,-5) {$3$};
 	
 	\node[below left={0.35ex}, regular polygon, fill=black, draw, regular polygon sides=5, 
 minimum width=0pt, inner sep = 0.5ex,] at (9,-5) {};
 
 	\node[above right={0.35ex}, draw,fill=black,circle,inner sep=1.25pt, text=white] (9-5) at (9,-5) {};
 	
 	\node[draw,fill=black,circle,inner sep=1.25pt, text=white] (10-5) at (10,-5) {$2$};
 
 	\node[regular polygon, fill=black, draw, regular polygon sides=5, 
 minimum width=0pt, inner sep = 0.5ex,] (11-5) at (11,-5) {};
 
 \node[draw,fill=black,circle,inner sep=1.25pt, text=white] (12-5) at (12,-5) {};
			
			    			
	\node[draw,fill=black,circle,inner sep=1.25pt, text=white] (2-6) at (2,-6) {$3$}; 
	
	\node[draw,fill=black,circle,inner sep=1.25pt, text=white] (3-6) at (3,-6) {};
	
	\node[draw,fill=black,circle,inner sep=1.25pt, text=white] (4-6) at (4,-6) {$5$};
	
	\node[draw,fill=black,circle,inner sep=1.25pt, text=white] (5-6) at (5,-6) {$3$};
	
	\node[above right={0.35ex}, draw,fill=black,circle,inner sep=1.25pt, text=white] (6-6) at (6,-6) {$4$};   
	
	\node[below left={0.35ex}, regular polygon, fill=white, draw, regular polygon sides=5, 
 minimum width=0pt, inner sep = 0.5ex,] (6-6') at (6,-6) {};
	
	\node[draw,fill=black,circle,inner sep=1.25pt, text=white] (7-6) at (7,-6) {$5$};
	
	\node[below left={0.35ex}, regular polygon, fill=black, draw, regular polygon sides=5, 
 minimum width=0pt, inner sep = 0.5ex,] at (8,-6) {};
 
 	\node[above right={0.35ex}, draw,fill=black,circle,inner sep=1.25pt, text=white] (8-6) at (8,-6) {$3$};
 	
 	\node[draw,fill=black,circle,inner sep=1.25pt, text=white] (9-6) at (9,-6) {$4$};
 	
 	\node[below left={0.35ex}, regular polygon, fill=black, draw, regular polygon sides=5, 
 minimum width=0pt, inner sep = 0.5ex,] at (10,-6) {};
 
 	\node[above right={0.35ex}, draw,fill=black,circle,inner sep=1.25pt, text=white] (10-6) at (10,-6) {$2$};
 
 	\node[draw,fill=black,circle,inner sep=1.25pt, text=white] (11-6) at (11,-6) {$3$};
 	
 	\node[below left={0.35ex}, regular polygon, fill=black, draw, regular polygon sides=5, 
 minimum width=0pt, inner sep = 0.5ex,] at (12,-6) {};
 
 	\node[above right={0.35ex}, draw,fill=black,circle,inner sep=1.25pt, text=white] (12-6) at (12,-6) {};
	

	\foreach \n in {1,...,6}{ 	                      
		\node[draw, fill=white, circle, inner sep=1.25pt] (0\n) at (0,\n) {};
		}
		
	
	\node (-11) at (-1,1) {$\bardot$};
	
	\node[regular polygon, fill=white, draw, regular polygon sides=4, 
 minimum width=0pt, inner sep = 0.5ex,] (-31) at (-3, 1) {};
	\node at (-3,1) {\large $\ast$};	
	
	
	\node[draw, fill=black, circle, inner sep=1.25pt] (-12) at (-1,2) {};
	
	\node[regular polygon, fill=white, draw, regular polygon sides=5, 
 minimum width=0pt, inner sep = 0.5ex,] (-22) at (-2,2) {};
	\node at (-2,2) {\large $\ast$};
	
	\node[draw, fill=black, circle, inner sep = 1.25pt] (-42) at (-4,2) {};
	
	\node[trapezium, fill=black, draw, inner sep=2.5pt,scale=1] (-52) at (-5,2) {};
	\node[text=white] at (-5,2) {\large $\ast$};
	
	\node[regular polygon, fill=white, draw, regular polygon sides=4, 
 minimum width=0pt, inner sep = 0.5ex,] (-62) at (-6,2) {};
	\node at (-6,2) {\large $\ast$};
	
	
	\node[draw, fill=black, circle, inner sep =1.25pt] (-13) at (-1,3) {};
	
	\node (-33) at (-3,3) {$\phiLDRQ$};
	
	\node[draw, fill=black, circle, inner sep=1.25pt] (-43) at (-4,3) {};	
	
	\node[draw, fill=black, circle, inner sep=1.25pt, text=white] (-53) at (-5,3) {$2$};
	
	\node[regular polygon, fill=black, draw, regular polygon sides=5, 
 minimum width=0pt, inner sep = 0.5ex,] (-63) at (-6,3) {};
	\node[text=white] at (-6,3) {\large $\ast$};

	\node[draw, fill=black, circle, inner sep=1.25pt, text=white] (-73) at (-7,3) {};	
	
	\node[draw, trapezium, fill=black, inner sep=2.5pt,scale=1] (-83) at (-8,3) {};
	\node[text=white] at (-8,3) {\large $\ast$};
	
	\node[regular polygon, fill=white, draw, regular polygon sides=4, 
 minimum width=0pt, inner sep = 0.5ex,] (-93) at (-9,3) {};
	\node at (-9,3) {\large $\ast$};
	

	\node[draw, fill=black, circle, inner sep=1.25pt] (-14) at (-1,4) {};
			
	\node[draw, fill=black, circle, inner sep=1.25pt, text=white] (-34) at (-3,4) {$3$};			
			
	\begin{scope}[xshift=-0.2cm, yshift=-0.2cm]
	\node[regular polygon, fill=white, draw, regular polygon sides=5, 
 minimum width=0pt, inner sep = 0.5ex] (-44') at (-4,4) {};
	\node[text=black] at (-4,4) {\large $\ast$};
	\end{scope}
	
	\begin{scope}[xshift=0.2cm, yshift=0.2cm]
	\node[draw, fill=black, circle, inner sep =1.25pt, text=white] (-44) at (-4,4) {};
	\end{scope}
	
	\node[draw, fill=black, circle, inner sep =1.25pt, text=white] (-54) at (-5,4) {$2$};
	
	\node[draw, fill=black, circle, inner sep =1.25pt, text=white] (-64) at (-6,4) {$3$};
	
	\begin{scope}[xshift=-0.2cm, yshift=-0.2cm]
	\node[regular polygon, fill=black, draw, regular polygon sides=5, 
 minimum width=0pt, inner sep = 0.5ex] at (-7,4) {};
	\node[text=white] at (-7,4) {\large $\ast$};
	\end{scope}
	
	\begin{scope}[xshift=0.2cm, yshift=0.2cm]
	\node[draw, fill=black, circle, inner sep =1.25pt, text=white] (-74) at (-7,4) {};
	\end{scope}
	
	\node[draw, fill=black, circle, inner sep =1.25pt, text=white] (-84) at (-8,4) {$2$};
	
	\node[regular polygon, fill=black, draw, regular polygon sides=5, 
 minimum width=0pt, inner sep = 0.5ex,] (-94) at (-9,4) {};
	\node[text=white] at (-9,4) {\large $\ast$};

	\node[draw, fill=black, circle, inner sep=1.25pt, text=white] (-104) at (-10,4) {};
	
	\node[draw, trapezium, fill=black, inner sep=2.5pt,scale=1] (-114) at (-11,4) {};
	\node[text=white] at (-11,4) {\large $\ast$};
	
	\node[regular polygon, fill=white, draw, regular polygon sides=4, 
 minimum width=0pt, inner sep = 0.5ex,] (-124) at (-12,4) {};
	\node at (-12,4) {\large $\ast$};
	
			
	\node[draw, fill=black, circle, inner sep =1.25pt] (-15) at (-1,5) {};
	
	\node[draw, fill=black, circle, inner sep=1.25pt, text=white] (-35) at (-3,5) {$3$};
	
	\node[draw, fill=black, circle, inner sep =1.25pt] (-45) at (-4,5) {};
	
	\begin{scope}[xshift=0.2cm,yshift=0.2cm]
	\node[draw, fill=black, circle, inner sep=1.25pt, text=white] (-55) at (-5,5) {$2$};
	\end{scope}
	
	\begin{scope}[xshift=-0.2cm,yshift=-0.2cm]
	\node at (-5,5) {$\phiLDRQ$};
	\end{scope}
	
	\node[draw, fill=black, circle, inner sep =1.25pt, text=white] (-65) at (-6,5) {$3$};
	
	\node[draw, fill=black, circle, inner sep =1.25pt, text=white] (-75) at (-7,5) {$4$};
	
	\begin{scope}[xshift=-0.2cm, yshift=-0.2cm]
	\node[regular polygon, fill=black, draw, regular polygon sides=5, 
 minimum width=0pt, inner sep = 0.5ex] (-85) at (-8,5) {};
	\node[text=white] at (-8,5) {\large $\ast$};
	\end{scope}
	
	\begin{scope}[xshift=0.2cm, yshift=0.2cm]
	\node[draw, fill=black, circle, inner sep =1.25pt, text=white] (-85) at (-8,5) {$2$};
	\end{scope}
	
	\node[draw, fill=black, circle, inner sep =1.25pt, text=white] (-95) at (-9,5) {$3$};
	
	\begin{scope}[xshift=-0.2cm, yshift=-0.2cm]
	\node[regular polygon, fill=black, draw, regular polygon sides=5, 
 minimum width=0pt, inner sep = 0.5ex] (-105) at (-10,5) {};
	\node[text=white] at (-10,5) {\large $\ast$};
	\end{scope}
	
	\begin{scope}[xshift=0.2cm, yshift=0.2cm]
	\node[draw, fill=black, circle, inner sep =1.25pt, text=white] (-105) at (-10,5) {};
	\end{scope}
	
	\node[draw, fill=black, circle, inner sep =1.25pt, text=white] (-115) at (-11,5) {$2$};
	
	\node[regular polygon, fill=black, draw, regular polygon sides=5, 
 minimum width=0pt, inner sep = 0.5ex,] (-125) at (-12,5) {};
	\node[text=white] at (-12,5) {\large $\ast$};
	
	
	\node[draw, fill=black, circle, inner sep =1.25pt, text=white] (-16) at (-1,6) {};
	
	\node[draw, fill=black, circle, inner sep=1.25pt, text=white] (-36) at (-3,6) {$3$};
	
	\node[draw, fill=black, circle, inner sep =1.25pt, text=white] (-46) at (-4,6) {};
	
	\node[draw, fill=black, circle, inner sep =1.25pt, text=white] (-56) at (-5,6) {$5$};
	
	\begin{scope}[xshift=-0.2cm, yshift=-0.2cm]
	\node[regular polygon, fill=white, draw, regular polygon sides=5, 
 minimum width=0pt, inner sep = 0.5ex] (-66') at (-6,6) {};
	\node[text=black] at (-6,6) {\large $\ast$};
	\end{scope}
	
	\begin{scope}[xshift=0.2cm, yshift=0.2cm]
	\node[draw, fill=black, circle, inner sep =1.25pt, text=white] (-66) at (-6,6) {$3$};
	\end{scope}
	
	\node[draw, fill=black, circle, inner sep =1.25pt, text=white] (-76) at (-7,6) {$4$};
	
	\node[draw, fill=black, circle, inner sep =1.25pt, text=white] (-86) at (-8,6) {$5$};
	
	\begin{scope}[xshift=-0.2cm, yshift=-0.2cm]
	\node[regular polygon, fill=black, draw, regular polygon sides=5, 
 minimum width=0pt, inner sep = 0.5ex] at (-9,6) {};
	\node[text=white] at (-9,6) {\large $\ast$};
	\end{scope}
	
	\begin{scope}[xshift=0.2cm, yshift=0.2cm]
	\node[draw, fill=black, circle, inner sep =1.25pt, text=white] (-96) at (-9,6) {$3$};
	\end{scope}
	
	\node[draw, fill=black, circle, inner sep =1.25pt, text=white] (-106) at (-10,6) {$4$};
	
	\begin{scope}[xshift=-0.2cm, yshift=-0.2cm]
	\node[regular polygon, fill=black, draw, regular polygon sides=5, 
 minimum width=0pt, inner sep = 0.5ex] at (-11,6) {};
	\node[text=white] at (-11,6) {\large $\ast$};
	\end{scope}
	
	\begin{scope}[xshift=0.2cm, yshift=0.2cm]
	\node[draw, fill=black, circle, inner sep =1.25pt, text=white] (-116) at (-11,6) {$2$};
	\end{scope}
		
	\node[draw, fill=black, circle, inner sep =1.25pt, text=white] (-126) at (-12,6) {$3$};

			
	\draw[color=\amult, bend right=25, ->] (00) to (0-1) {};
	\draw[color=\amult, bend right=25, ->] (0-1) to (0-2) {};
	\draw[color=\amult, bend right=25, ->] (0-2) to (0-3) {};
	\draw[color=\amult, bend right=25, ->] (0-3) to (0-4) {};
	\draw[color=\amult, bend right=25, ->] (0-4) to (0-5) {};
	\draw[color=\amult, bend right=25, ->] (0-5) to (0-6) {};		
	
	\draw[color=\amult, ->] (2-2) to (2-3) {};
	\draw[color=\amult, ->] (2-3) to (2-4) {};
	\draw[color=\amult, ->] (2-4) to (2-5) {};
	\draw[color=\amult, ->] (2-5) to (2-6) {};	
	
	\draw[color=\amult, ->] (3-1) to (3-2) {};
	\draw[color=\amult, ->] (3-2) to (3-3) {};
	\draw[color=\amult, ->] (3-3) to (3-4) {};
	\draw[color=\amult, ->] (3-4) to (3-5) {};
	\draw[color=\amult, ->] (3-5) to (3-6) {};	
	
	\draw[color=\amult, ->] (4-2) to (4-3) {};
	\draw[color=\amult, ->] (4-3) to (4-4) {};
	\draw[color=\amult, ->] (4-4) to (4-5) {};
	\draw[color=\amult, ->] (4-5) to (4-6) {};	
	
	\draw[color=\amult, ->] (5-3) to (5-4) {};
	\draw[color=\amult, ->] (5-4) to (5-5) {};
	\draw[color=\amult, ->] (5-5) to (5-6) {};	
	
	\draw[color=\amult, ->] (6-2) to (6-3) {};
	\draw[color=\amult, ->] (6-3) to (6-4) {};
	\draw[color=\amult, ->] (6-4) to (6-5) {};
	\draw[color=\amult, ->] (6-5) to (6-6) {};	
	
	\draw[color=\amult, ->] (7-3) to (7-4) {};
	\draw[color=\amult, ->] (7-4) to (7-5) {};
	\draw[color=\amult, ->] (7-5) to (7-6) {};	
	
	\draw[color=\amult, ->] (8-4) to (8-5) {};
	\draw[color=\amult, ->] (8-5) to (8-6) {};	
	
	\draw[color=\amult, ->] (9-3) to (9-4) {};
	\draw[color=\amult, ->] (9-4) to (9-5) {};
	\draw[color=\amult, ->] (9-5) to (9-6) {};
	
	\draw[color=\amult, ->] (10-4) to (10-5) {};
	\draw[color=\amult, ->] (10-5) to (10-6) {};		
	
	\draw[color=\amult, ->] (11-5) to (11-6) {};
	
	\draw[color=\amult, ->] (12-4) to (12-5) {};
	\draw[color=\amult, ->] (12-5) to (12-6) {};		
	
	\draw[color=\amult, bend left=25, ->] (01) to (00) {};
	\draw[color=\amult, bend left=25, ->] (02) to (01) {};
	\draw[color=\amult, bend left=25, ->] (03) to (02) {};
	\draw[color=\amult, bend left=25, ->] (04) to (03) {};
	\draw[color=\amult, bend left=25, ->] (05) to (04) {};
	\draw[color=\amult, bend left=25, ->] (06) to (05) {};
	
	\draw[color=\amult, ->] (-12) to (-11) {};
	\draw[color=\amult, ->] (-13) to (-12) {};
	\draw[color=\amult, ->] (-14) to (-13) {};
	\draw[color=\amult, ->] (-15) to (-14) {};
	\draw[color=\amult, ->] (-16) to (-15) {};

	\draw[color=\amult, ->] (-34) to (-33) {};
	\draw[color=\amult, ->] (-35) to (-34) {};
	\draw[color=\amult, ->] (-36) to (-35) {};	
	
	\draw[color=\amult, ->] (-43) to (-42) {};
	\draw[color=\amult, ->] (-44) to (-43) {};
	\draw[color=\amult, ->] (-45) to (-44) {};
	\draw[color=\amult, ->] (-46) to (-45) {};	
	
	\draw[color=\amult, ->] (-54) to (-53) {};
	\draw[color=\amult, ->] (-55) to (-54) {};
	\draw[color=\amult, ->] (-56) to (-55) {};	
	
	\draw[color=\amult, ->] (-63) to (-62) {};
	\draw[color=\amult, ->] (-64) to (-63) {};
	\draw[color=\amult, ->] (-65) to (-64) {};
	\draw[color=\amult, ->] (-66) to (-65) {};	
	
	\draw[color=\amult, ->] (-74) to (-73) {};
	\draw[color=\amult, ->] (-75) to (-74) {};
	\draw[color=\amult, ->] (-76) to (-75) {};	
	
	\draw[color=\amult, ->] (-85) to (-84) {};
	\draw[color=\amult, ->] (-86) to (-85) {};	
	
	\draw[color=\amult, ->] (-94) to (-93) {};
	\draw[color=\amult, ->] (-95) to (-94) {};
	\draw[color=\amult, ->] (-96) to (-95) {};
	
	\draw[color=\amult, ->] (-105) to (-104) {};
	\draw[color=\amult, ->] (-106) to (-105) {};		
	
	\draw[color=\amult, ->] (-115) to (-114) {};
	\draw[color=\amult, ->] (-116) to (-115) {};
	
	\draw[color=\amult, ->] (-125) to (-124) {};
	\draw[color=\amult, ->] (-126) to (-125) {};
	
	\draw[color=\umult, ->] (-124) to (-93) {};
	\draw[color=\umult, ->] (-93) to (-62) {};
	\draw[color=\umult, ->] (-62) to (-31) {};
	\draw[color=\umult, ->] (-31) to (00) {};
	\draw[color=\umult, ->] (00) to (3-1) {};
	\draw[color=\umult, ->] (3-1) to (6-2) {};
	\draw[color=\umult, ->] (6-2) to (9-3) {};
	\draw[color=\umult, ->] (9-3) to (12-4) {};
	
	\draw[color=\umult, ->] (-114) to (-83) {};
	\draw[color=\umult, ->] (-83) to (-52) {};
	\draw[color=\umult, ->] (4-2) to (7-3) {};	
	\draw[color=\umult, ->] (7-3) to (10-4) {};
	
	\draw[color=\xmult, ->] (-66') to (-44') {};
	\draw[color=\xmult, ->] (-44') to (-22) {};
	\draw[color=\xmult, ->] (-22) to (00) {};
	\draw[color=\xmult, ->] (00) to (2-2) {};
	\draw[color=\xmult, ->] (2-2) to (4-4') {};
	\draw[color=\xmult, ->] (4-4') to (6-6') {};

				
	\node[color=\amult=100!] at (-0.5,-1.25) {$a_{\overline{\rho}}$};
	\node[color=\amult=100!] at (2.85,1) {$\frac{\K-\K/L-\K/D-\K/R+2\K/\K}{a_{\overline{\rho}}}$};
	
	\node[color=\umult=100!] at (3.3,-0.7) {$u$};
	\node[color=\umult=100!] at (-3.35,0.7) {$\frac{\K}{u}$};
	
	\node[color=\xmult=100!] at (1.7,-2.3) {$x_H$}; 
	\node[color=\xmult=100!] at (-2.3,1.7) {$z_H$};

	
	\draw[fill=white] (-11,-5.25) rectangle (-1,-0.75);
	\draw (-10,-1.6) -- (-3,-1.6);
	
	\node at (-6.4,-1.25) {\large Key:};
	
	\node at (-9,-2.5) {$\whitesquare = \underline{\Z}$};
	\node at (-8.875,-3.5) {$\whitesquaredual = \underline{\Z}^*$};
	\node at (-8.875,-4.5) {$\whitecirc = \langle \Z \rangle$};
	
	\node at (-6.4,-2.5) {$\filltrap = \underline{mg}$};
	\node at (-6.3,-3.5) {$\filltrapdual = \underline{mg}^*$};
	
	\node at (-3,-2.5) {$\fillpent = \phi_{LDR}^* \underline{\F_2}$};
	\node at (-2.9,-3.5) {$\fillpentdual = \phi_{LDR}^* \underline{\F_2}^*$};
	\node at (-3.7, -4.5) {$\bardot = \underline{E}$};
	
	\node at (-6,-4.5) {$\begin{tikzpicture} \node[draw,circle,fill=black,inner sep=0.5pt, text=white] {n}; \end{tikzpicture} = \langle \mathbb{F}_2 \rangle^{\oplus n}$};

	
	\draw[fill=white] (11,5.75) rectangle (2,1.75);
	\draw (10,4.9) -- (3,4.9);
	
	\node at (6.5,5.25) {\large Key:};
	
	\node at (4.5,4) {$\whitepent = \phi_{LDR}^* \underline{\Z}$};
	\node at (4.6,3) {$\whitepentdual = \phi_{LDR}^* \underline{\Z^*}$};
	
	\node at (8.5,4) {$\phiLDRf = \phi_{LDR}^* \underline{f}$};
	\node at (8.6,3) {$\phiLDRQ = \phi_{LDR}^* \underline{Q}$};
	
	\end{tikzpicture}
	
    \end{center}
    \end{figure}

Here, we will collect all of the computations we have done. For detailed descriptions of the symbols used in these charts, see \cref{KZoo}. The blue (quadrant four) and red (quadrant two) regions in \cref{fig:piK4HA} and \cref{fig:piK4HZ} are meant to indicate the regions in which the comparison results \cref{PosConeComparison} and \cref{NegConeComparison}, respectively, apply for $H\underline{A}$ and $H\underline{\Z}$. The green (quadrant two in \cref{fig:piK4HI}) region indicates a region in which $H\underline{I}$ and $H\underline{A}$ agree, while the pink region (quadrant four in \ref{fig:piK4HJ}) indicates a region in which $H\underline{A}$ and $H\underline{J}$ agree. 

\begin{figure}
\caption{$\underline{\pi}^\K_{x + y\overline\rho}(H\underline{I})$}
\label{fig:piK4HI}
\begin{center}
	\begin{tikzpicture}[scale=0.625,font=\tiny]
	
	\draw[color=green!50, thick, fill=green!25, rounded corners] (-2.5,1.5) -- (-0.5,-0.5) -- (-0.5,6.25) -- (-2.5,6.25) -- cycle;
	
	\draw[gray] (-12,-6) grid (12,6); 
	\draw[->] (0,-6) to (0,6.25); 
	\draw[->] (-12,0) to (12,0);
	
	\node at (11.5,0.5) [draw] {$x \cdot 1$};
	\node at (0.5,5.5) [draw, rotate=90] {$y \cdot \overline{\rho}$};

            	           
    \node[draw, circle, fill=white, inner sep=0.5pt] at (0,0) {\normalsize $\underline{I}$};
            	           
    \begin{scope}[every node/.style={draw,fill=white,circle,inner sep=1.25pt}]
		\node[below left={0.35ex}] at (0,-1){};
		\node[below left={0.35ex}] at (0,-2){};
		\node[below left={0.35ex}] at (0,-3){};
		\node[below left={0.35ex}] at (0,-4){};
		\node[below left={0.35ex}] at (0,-5){};
		\node[below left={0.35ex}] at (0,-6){};
	\end{scope}
	
	\begin{scope}[every node/.style={draw,fill=black,circle,inner sep=1pt}]
		\node[text=white, above right={0.35ex}] at (0,-1){2};
		\node[text=white, above right={0.35ex}] at (0,-2){2};
		\node[text=white, above right={0.35ex}] at (0,-3){2};
		\node[text=white, above right={0.35ex}] at (0,-4){2};
		\node[text=white, above right={0.35ex}] at (0,-5){2};
		\node[text=white, above right={0.35ex}] at (0,-6){2};
		
		\foreach \n in {-3,...,-6}{
			\node[text=white] at (2,\n){3};
		}
		
		\foreach \n in {-5,-6}{
			\node[text=white] at (4,\n){3};
		}
    \end{scope}
	
	\node at (1,-1) {$\phiLDRf$};
	
	\node[regular polygon, fill=white, draw, regular polygon sides=5, 
 minimum width=0pt, inner sep = 0.5ex,] at (2,-2) {};
	
	\node at (3,-3) {$\phiLDRf$};
	
	\node[regular polygon, fill=white, draw, regular polygon sides=5, 
 minimum width=0pt, inner sep = 0.5ex,] at (4,-4) {};
	
	\node at (5,-5) {$\phiLDRf$};
	
	\node[regular polygon, fill=white, draw, regular polygon sides=5, 
 minimum width=0pt, inner sep = 0.5ex,] at (6,-6) {};
		
	                      
	\begin{scope}[every node/.style={draw,fill=white,circle,inner sep=1.25pt}]
		\foreach \n in {1,...,6} {
			\node at (0,\n) {};
		}
	\end{scope}
	
	\begin{scope}[every node/.style={draw,fill=black,circle,inner sep=1.25pt}]
		\foreach \n in {2,...,6} {
			\node at (-1,\n) {};
		}
	\end{scope}

	\node at (-1,1) {$\bardot$};
	
	\node[regular polygon, fill=white, draw, regular polygon sides=5, 
 minimum width=0pt, inner sep = 0.5ex,] at (-2,2) {};
	\node at (-2,2) {\large $\ast$};
	
	\node at (-3,3) {$\phiLDRQ$};
	
	\node[draw, circle, fill=black, text=white, inner sep=1.25pt] at (-3,4) {$3$};
	\node[regular polygon, fill=white, draw, regular polygon sides=5, 
 minimum width=0pt, inner sep = 0.5ex,] at (-4,4) {};
	\node at (-4,4) {\large $\ast$};
	
	\node[draw, circle, fill=black, text=white, inner sep=1.25pt] at (-3,5) {$3$};
	
	\node at (-5,5) {$\phiLDRQ$};
	
	\node[draw, circle, fill=black, text=white, inner sep=1.25pt] at (-3,6) {$3$};
	\node[draw, circle, fill=black, text=white, inner sep=1.25pt] at (-5,6) {$3$};
	\node[regular polygon, fill=white, draw, regular polygon sides=5, 
 minimum width=0pt, inner sep = 0.5ex,] at (-6,6) {};
	\node at (-6,6) {\large $\ast$};
	
%
%
%

	
	\draw[fill=white] (-11,-5.25) rectangle (-1,-0.75);
	\draw (-10,-1.6) -- (-3,-1.6);
	
	\node at (-6.4,-1.25) {\large Key:};
	
	\node at (-9,-2.5) {$\whitesquare = \underline{\Z}$};
	\node at (-8.875,-3.5) {$\whitesquaredual = \underline{\Z}^*$};
	\node at (-8.875,-4.5) {$\whitecirc = \langle \Z \rangle$};
	
	\node at (-6.4,-2.5) {$\filltrap = \underline{mg}$};
	\node at (-6.3,-3.5) {$\filltrapdual = \underline{mg}^*$};
	
	\node at (-3,-2.5) {$\fillpent = \phi_{LDR}^* \underline{\F_2}$};
	\node at (-2.9,-3.5) {$\fillpentdual = \phi_{LDR}^* \underline{\F_2}^*$};
	\node at (-3.7, -4.5) {$\bardot = \underline{E}$};
	
	\node at (-6,-4.5) {$\begin{tikzpicture} \node[draw,circle,fill=black,inner sep=0.5pt, text=white] {n}; \end{tikzpicture} = \langle \mathbb{F}_2 \rangle^{\oplus n}$};

	
	\draw[fill=white] (11,5.75) rectangle (2,1.75);
	\draw (10,4.9) -- (3,4.9);
	
	\node at (6.5,5.25) {\large Key:};
	
	\node at (4.5,4) {$\whitepent = \phi_{LDR}^* \underline{\Z}$};
	\node at (4.6,3) {$\whitepentdual = \phi_{LDR}^* \underline{\Z^*}$};
	
	\node at (8.5,4) {$\phiLDRf = \phi_{LDR}^* \underline{f}$};
	\node at (8.6,3) {$\phiLDRQ = \phi_{LDR}^* \underline{Q}$};
	\end{tikzpicture}
        
    \end{center}
    \end{figure}

\begin{figure}
\caption{$\underline{\pi}^\K_{x + y\overline\rho}(H\underline{J})$}
\label{fig:piK4HJ}
\begin{center}
	\begin{tikzpicture}[scale=0.625,font=\tiny]
	
	\draw[color=purple!50, thick, fill=purple!25, rounded corners] (2.5,-1.5) -- (0.5,0.5) -- (0.5,-6.25) -- (2.5,-6.25) -- cycle;
	
	\draw[gray] (-12,-6) grid (12,6); 
	\draw[->] (0,-6) to (0,6.25); 
	\draw[->] (-12,0) to (12,0);
	
	\node at (11.5,0.5) [draw] {$x \cdot 1$};
	\node at (0.5,5.5) [draw, rotate=90] {$y \cdot \overline{\rho}$};

            	           
    \node[draw, circle, fill=white, inner sep=0.5pt] at (0,0) {\normalsize $\underline{J}$};
            	           
    \begin{scope}[every node/.style={draw,fill=white,circle,inner sep=1.25pt}]
		\foreach \n in {1,...,6} {
			\node at (0,-\n) {$2$};
		}
	\end{scope}
	
	\foreach \n in {1,3,5}{
	\node at (\n,-\n) {$\phiLDRf$};
	}
	
	\foreach \n in {2,4,6}{
	\node at (\n,-\n) {$\whitepent$};
	}
	
	\foreach \k in {2,4}{
		\foreach \n in {\k,...,4,5}{
			\node[draw, circle, fill=black, text=white, inner sep=1.25pt] at (\k,-\n-1) {$3$};
		}
	}
		
	                      
	\begin{scope}[every node/.style={draw,fill=white,circle,inner sep=1.25pt}]
		\foreach \n in {1,...,6} {
			\node at (0,\n) {$2$};
		}
	\end{scope}
	
	\begin{scope}[every node/.style={draw,fill=white,circle,inner sep=1.25pt}]
		\foreach \n in {2,...,6} {
			\node at (-1,\n) {$3$};
		}
	\end{scope}

	\node at (-1,1) {$\phiLDRf$};
	
	\node[regular polygon, fill=white, draw, regular polygon sides=5, 
 minimum width=0pt, inner sep = 0.5ex,] at (-2,2) {};
 
	\node at (-2,2) {\large $\ast$};
 
	\node at (-3,3) {$\phiLDRQ$};
	
	\node[draw, circle, fill=black, text=white, inner sep=1.25pt] at (-3,4) {$3$};
	
	\node[regular polygon, fill=white, draw, regular polygon sides=5, 
 minimum width=0pt, inner sep = 0.5ex,] at (-4,4) {};
 
	\node at (-4,4) {\large $\ast$};
	
	\node[draw, circle, fill=black, text=white, inner sep=1.25pt] at (-3,5) {$3$};
	
	\node at (-5,5) {$\phiLDRQ$};
	
	\node[draw, circle, fill=black, text=white, inner sep=1.25pt] at (-3,6) {$3$};
	\node[draw, circle, fill=black, text=white, inner sep=1.25pt] at (-5,6) {$3$};
	\node[regular polygon, fill=white, draw, regular polygon sides=5, 
 minimum width=0pt, inner sep = 0.5ex,] at (-6,6) {};
	\node at (-6,6) {\large $\ast$};
	
%
%
%

	
	\draw[fill=white] (-11,-5.25) rectangle (-1,-0.75);
	\draw (-10,-1.6) -- (-3,-1.6);
	
	\node at (-6.4,-1.25) {\large Key:};
	
	\node at (-9,-2.5) {$\whitesquare = \underline{\Z}$};
	\node at (-8.875,-3.5) {$\whitesquaredual = \underline{\Z}^*$};
	\node at (-8.875,-4.5) {$\whitecirc = \langle \Z \rangle$};
	
	\node at (-6.4,-2.5) {$\filltrap = \underline{mg}$};
	\node at (-6.3,-3.5) {$\filltrapdual = \underline{mg}^*$};
	
	\node at (-3,-2.5) {$\fillpent = \phi_{LDR}^* \underline{\F_2}$};
	\node at (-2.9,-3.5) {$\fillpentdual = \phi_{LDR}^* \underline{\F_2}^*$};
	\node at (-3.7, -4.5) {$\bardot = \underline{E}$};
	
	\node at (-6,-4.5) {$\begin{tikzpicture} \node[draw,circle,fill=black,inner sep=0.5pt, text=white] {n}; \end{tikzpicture} = \langle \mathbb{F}_2 \rangle^{\oplus n}$};

	
	\draw[fill=white] (11,5.75) rectangle (2,1.75);
	\draw (10,4.9) -- (3,4.9);
	
	\node at (6.5,5.25) {\large Key:};
	
	\node at (4.5,4) {$\whitepent = \phi_{LDR}^* \underline{\Z}$};
	\node at (4.6,3) {$\whitepentdual = \phi_{LDR}^* \underline{\Z^*}$};
	
	\node at (8.5,4) {$\phiLDRf = \phi_{LDR}^* \underline{f}$};
	\node at (8.6,3) {$\phiLDRQ = \phi_{LDR}^* \underline{Q}$};
	\end{tikzpicture}
        
    \end{center}
    \end{figure}
    
\clearpage

\begin{figure}
\caption{$\underline{\pi}^\K_{x + y\overline\rho}(H\underline{A})$}
\label{fig:piK4HA}
\begin{center}
	\begin{tikzpicture}[scale=0.625,font=\tiny]
	
	\draw[color=red!50, thick, fill=red!25, rounded corners] (-2,0.5) -- (-8.15,6.65) -- (-12.5,6.65) -- (-12.5,0.5) -- cycle;

	\draw[color=blue!50, thick, fill=blue!25, rounded corners] (2,-0.5) -- (8,-6.75) -- (12.5,-6.75) -- (12.5,-0.5) -- cycle;
	
	\draw[color=green!50, thick, fill=green!25, rounded corners] (-2.5,1.5) -- (-0.5,-0.5) -- (-0.5,6.25) -- (-2.5,6.25) -- cycle;
	
	\draw[color=purple!50, thick, fill=purple!25, rounded corners] (2.275,-1.5) -- (0.5,0.5) -- (0.5,-6.35) -- (2.275,-6.35) -- cycle;
	
	\draw[gray] (-12,-6) grid (12,6); 
	\draw[->] (0,-6) to (0,6.25); 
	\draw[->] (-12,0) to (12,0);
	
	\node at (11.5,0.5) [draw] {$x \cdot 1$};
	\node at (0.5,5.5) [draw, rotate=90] {$y \cdot \overline{\rho}$};       
	
            	           
    \node[draw, circle, fill=white, inner sep=0.5pt] at (0,0) {\normalsize $\underline{A}$};
            	           
    \begin{scope}[every node/.style={draw,fill=white,circle,inner sep=1.25pt}]
		\node at (0,-1){};
		\node at (0,-2){};
		\node at (0,-3){};
		\node at (0,-4){};
		\node at (0,-5){};
		\node at (0,-6){};
	\end{scope}
	            
	\node at (1,-1) {$\phiLDRf$};
	
	\node[regular polygon, fill=white, draw, regular polygon sides=4, 
 minimum width=0pt, inner sep = 0.5ex,] at (3,-1) {};
	
	
	\node[regular polygon, fill=white, draw, regular polygon sides=5, 
 minimum width=0pt, inner sep = 0.5ex,] at (2,-2) {};
	
	\begin{scope}[every node/.style={draw,fill=black,circle,inner sep=1.25pt}]
		\node at (3,-2) {};
	\end{scope}
	
	\node[trapezium, fill=black, draw, inner sep=2.5pt,scale=1] at (4,-2) {};
	
	\node[regular polygon, fill=white, draw, regular polygon sides=4, 
 minimum width=0pt, inner sep = 0.5ex,] at (6,-2) {};
	

	\node[draw,fill=black,circle,inner sep=1.25pt, text=white] at (2,-3) {$3$}; 
	
	\node[below left={-0.5ex}] at (3,-3) {$\phiLDRf$};
	
	\begin{scope}[every node/.style={draw,fill=black,circle,inner sep=1.25pt}]
		\node[above right={0.35ex}] at (3,-3) {};
	\end{scope}
				
	\node[draw,fill=black,circle,inner sep=1.25pt, text=white] at (4,-3) {$2$};
	
	\node[regular polygon, fill=black, draw, regular polygon sides=5, 
 minimum width=0pt, inner sep = 0.5ex,] at (5,-3) {};
 
 	\node[draw, fill=black, circle, inner sep=1.25pt] at (6,-3) {};
 	
 	\node[trapezium, fill=black, draw, inner sep=2.5pt,scale=1] at (7,-3) {};
 	
 	\node[regular polygon, fill=white, draw, regular polygon sides=4, 
 minimum width=0pt, inner sep = 0.5ex,] at (9,-3) {};
 
	
	\node[draw,fill=black,circle,inner sep=1.25pt, text=white] at (2,-4) {$3$}; 	
	
	\node[draw,fill=black,circle,inner sep=1.25pt, text=white] at (3,-4) {};
	
	\node[below left={0.35ex}, regular polygon, fill=white, draw, regular polygon sides=5, 
 minimum width=0pt, inner sep = 0.5ex,] at (4,-4) {};
 
 	\node[above right={0.35ex}, draw,fill=black,circle,inner sep=1.25pt, text=white] at (4,-4) {$2$}; 
 	
 	\node[draw,fill=black,circle,inner sep=1.25pt, text=white] at (5,-4) {$3$}; 
 	
 	\node[below left={0.35ex}, regular polygon, fill=black, draw, regular polygon sides=5, 
 minimum width=0pt, inner sep = 0.5ex,] at (6,-4) {};
 
 	\node[above right={0.35ex}, draw,fill=black,circle,inner sep=1.25pt, text=white] at (6,-4) {};
 	
 	\node[draw,fill=black,circle,inner sep=1.25pt, text=white] at (7,-4) {$2$};
 	
 	\node[regular polygon, fill=black, draw, regular polygon sides=5, 
 minimum width=0pt, inner sep = 0.5ex,] at (8,-4) {};
	
	\node[draw,fill=black,circle,inner sep=1.25pt, text=white] at (9,-4) {};
	
	\node[trapezium, fill=black, draw, inner sep=2.5pt,scale=1] at (10,-4) {};
	
	\node[regular polygon, fill=white, draw, regular polygon sides=4, 
 minimum width=0pt, inner sep = 0.5ex,] at (12,-4) {};
				
			    			
	\node[draw,fill=black,circle,inner sep=1.25pt, text=white] at (2,-5) {$3$}; 
	
	\node[draw,fill=black,circle,inner sep=1.25pt, text=white] at (3,-5) {};
	
	\node[draw,fill=black,circle,inner sep=1.25pt, text=white] at (4,-5) {$5$};  	
	
	\node[below left={-0.5ex}] at (5,-5) {$\phiLDRf$};
	
	\begin{scope}[every node/.style={draw,fill=black,circle,inner sep=1.25pt}]
		\node[above right={0.35ex}, text=white] at (5,-5) {$3$};
	\end{scope}
	
	\node[draw,fill=black,circle,inner sep=1.25pt, text=white] at (6,-5) {$4$};
	
	\node[below left={0.35ex}, regular polygon, fill=black, draw, regular polygon sides=5, 
 minimum width=0pt, inner sep = 0.5ex,] at (7,-5) {};
 
 	\node[above right={0.35ex}, draw,fill=black,circle,inner sep=1.25pt, text=white] at (7,-5) {$2$}; 
 	
 	\node[draw,fill=black,circle,inner sep=1.25pt, text=white] at (8,-5) {$3$};
 	
 	\node[below left={0.35ex}, regular polygon, fill=black, draw, regular polygon sides=5, 
 minimum width=0pt, inner sep = 0.5ex,] at (9,-5) {};
 
 	\node[above right={0.35ex}, draw,fill=black,circle,inner sep=1.25pt, text=white] at (9,-5) {};
 	
 	\node[draw,fill=black,circle,inner sep=1.25pt, text=white] at (10,-5) {$2$};
 
 	\node[regular polygon, fill=black, draw, regular polygon sides=5, 
 minimum width=0pt, inner sep = 0.5ex,] at (11,-5) {};
 
 \node[draw,fill=black,circle,inner sep=1.25pt, text=white] at (12,-5) {};
			
			    			
	\node[draw,fill=black,circle,inner sep=1.25pt, text=white] at (2,-6) {$3$}; 
	
	\node[draw,fill=black,circle,inner sep=1.25pt, text=white] at (3,-6) {};
	
	\node[draw,fill=black,circle,inner sep=1.25pt, text=white] at (4,-6) {$5$};
	
	\node[draw,fill=black,circle,inner sep=1.25pt, text=white] at (5,-6) {$3$};
	
	\node[above right={0.35ex}, draw,fill=black,circle,inner sep=1.25pt, text=white] at (6,-6) {$4$};   
	
	\node[below left={0.35ex}, regular polygon, fill=white, draw, regular polygon sides=5, 
 minimum width=0pt, inner sep = 0.5ex,] at (6,-6) {};
	
	\node[draw,fill=black,circle,inner sep=1.25pt, text=white] at (7,-6) {$5$};
	
	\node[below left={0.35ex}, regular polygon, fill=black, draw, regular polygon sides=5, 
 minimum width=0pt, inner sep = 0.5ex,] at (8,-6) {};
 
 	\node[above right={0.35ex}, draw,fill=black,circle,inner sep=1.25pt, text=white] at (8,-6) {$3$};
 	
 	\node[draw,fill=black,circle,inner sep=1.25pt, text=white] at (9,-6) {$4$};
 	
 	\node[below left={0.35ex}, regular polygon, fill=black, draw, regular polygon sides=5, 
 minimum width=0pt, inner sep = 0.5ex,] at (10,-6) {};
 
 	\node[above right={0.35ex}, draw,fill=black,circle,inner sep=1.25pt, text=white] at (10,-6) {$2$};
 
 	\node[draw,fill=black,circle,inner sep=1.25pt, text=white] at (11,-6) {$3$};
 	
 	\node[below left={0.35ex}, regular polygon, fill=black, draw, regular polygon sides=5, 
 minimum width=0pt, inner sep = 0.5ex,] at (12,-6) {};
 
 	\node[above right={0.35ex}, draw,fill=black,circle,inner sep=1.25pt, text=white] at (12,-6) {};
	

	\foreach \n in {1,...,6}{ 	                      
		\node[draw, fill=white, circle, inner sep=1.25pt] at (0,\n) {};
		}
		
	
	\node at (-1,1) {$\bardot$};
	
	\node[regular polygon, fill=white, draw, regular polygon sides=4, 
 minimum width=0pt, inner sep = 0.5ex,] at (-3, 1) {};
	\node at (-3,1) {\large $\ast$};	
	
	
	\node[draw, fill=black, circle, inner sep=1.25pt] at (-1,2) {};
	
	\node[regular polygon, fill=white, draw, regular polygon sides=5, 
 minimum width=0pt, inner sep = 0.5ex,] at (-2,2) {};
	\node at (-2,2) {\large $\ast$};
	
	\node[draw, fill=black, circle, inner sep = 1.25pt] at (-4,2) {};
	
	\node[trapezium, fill=black, draw, inner sep=2.5pt,scale=1] at (-5,2) {};
	\node[text=white] at (-5,2) {\large $\ast$};
	
	\node[regular polygon, fill=white, draw, regular polygon sides=4, 
 minimum width=0pt, inner sep = 0.5ex,] at (-6,2) {};
	\node at (-6,2) {\large $\ast$};
	
	
	\node[draw, fill=black, circle, inner sep =1.25pt] at (-1,3) {};
	
	\node at (-3,3) {$\phiLDRQ$}; 
	
	\node[draw, fill=black, circle, inner sep=1.25pt] at (-4,3) {};	
	
	\node[draw, fill=black, circle, inner sep=1.25pt, text=white] at (-5,3) {$2$};
	
	\node[regular polygon, fill=black, draw, regular polygon sides=5, 
 minimum width=0pt, inner sep = 0.5ex,] at (-6,3) {};
	\node[text=white] at (-6,3) {\large $\ast$};

	\node[draw, fill=black, circle, inner sep=1.25pt, text=white] at (-7,3) {};	
	
	\node[draw, trapezium, fill=black, inner sep=2.5pt,scale=1] at (-8,3) {};
	\node[text=white] at (-8,3) {\large $\ast$};
	
	\node[regular polygon, fill=white, draw, regular polygon sides=4, 
 minimum width=0pt, inner sep = 0.5ex,] at (-9,3) {};
	\node at (-9,3) {\large $\ast$};
	

	\node[draw, fill=black, circle, inner sep=1.25pt] at (-1,4) {};
			
	\node[draw, fill=black, circle, inner sep=1.25pt, text=white] at (-3,4) {$3$};			
			
	\begin{scope}[xshift=-0.2cm,yshift=-0.2cm]
	\node[regular polygon, fill=white, draw, regular polygon sides=5, 
 minimum width=0pt, inner sep = 0.5ex] at (-4,4) {};
	\node[text=black] at (-4,4) {\large $\ast$};
	\end{scope}
	
	\node[above right={0.35ex}, draw, fill=black, circle, inner sep =1.25pt, text=white] at (-4,4) {};
	
	\node[draw, fill=black, circle, inner sep =1.25pt, text=white] at (-5,4) {$2$};
	
	\node[draw, fill=black, circle, inner sep =1.25pt, text=white] at (-6,4) {$3$};
	
	\node[above right={0.35ex}, draw, fill=black, circle, inner sep =1.25pt] at (-7,4) {};
	
	\begin{scope}[xshift=-0.2cm,yshift=-0.2cm]
	\node[regular polygon, fill=black, draw, regular polygon sides=5, 
 minimum width=0pt, inner sep = 0.5ex] at (-7,4) {};
	\node[text=white] at (-7,4) {\large $\ast$};
	\end{scope}
	
	\node[draw, fill=black, circle, inner sep =1.25pt, text=white] at (-8,4) {$2$};
	
	\node[regular polygon, fill=black, draw, regular polygon sides=5, 
 minimum width=0pt, inner sep = 0.5ex,] at (-9,4) {};
	\node[text=white] at (-9,4) {\large $\ast$};

	\node[draw, fill=black, circle, inner sep=1.25pt, text=white] at (-10,4) {};
	
	\node[draw, trapezium, fill=black, inner sep=2.5pt,scale=1] at (-11,4) {};
	\node[text=white] at (-11,4) {\large $\ast$};
	
	\node[regular polygon, fill=white, draw, regular polygon sides=4, 
 minimum width=0pt, inner sep = 0.5ex,] at (-12,4) {};
	\node at (-12,4) {\large $\ast$};
	
			
	\node[draw, fill=black, circle, inner sep =1.25pt] at (-1,5) {};
	
	\node[draw, fill=black, circle, inner sep=1.25pt, text=white] at (-3,5) {$3$};
	
	\node[draw, fill=black, circle, inner sep =1.25pt] at (-4,5) {};
	
	\node[above right, draw, fill=black, circle, inner sep=1.25pt, text=white] at (-5,5) {$2$};
	
	\begin{scope}[xshift=-0.2cm,yshift=-0.2cm]
	\node at (-5,5) {$\phiLDRQ$};
	\end{scope}
	
	\node[draw, fill=black, circle, inner sep =1.25pt, text=white] at (-6,5) {$3$};
	
	\node[draw, fill=black, circle, inner sep =1.25pt, text=white] at (-7,5) {$4$};
	
	\begin{scope}[xshift=-0.2cm,yshift=-0.2cm]
	\node[regular polygon, fill=black, draw, regular polygon sides=5, 
 minimum width=0pt, inner sep = 0.5ex] at (-8,5) {};
	\node[text=white] at (-8,5) {\large $\ast$};
	\end{scope}
	
	\node[above right={0.35ex}, draw, fill=black, circle, inner sep =1.25pt, text=white] at (-8,5) {$2$};
	
	\node[draw, fill=black, circle, inner sep =1.25pt, text=white] at (-9,5) {$3$};
	
	\begin{scope}[xshift=-0.2cm,yshift=-0.2cm]
	\node[regular polygon, fill=black, draw, regular polygon sides=5, 
 minimum width=0pt, inner sep = 0.5ex] at (-10,5) {};
	\node[text=white] at (-10,5) {\large $\ast$};
	\end{scope}
	
	\node[above right={0.35ex}, draw, fill=black, circle, inner sep =1.25pt, text=white] at (-10,5) {};
	
	\node[draw, fill=black, circle, inner sep =1.25pt, text=white] at (-11,5) {$2$};
	
	\node[regular polygon, fill=black, draw, regular polygon sides=5, 
 minimum width=0pt, inner sep = 0.5ex,] at (-12,5) {};
	\node[text=white] at (-12,5) {\large $\ast$};
	
	
	\node[draw, fill=black, circle, inner sep =1.25pt, text=white] at (-1,6) {};
	
	\node[draw, fill=black, circle, inner sep=1.25pt, text=white] at (-3,6) {$3$};
	
	\node[draw, fill=black, circle, inner sep =1.25pt, text=white] at (-4,6) {};
	
	\node[draw, fill=black, circle, inner sep =1.25pt, text=white] at (-5,6) {$5$};
	
	\begin{scope}[xshift=-0.2cm,yshift=-0.2cm]
	\node[regular polygon, fill=white, draw, regular polygon sides=5, 
 minimum width=0pt, inner sep = 0.5ex] at (-6,6) {};
	\node[text=black] at (-6,6) {\large $\ast$};
	\end{scope}
	
	\node[above right={0.35ex}, draw, fill=black, circle, inner sep =1.25pt, text=white] at (-6,6) {$3$};
	
	\node[draw, fill=black, circle, inner sep =1.25pt, text=white] at (-7,6) {$4$};
	
	\node[draw, fill=black, circle, inner sep =1.25pt, text=white] at (-8,6) {$5$};
	
	\begin{scope}[xshift=-0.2cm,yshift=-0.2cm]
	\node[regular polygon, fill=black, draw, regular polygon sides=5, 
 minimum width=0pt, inner sep = 0.5ex] at (-9,6) {};
	\node[text=white] at (-9,6) {\large $\ast$};
	\end{scope}
	
	\node[above right={0.35ex}, draw, fill=black, circle, inner sep =1.25pt, text=white] at (-9,6) {$3$};
	
	\node[draw, fill=black, circle, inner sep =1.25pt, text=white] at (-10,6) {$4$};
	
	\begin{scope}[xshift=-0.2cm,yshift=-0.2cm]
	\node[regular polygon, fill=black, draw, regular polygon sides=5, 
 minimum width=0pt, inner sep = 0.5ex] at (-11,6) {};
	\node[text=white] at (-11,6) {\large $\ast$};
	\end{scope}
	
	\node[above right={0.35ex}, draw, fill=black, circle, inner sep =1.25pt, text=white] at (-11,6) {$2$};
		
	\node[draw, fill=black, circle, inner sep =1.25pt, text=white] at (-12,6) {$3$};
	
%
%
%

	
	\draw[fill=white] (-11,-5.25) rectangle (-1,-0.75);
	\draw (-10,-1.6) -- (-3,-1.6);
	
	\node at (-6.4,-1.25) {\large Key:};
	
	\node at (-9,-2.5) {$\whitesquare = \underline{\Z}$};
	\node at (-8.875,-3.5) {$\whitesquaredual = \underline{\Z}^*$};
	\node at (-8.875,-4.5) {$\whitecirc = \langle \Z \rangle$};
	
	\node at (-6.4,-2.5) {$\filltrap = \underline{mg}$};
	\node at (-6.3,-3.5) {$\filltrapdual = \underline{mg}^*$};
	
	\node at (-3,-2.5) {$\fillpent = \phi_{LDR}^* \underline{\F_2}$};
	\node at (-2.9,-3.5) {$\fillpentdual = \phi_{LDR}^* \underline{\F_2}^*$};
	\node at (-3.7, -4.5) {$\bardot = \underline{E}$};
	
	\node at (-6,-4.5) {$\begin{tikzpicture} \node[draw,circle,fill=black,inner sep=0.5pt, text=white] {n}; \end{tikzpicture} = \langle \mathbb{F}_2 \rangle^{\oplus n}$};

	
	\draw[fill=white] (11,5.75) rectangle (2,1.75);
	\draw (10,4.9) -- (3,4.9);
	
	\node at (6.5,5.25) {\large Key:};
	
	\node at (4.5,4) {$\whitepent = \phi_{LDR}^* \underline{\Z}$};
	\node at (4.6,3) {$\whitepentdual = \phi_{LDR}^* \underline{\Z}^*$};
	
	\node at (8.5,4) {$\phiLDRf = \phi_{LDR}^* \underline{f}$};
	\node at (8.6,3) {$\phiLDRQ = \phi_{LDR}^* \underline{Q}$};
	\end{tikzpicture}
	
    \end{center}
    \end{figure}
    
\begin{figure}
\caption{$\underline{\pi}^\K_{x + y\overline\rho}(H\underline{\Z})$}
\label{fig:piK4HZ}
\begin{center}
	\begin{tikzpicture}[scale=0.625,font=\tiny]
	
	\draw[color=red!50, thick, fill=red!25, rounded corners] (-2,0.5) -- (-8.15,6.65) -- (-12.5,6.65) -- (-12.5,0.5) -- cycle;

	\draw[color=blue!50, thick, fill=blue!25, rounded corners] (2,-0.5) -- (8,-6.75) -- (12.5,-6.75) -- (12.5,-0.5) -- cycle;
	
	\draw[gray] (-12,-6) grid (12,6); 
	\draw[->] (0,-6) to (0,6.25); 
	\draw[->] (-12,0) to (12,0);
	
	\node at (11.5,0.5) [draw] {$x \cdot 1$};
	\node at (0.5,5.5) [draw, rotate=90] {$y \cdot \overline{\rho}$};

            	           
    \node[draw, circle, fill=white, inner sep=0.5pt] at (0,0) {\normalsize $\underline{\Z}$};
            	           
    \begin{scope}[every node/.style={draw,fill=black,circle,inner sep=1.25pt}]
		\node at (0,-1){};
		\node at (0,-2){};
		\node at (0,-3){};
		\node at (0,-4){};
		\node at (0,-5){};
		\node at (0,-6){};
	\end{scope}
	            	         
	\node[draw, fill=black, trapezium, minimum width=0pt, inner sep=2.5] at (1,-1) {};
	
	\node[regular polygon, fill=white, draw, regular polygon sides=4, 
 minimum width=0pt, inner sep = 0.5ex,] at (3,-1) {};
	
	
	\node[draw, fill=black, circle, text=white, inner sep=1.25] at (1,-2) {$2$};
	
	\node[regular polygon, fill=black, draw, regular polygon sides=5, 
 minimum width=0pt, inner sep = 0.5ex, below left={0.35ex}] at (2,-2) {};
	
	\node[draw, fill=black, circle, inner sep=1.25pt] at (3,-2) {};
	
	\node[trapezium, fill=black, draw, minimum width=0pt, inner sep=2.5pt,scale=1] at (4,-2) {};
	
	\node[regular polygon, fill=white, draw, regular polygon sides=4, 
 minimum width=0pt, inner sep = 0.5ex,] at (6,-2) {};
	

	\node[draw,fill=black,circle,inner sep=1.25pt, text=white] at (1,-3) {$2$};
	
	\node[draw,fill=black,circle,inner sep=1.25pt, text=white] at (2,-3) {$3$}; 
	
	\node[regular polygon, fill=black, draw, regular polygon sides=5, 
 minimum width=0pt, inner sep = 0.5ex, below left={0.35ex}] at (3,-3) {};
	
	\node[draw, fill=black, inner sep=1.25pt, circle, above right={0.35ex}] at (3,-3) {};
	
	\node[draw,fill=black,circle,inner sep=1.25pt, text=white] at (4,-3) {$2$};
	
	\node[regular polygon, fill=black, draw, regular polygon sides=5, 
 minimum width=0pt, inner sep = 0.5ex,] at (5,-3) {};
 
 	\node[draw, fill=black, circle, inner sep=1.25pt] at (6,-3) {};
 	
 	\node[trapezium, fill=black, draw, inner sep=2.5pt,scale=1] at (7,-3) {};
 	
 	\node[regular polygon, fill=white, draw, regular polygon sides=4, 
 minimum width=0pt, inner sep = 0.5ex,] at (9,-3) {};
 
	
	\node[draw, fill=black, inner sep=1.25pt, circle, text=white] at (1,-4) {$2$};

	\node[draw, fill=black, inner sep=1.25pt, circle, text=white] at (2,-4) {$3$};
	
	\node[draw, fill=black, inner sep=1.25pt, circle, text=white] at (3,-4) {$4$};
	
	\node[draw, fill=black, inner sep=1.25pt, circle, text=white, above right={0.35ex}] at (4,-4) {$2$};
	
	\node[regular polygon, fill=black, draw, regular polygon sides=5, 
 minimum width=0pt, inner sep = 0.5ex, below left={0.35ex}] at (4,-4) {};

	\node[draw, fill=black, inner sep=1.25pt, circle, text=white] at (5,-4) {$3$};
 	
 	\node[below left={0.35ex}, regular polygon, fill=black, draw, regular polygon sides=5, 
 minimum width=0pt, inner sep = 0.5ex,] at (6,-4) {};
 
 	\node[above right={0.35ex}, draw,fill=black,circle,inner sep=1.25pt, text=white] at (6,-4) {};
 	
 	\node[draw,fill=black,circle,inner sep=1.25pt, text=white] at (7,-4) {$2$};
 	
 	\node[regular polygon, fill=black, draw, regular polygon sides=5, 
 minimum width=0pt, inner sep = 0.5ex,] at (8,-4) {};
	
	\node[draw,fill=black,circle,inner sep=1.25pt, text=white] at (9,-4) {};
	
	\node[trapezium, fill=black, draw, inner sep=2.5pt,scale=1] at (10,-4) {};
	
	\node[regular polygon, fill=white, draw, regular polygon sides=4, 
 minimum width=0pt, inner sep = 0.5ex,] at (12,-4) {};
				
			    			
	\node[draw, fill=black, inner sep=1.25pt, circle, text=white] at (1,-5) {$2$};
	
	\node[draw, fill=black, inner sep=1.25pt, circle, text=white] at (2,-5) {$3$};
	
	\node[draw, fill=black, inner sep=1.25pt, circle, text=white] at (3,-5) {$4$};
	
	\node[draw, fill=black, inner sep=1.25pt, circle, text=white] at (4,-5) {$5$};
	
	\node[draw, fill=black, inner sep=1.25pt, circle, text=white, above right] at (5,-5) {$3$};
	
	\node[regular polygon, fill=black, draw, regular polygon sides=5, 
 minimum width=0pt, inner sep = 0.5ex, below left={0.35ex}] at (5,-5) {};
	
	\node[draw, fill=black, inner sep=1.25pt, circle, text=white] at (6,-5) {$4$};
	
	\node[below left={0.35ex}, regular polygon, fill=black, draw, regular polygon sides=5, 
 minimum width=0pt, inner sep = 0.5ex,] at (7,-5) {};
 
 	\node[above right={0.35ex}, draw,fill=black,circle,inner sep=1.25pt, text=white] at (7,-5) {$2$}; 
 	
 	\node[draw,fill=black,circle,inner sep=1.25pt, text=white] at (8,-5) {$3$};
 	
 	\node[below left={0.35ex}, regular polygon, fill=black, draw, regular polygon sides=5, 
 minimum width=0pt, inner sep = 0.5ex,] at (9,-5) {};
 
 	\node[above right={0.35ex}, draw,fill=black,circle,inner sep=1.25pt, text=white] at (9,-5) {};
 	
 	\node[draw,fill=black,circle,inner sep=1.25pt, text=white] at (10,-5) {$2$};
 
 	\node[regular polygon, fill=black, draw, regular polygon sides=5, 
 minimum width=0pt, inner sep = 0.5ex,] at (11,-5) {};
 
 \node[draw,fill=black,circle,inner sep=1.25pt, text=white] at (12,-5) {};
			
			    			
	\node[draw, fill=black, inner sep=1.25pt, circle, text=white] at (1,-6) {$2$};
	
	\node[draw, fill=black, inner sep=1.25pt, circle, text=white] at (2,-6) {$3$};
	
	\node[draw, fill=black, inner sep=1.25pt, circle, text=white] at (3,-6) {$4$};
	
	\node[draw, fill=black, inner sep=1.25pt, circle, text=white] at (4,-6) {$5$};
	
	\node[draw, fill=black, inner sep=1.25pt, circle, text=white] at (5,-6) {$6$};	
	
	\node[draw, fill=black, minimum width=0pt, inner sep=1.25pt, circle, text=white, above right={0.35ex}] at (6,-6) {$4$};
	
	\node[regular polygon, fill=black, draw, regular polygon sides=5, 
 minimum width=0pt, inner sep = 0.5ex, below left={0.35ex}] at (6,-6) {};
	
	\node[draw, fill=black, inner sep=1.25pt, circle, text=white] at (7,-6) {$5$};
	
	\node[below left={0.35ex}, regular polygon, fill=black, draw, regular polygon sides=5, 
 minimum width=0pt, inner sep = 0.5ex,] at (8,-6) {};
 
 	\node[above right={0.35ex}, draw,fill=black,circle,inner sep=1.25pt, text=white] at (8,-6) {$3$};
 	
 	\node[draw,fill=black,circle,inner sep=1.25pt, text=white] at (9,-6) {$4$};
 	
 	\node[below left={0.35ex}, regular polygon, fill=black, draw, regular polygon sides=5, 
 minimum width=0pt, inner sep = 0.5ex,] at (10,-6) {};
 
 	\node[above right={0.35ex}, draw,fill=black,circle,inner sep=1.25pt, text=white] at (10,-6) {$2$};
 
 	\node[draw,fill=black,circle,inner sep=1.25pt, text=white] at (11,-6) {$3$};
 	
 	\node[below left={0.35ex}, regular polygon, fill=black, draw, regular polygon sides=5, 
 minimum width=0pt, inner sep = 0.5ex,] at (12,-6) {};
 
 	\node[above right={0.35ex}, draw,fill=black,circle,inner sep=1.25pt, text=white] at (12,-6) {};
	
		
	
	\node[regular polygon, fill=white, draw, regular polygon sides=4, 
 minimum width=0pt, inner sep = 0.5ex,] at (-3,1) {};
	\node at (-3,1) {\large $\ast$};	
	
	
	\node[draw, fill=black, circle, inner sep = 1.25pt] at (-4,2) {};
	
	\node[trapezium, fill=black, draw, inner sep=2.5pt,scale=1] at (-5,2) {};
	\node[text=white] at (-5,2) {\large $\ast$};
	
	\node[regular polygon, fill=white, draw, regular polygon sides=4, 
 minimum width=0pt, inner sep = 0.5ex,] at (-6,2) {};
	\node at (-6,2) {\large $\ast$};
	
	\node[draw, fill=black, circle, inner sep=1.25pt] at (-4,3) {};	
	
	\node[draw, fill=black, circle, inner sep=1.25pt, text=white] at (-5,3) {$2$};
	
	\node[regular polygon, fill=black, draw, regular polygon sides=5, 
 minimum width=0pt, inner sep = 0.5ex,] at (-6,3) {};
	\node[text=white] at (-6,3) {\large $\ast$};
	
	\node[draw, fill=black, inner sep=1.25pt, circle] at (-7,3) {};
	
	\node[draw, trapezium, fill=black, inner sep=2.5pt,scale=1] at (-8,3) {};
	\node[text=white] at (-8,3) {\large $\ast$};
	
	\node[regular polygon, fill=white, draw, regular polygon sides=4, 
 minimum width=0pt, inner sep = 0.5ex,] at (-9,3) {};
	\node at (-9,3) {\large $\ast$};
	
	
	\node[draw, fill=black, circle, inner sep =1.25pt, text=white] at (-4,4) {};
	
	\node[draw, fill=black, circle, inner sep =1.25pt, text=white] at (-5,4) {$2$};
	
	\node[draw, fill=black, circle, inner sep =1.25pt, text=white] at (-6,4) {$3$};
	
	\node[above right={0.35ex}, draw, fill=black, circle, inner sep =1.25pt] at (-7,4) {};
	
	\begin{scope}[xshift=-0.2cm,yshift=-0.2cm]
	\node[regular polygon, fill=black, draw, regular polygon sides=5, 
 minimum width=0pt, inner sep = 0.5ex] at (-7,4) {};
	\node[text=white] at (-7,4) {\large $\ast$};
	\end{scope}
	
	\node[draw, fill=black, circle, inner sep =1.25pt, text=white] at (-8,4) {$2$};
	
	\node[regular polygon, fill=black, draw, regular polygon sides=5, 
 minimum width=0pt, inner sep = 0.5ex,] at (-9,4) {};
	\node[text=white] at (-9,4) {\large $\ast$};
	
	\node[draw, fill=black, inner sep=1.25pt, circle] at (-10,4) {};
	
	\node[draw, trapezium, fill=black, inner sep=2.5pt,scale=1] at (-11,4) {};
	\node[text=white] at (-11,4) {\large $\ast$};
	
	\node[regular polygon, fill=white, draw, regular polygon sides=4, 
 minimum width=0pt, inner sep = 0.5ex,] at (-12,4) {};
	\node at (-12,4) {\large $\ast$};
	
	
	\node[draw, fill=black, circle, inner sep =1.25pt] at (-4,5) {};
	
	\node[fill=black, circle, inner sep=1.25pt, text=white] at (-5,5) {$2$};
	
	\node[draw, fill=black, circle, inner sep =1.25pt, text=white] at (-6,5) {$3$};
	
	\node[draw, fill=black, circle, inner sep =1.25pt, text=white] at (-7,5) {$4$};
	
	\begin{scope}[xshift=-0.2cm,yshift=-0.2cm]
	\node[regular polygon, fill=black, draw, regular polygon sides=5, 
 minimum width=0pt, inner sep = 0.5ex] at (-8,5) {};
	\node[text=white] at (-8,5) {\large $\ast$};
	\end{scope}
	
	\node[above right={0.35ex}, draw, fill=black, circle, inner sep =1.25pt, text=white] at (-8,5) {$2$};
	
	\node[draw, fill=black, circle, inner sep =1.25pt, text=white] at (-9,5) {$3$};
	
	\begin{scope}[xshift=-0.2cm,yshift=-0.2cm]
	\node[regular polygon, fill=black, draw, regular polygon sides=5, 
 minimum width=0pt, inner sep = 0.5ex] at (-10,5) {};
	\node[text=white] at (-10,5) {\large $\ast$};
	\end{scope}
	
	\node[above right={0.35ex}, draw, fill=black, circle, inner sep =1.25pt, text=white] at (-10,5) {};
	
	\node[draw, fill=black, circle, inner sep =1.25pt, text=white] at (-11,5) {$2$};
	
	\node[regular polygon, fill=black, draw, regular polygon sides=5, 
 minimum width=0pt, inner sep = 0.5ex,] at (-12,5) {};
	\node[text=white] at (-12,5) {\large $\ast$};
	
	
	\node[draw, fill=black, circle, inner sep =1.25pt, text=white] at (-4,6) {};
	
	\node[draw, fill=black, circle, inner sep =1.25pt, text=white] at (-5,6) {$2$};
	
	\node[draw, fill=black, circle, inner sep =1.25pt, text=white] at (-6,6) {$3$};
	
	\node[draw, fill=black, circle, inner sep =1.25pt, text=white] at (-7,6) {$4$};
	
	\node[draw, fill=black, circle, inner sep =1.25pt, text=white] at (-8,6) {$5$};
	
	\begin{scope}[xshift=-0.2cm,yshift=-0.2cm]
	\node[regular polygon, fill=black, draw, regular polygon sides=5, 
 minimum width=0pt, inner sep = 0.5ex] at (-9,6) {};
	\node[text=white] at (-9,6) {\large $\ast$};
	\end{scope}
	
	\node[above right={0.35ex}, draw, fill=black, circle, inner sep =1.25pt, text=white] at (-9,6) {$3$};
	
	\node[draw, fill=black, circle, inner sep =1.25pt, text=white] at (-10,6) {$4$};
	
	\begin{scope}[xshift=-0.2cm,yshift=-0.2cm]
	\node[regular polygon, fill=black, draw, regular polygon sides=5, 
 minimum width=0pt, inner sep = 0.5ex] at (-11,6) {};
	\node[text=white] at (-11,6) {\large $\ast$};
	\end{scope}
	
	\node[above right={0.35ex}, draw, fill=black, circle, inner sep =1.25pt, text=white] at (-11,6) {$2$};
		
	\node[draw, fill=black, circle, inner sep =1.25pt, text=white] at (-12,6) {$3$};
	
%
%
%

	
	\draw[fill=white] (-11,-5.25) rectangle (-1,-0.75);
	\draw (-10,-1.6) -- (-3,-1.6);
	
	\node at (-6.4,-1.25) {\large Key:};
	
	\node at (-9,-2.5) {$\whitesquare = \underline{\Z}$};
	\node at (-8.875,-3.5) {$\whitesquaredual = \underline{\Z}^*$};
	
	\node at (-6.4,-2.5) {$\filltrap = \underline{mg}$};
	\node at (-6.3,-3.5) {$\filltrapdual = \underline{mg}^*$};
	
	\node at (-3,-2.5) {$\fillpent = \phi_{LDR}^* \underline{\F_2}$};
	\node at (-2.9,-3.5) {$\fillpentdual = \phi_{LDR}^* \underline{\F_2}^*$};
	
	\node at (-6,-4.5) {$\begin{tikzpicture} \node[draw,circle,fill=black,inner sep=0.5pt, text=white] {n}; \end{tikzpicture} = \langle \mathbb{F}_2 \rangle^{\oplus n}$};
	\end{tikzpicture}
	
    \end{center}
    \end{figure}

\clearpage
\newpage
\bibliographystyle{plain}
\bibliography{refs}

\end{document}